\documentclass[english]{smfbook}

\usepackage[english]{babel}

\usepackage{xcolor}
\definecolor{Chocolat}{rgb}{0.36, 0.2, 0.09}
\definecolor{BleuTresFonce}{rgb}{0.215, 0.215, 0.36}

\usepackage[linktoc=page,colorlinks,final,hyperindex,naturalnames]{hyperref}
\hypersetup{citecolor=BleuTresFonce, linkcolor=Chocolat}

\usepackage{graphicx}
\usepackage{epstopdf}
\usepackage[all]{xy}
\usepackage{amsmath,amsfonts,amssymb,MnSymbol}
\usepackage[centering,vscale=0.68,hscale=0.68]{geometry}
\usepackage{latexsym,amscd}
\usepackage{array}
\usepackage{tikz}
\usetikzlibrary{arrows,decorations.markings,shapes.geometric}
\tikzset{>=latex}
\usepackage{url}
\usepackage{dsfont}
\usepackage{microtype}
\usepackage{todonotes}
\makeatletter
\providecommand\@dotsep{5}
\renewcommand{\listoftodos}[1][\@todonotes@todolistname]{%
  \@starttoc{tdo}{#1}}
\makeatother

\newcommand{\on}{\operatorname}

\renewcommand{\O}{\mathcal{O}}

\newcommand{\kk}{\mathbf{k}}
\newcommand{\Z}{{\mathbb Z}}
\newcommand{\Q}{{\mathbb Q}}
\newcommand{\C}{{\mathbb C}}

\newcommand{\KK}{\mathbf{k}}
\newcommand{\0}{\mathbf{0}}

\newcommand{\Pa}{\mathbf{Pa}}
\newcommand{\PaB}{\mathbf{PaB}}
\newcommand{\PaCD}{\mathbf{PaCD}}
\renewcommand{\CD}{\mathbf{CD}}

\renewcommand{\sl}{\mathfrak{sl}}

\renewcommand{\t}{{\mathfrak{t}}}

\renewcommand{\d}{{\mathfrak d}}
\newcommand{\g}{{\mathfrak{g}}}

\newcommand{\h}{{\mathfrak{h}}}
\renewcommand{\l}{{\mathfrak{l}}}

\newcommand{\SG}{{\mathfrak{S}}}
\newcommand{\f}{{\mathfrak{f}}}

\renewcommand{\i}{\on{i}}

\newcommand{\PB}{\on{PB}}

\newcommand{\zz}{\mathbf{z}}

\newcommand{\Assoc}{\mathbf{Assoc}}
\newcommand{\Ass}{\mathbf{Assoc}}
\newcommand{\Ell}{\mathbf{Ell}}

\newcommand{\GT}{\mathbf{GT}}
\newcommand{\GRT}{\mathbf{GRT}}

\newcommand{\M}{{\mathcal M}}

\newcommand{\cM}{{\mathcal M}}

\tikzstyle cross=[preaction={draw=white, -, line width=4pt}, thick]
\tikzstyle normal=[thick]
\tikzstyle chord=[densely dotted, thick]
\tikzstyle zero=[ultra thick, gray]
\tikzstyle zell=[ultra thick, white]
\tikzstyle zerocross=[preaction={draw=white, -, line width=4pt}, ultra thick, gray]
\tikzstyle point=[draw,circle,inner sep=1,fill=black]
\tikzstyle diam=[draw,diamond,inner sep=1,fill=black]
\tikzstyle petitpoint=[draw,circle,inner sep=0.3,fill=black]

\newcommand{\straight}[3][-]{\draw[normal,#1] (#2,-#3) -- (#2,-#3-1);}

\newcommand{\hori}[4][-]{\draw[normal,#1] (#2,-#3)--(#2,-#3-1);\draw[normal,#1] (#2+#4,-#3)--(#2+#4,-#3-1);\draw[chord] (#2,-#3-0.5)--(#2+#4,-#3-0.5);}

\newcommand{\tell}[4][-]{\draw[zell] (0,-#2)--(0,-#2-1);\draw[normal,#1] (#3,-#2)--(#3,-#2-1);\draw[chord] (0,-#2-0.5)--(#3,-#2-0.5); \node[point,label=above:$#4$] at (0,-#2-0.5) {};}

\newcommand{\tik}[1]{\begin{tikzpicture}[baseline=(current bounding box.center)] #1 \end{tikzpicture} }
\newcommand{\tmmathbf}[1]{\ensuremath{\boldsymbol{#1}}}
\newcommand{\tmop}[1]{\ensuremath{\operatorname{#1}}}
\newcommand{\assign}{:=}

\newcommand*{\doublerightarrow}[2]{
\begin{matrix}
\overset{\!#1}{\longrightarrow} \\
\underset{\!#2}{\longrightarrow}
\end{matrix}}

\usepackage[toc,section=subsection]{glossaries}

\newglossary*{operads}{Operads}
\newglossary*{groups}{Groups}
\newglossary*{spaces}{Spaces}
\newglossary*{liealgebras}{Lie and associative algebras}
\newglossary*{torsors}{Torsor sets}
\newglossary*{associators}{Series}

\makenoidxglossaries

\newglossaryentry{PaB}{type=operads, name={$\PaB$}, 
  description={Operad of parenthesized braids} , sort={S}}

\newglossaryentry{PaBe}{type=operads, name={$\PaB_{e\ell\ell}$}, 
  description={$\PaB$-module of elliptic parenthesized braids} , sort={S}}
  
\newglossaryentry{PaB1}{type=operads, name={$\PaB^{1}$}, 
  description={$\PaB$-moperad of parenthesized braids with a frozen strand}, sort={S}}  
  
\newglossaryentry{PaBg}{type=operads, name={$\PaB^{\mathbf{\Gamma}}$}, 
  description={$\PaB$-moperad of cyclotomic parenthesized braids}, sort={S}}
  
\newglossaryentry{PaBeg}{type=operads, name={$\PaB_{e\ell\ell}^{\mathbf{\Gamma}}$}, 
  description={$\PaB$-module of ellipsitomic parenthesized braids}, sort={S}}
  
\newglossaryentry{PaCD}{type=operads, name={$\PaCD(\KK)$}, 
  description={Operad of parenthesized chord diagrams}, sort={S}}

\newglossaryentry{PaCDe}{type=operads, name={$\PaCD_{e\ell\ell}(\KK)$}, 
  description={$\PaCD(\KK)$-module of ellitpic parenthesized chord diagrams}, sort={S}}
  
\newglossaryentry{PaCDeg}{type=operads, name={$\PaCD_{e\ell\ell}^{\mathbf{\Gamma}}(\KK)$}, 
  description={$\PaCD(\KK)$-module of ellipsitomic parenthesized chord diagrams}, sort={S}}  

\newglossaryentry{GT}{type=groups, name={$\GT$}, 
  description={Operadic Grothendieck-Teichm\"uller group}, sort={S}}

\newglossaryentry{kGTe}{type=groups, name={$\widehat{\GT}_{e\ell\ell}(\kk)$}, 
  description={Operadic $\kk$-pro-unipotent elliptic Grothendieck-Teichm\"uller group}, sort={S}}
 
\newglossaryentry{kGTeg}{type=groups, name={$\widehat{\GT}_{e\ell\ell}^{\mathbf{\Gamma}}(\kk)$}, 
  description={Operadic $\kk$-pro-unipotent ellipsitomic Grothendieck-Teichm\"uller group}, sort={S}}
  
\newglossaryentry{GRT}{type=groups, name=$\GRT(\KK)$, 
  description={Operadic graded Grothendieck-Teichm\"uller group}, sort={S}}

\newglossaryentry{GRTe}{type=groups, name=$\GRT_{e\ell\ell}(\KK)$, 
  description={Operadic graded elliptic Grothendieck-Teichm\"uller group}, sort={S}}
  
\newglossaryentry{GRTeg}{type=groups, name=$\GRT_{e\ell\ell}^{\mathbf{\Gamma}}(\KK)$, 
  description={Operadic graded ellipsitomic Grothendieck-Teichm\"uller group}, sort={S}}

\newglossaryentry{bGT}{type=groups, name=$\widehat{\on{GT}}(\kk)$, 
  description={$\kk$-pro-unipotent Grothendieck-Teichm\"uller group}, sort={S}}

\newglossaryentry{bkGTe}{type=groups, name=$\widehat{\on{GT}}_{e\ell\ell}(\kk)$, 
  description={$\kk$-pro-unipotent elliptic Grothendieck-Teichm\"uller group}, sort={S}}
  
\newglossaryentry{bkGTeg}{type=groups, name=$\widehat{\on{GT}}_{e\ell\ell}^{\Gamma}(\kk)$, 
  description={$\kk$-pro-unipotent ellipsitomic Grothendieck-Teichm\"uller group}, sort={S}}
  
\newglossaryentry{bGRT}{type=groups, name=$\on{GRT}(\KK)$, 
  description={Graded Grothendieck-Teichm\"uler group}, sort={S}}

\newglossaryentry{bGRTe}{type=groups, name=$\on{GRT}_{e\ell\ell}(\KK)$, 
  description={Graded elliptic Grothendieck-Teichm\"uller group}, sort={S}}

\newglossaryentry{bGRTeg}{type=groups, name=$\on{GRT}_{e\ell\ell}^{\Gamma}(\KK)$, 
  description={Graded ellipsitomic Grothendieck-Teichm\"uller group}, sort={S}}
  
\newglossaryentry{bGRTfg}{type=groups, name=$\on{GRT}^f_g(\KK)$, 
  description={Graded genus $g$ Grothendieck-Teichm\"uller group}, sort={S}}

\newglossaryentry{PBn}{type=groups, name=$\on{PB}_n$, 
  description={Pure braid group on the complex plane}, sort={S}}
  
\newglossaryentry{PB1n}{type=groups, name=$\overline{\on{PB}}_{1,n}$, 
  description={Reduced pure braid group on the torus}, sort={S}}
  
\newglossaryentry{B1nG}{type=groups, name=$\on{B}_{1,n}^\Gamma$, 
  description={$\Gamma$-decorated braid group on the torus}, sort={S}}

\newglossaryentry{Gng}{type=groups, name=$\mathbf{G}_n^\Gamma$, 
  description={Structure group of the principal bundle over $\M_{1,n}^\Gamma$}, sort={S}}  

\newglossaryentry{bGng}{type=groups, name=$\bar{\mathbf{G}}_n^\Gamma$, 
  description={Structure group of the principal bundle over $\bar{\M}_{1,n}^\Gamma$}, sort={S}}  

\newglossaryentry{SL2g}{type=groups, name=$\on{SL}^\Gamma_2(\Z)$, 
  description={$\Gamma$-level principal congruence subgroup of $\on{SL}_2(\Z)$}, sort={S}}  

\newglossaryentry{Ass}{type=torsors, name=$\Ass(\KK)$, 
  description={Operadic $\KK$-associators}, sort={S}}

\newglossaryentry{Asse}{type=torsors, name=$\Ell(\KK)$, 
  description={Operadic elliptic $\KK$-associators}, sort={S}}
  
\newglossaryentry{Asseg}{type=torsors, name=$\Ell^{\mathbf{\Gamma}}(\KK)$, 
  description={Operadic ellipsitomic $\KK$-associators}, sort={S}}
    
\newglossaryentry{bAss}{type=torsors, name=$\on{Ass}(\KK)$, 
  description={$\KK$-associators}, sort={S}}

\newglossaryentry{bAsse}{type=torsors, name=$\on{Ell}(\KK)$, 
  description={Elliptic $\KK$-associators}, sort={S}}
  
\newglossaryentry{bAsseg}{type=torsors, name=$\on{Ell}^{\Gamma}(\KK)$, 
  description={Ellipsitomic $\KK$-associators}, sort={S}}
  
\newglossaryentry{KZAss}{type=associators, name=$\Phi_{\on{KZ}}$, 
  description={KZ associator}, sort={S}}

\newglossaryentry{KZAsse}{type=associators, name=$e(\tau)$, 
  description={Elliptic KZB associator}, sort={S}}
  
\newglossaryentry{KZAssg}{type=associators, name=$\Psi_{\on{KZ}}$, 
  description={Cyclotomic KZ associator}, sort={S}}
  
\newglossaryentry{KZAsseg}{type=associators, name=$e^{\Gamma}(\tau)$, 
  description={Ellipsitomic KZB associator}, sort={S}}

\newglossaryentry{Ag}{type=associators, name=$A^{\Gamma}(\tau)$, 
  description={$A$-ellipsitomic KZB associator}, sort={S}  }

\newglossaryentry{Bg}{type=associators, name=$B^{\Gamma}(\tau)$, 
  description={$B$-ellipsitomic KZB associator} , sort={S} }

\newglossaryentry{Gsg}{type=associators, name=$G_{s,\gamma}(\tau)$, 
  description={Eisenstein-Hurwitz series}, sort={S}  }
  
\newglossaryentry{ConfCn}{type=spaces, name=\ensuremath{\on{Conf}(\C,n)}, 
  description={Configuration space of $n$ points in $\C$}, sort={S}}
  
\newglossaryentry{CCn}{type=spaces, name=\ensuremath{\on{C}(\C,I)}, 
  description={Reduced configuration space of $I$-indexed points in $\C$}, sort={S}}
  
\newglossaryentry{FMCn}{type=spaces, name=\ensuremath{\overline{\on{C}}(\C,I)}, 
  description={Fulton-MacPherson compactification of \ensuremath{\on{C}(\C,I)}}, sort={S}}

\newglossaryentry{ConfCn1}{type=spaces, name=\ensuremath{\on{Conf}(\C^{\times},n)}, 
  description={Configuration space of $n$ points in $\C^{\times}$}, sort={S}}
  
\newglossaryentry{CCI1}{type=spaces, name=\ensuremath{\on{C}(\C^{\times},I)}, 
  description={Reduced configuration space of $I$-indexed points in $\C^{\times}$}, sort={S}}

\newglossaryentry{ConfCnM}{type=spaces, name=\ensuremath{\on{Conf}(\C^{\times},n,M)}, 
  description={$M$-decorated configuration space of $n$ points in $\C^{\times}$}, sort={S}}
  
\newglossaryentry{CCIM}{type=spaces, name=\ensuremath{\on{C}(\C^{\times},I,\Gamma)}, 
  description={Reduced $\Gamma$-decorated configuration space of $I$-indexed points in $\C^{\times}$}, sort={S}}
  
\newglossaryentry{ConfTn}{type=spaces, name=\ensuremath{\on{Conf}(\mathbb{T},n)}, 
  description={Configuration space of $n$ points in \ensuremath{\mathbb{T}}}, sort={S}}

\newglossaryentry{CTn}{type=spaces, name=\ensuremath{\on{C}(\mathbb{T},I)}, 
  description={Reduced configuration space of $I$-indexed points in \ensuremath{\mathbb{T}}}, sort={S}}

\newglossaryentry{FMTn}{type=spaces, name=\ensuremath{\overline{\on{C}}(\mathbb{T},I)}, 
  description={Fulton-MacPherson compactification of \ensuremath{\on{C}(\mathbb{T},I)}}, sort={S}}

\newglossaryentry{ConfTIG}{type=spaces, name=\ensuremath{\on{Conf}(\mathbb{T},I,\Gamma)}, 
  description={$\Gamma$-decorated configuration space of $I$-indexed points in \ensuremath{\mathbb{T}}}, sort={S}}
  
\newglossaryentry{CTIG}{type=spaces, name=\ensuremath{\on{C}(\mathbb{T},I,\Gamma)}, 
  description={Reduced $\Gamma$-decorated configuration space of $I$-indexed points in \ensuremath{\mathbb{T}}}, sort={S}}
  
\newglossaryentry{FMTIG}{type=spaces, name=\ensuremath{\overline{\on{C}}(\mathbb{T},I,\Gamma)}, 
  description={Fulton-MacPherson compactification of \ensuremath{\on{C}(\mathbb{T},I,\Gamma)}}, sort={S}}
  
\newglossaryentry{bM1ng}{type=spaces, name=$\bar\cM_{1,n}^\Gamma$, 
  description={Reduced moduli space of $\Gamma$-structured $n$-marked elliptic curves}, sort={S}}  

\newglossaryentry{M1ng}{type=spaces, name=$\cM_{1,n}^\Gamma$, 
  description={Non-reduced moduli space of $\Gamma$-structured $n$-marked elliptic curves}, sort={S}}  

\newglossaryentry{M1nng}{type=spaces, name=$\M_{1,[n]}^\Gamma$, 
  description={Non-reduced moduli space of $\Gamma$-structured unorderly $n$-marked elliptic curves}, sort={S}}    
  
\newglossaryentry{bM1nng}{type=spaces, name=$\bar\cM_{1,[n]}^\Gamma$, 
  description={Reduced moduli space of $\Gamma$-structured unorderly $n$-marked elliptic curves}, sort={S}}    

\newglossaryentry{tn}{type=liealgebras, name=$\t_n(\kk)$, 
  description={Rational Kohno-Drinfeld Lie $\kk$-algebra}, sort={S}}
  
\newglossaryentry{t1n}{type=liealgebras, name=$\t_{1,n}(\kk)$, 
  description={Elliptic Kohno-Drinfeld Lie $\kk$-algebra}, sort={S}}
  
\newglossaryentry{tnM}{type=liealgebras, name=$\t_n^M(\kk)$, 
  description={$M$-cyclotomic Kohno-Drinfeld Lie $\kk$-algebra}, sort={S}}
  
\newglossaryentry{t1nG}{type=liealgebras, name=$\t_{1,n}^\Gamma(\kk)$, 
  description={$\Gamma$-ellipsitomic Kohno-Drinfeld Lie $\kk$-algebra}, sort={S}}

\newglossaryentry{dtg}{type=liealgebras, name=$\tilde\d^\Gamma$, 
  description={Intermediate twisted derivations Lie algebra}, sort={S}}  

\newglossaryentry{dg}{type=liealgebras, name=$\d^\Gamma$, 
  description={Twisted derivations Lie algebra}, sort={S}}  
  
\newglossaryentry{Hgl}{type=liealgebras, name=\ensuremath{H_{n}(\g,\l^*)}, 
  description={Hecke algebra of the pair $(\g,\l)$}, sort={S}  }

\newglossaryentry{Hgh}{type=liealgebras, name=\ensuremath{H_{n}(\g,\h^*_{reg})}, 
  description={Reduced Hecke algebra of the pair $(\g,\h)$}, sort={S}  }

\newtheorem*{theorem*}{Theorem}
\newtheorem{theorem}{Theorem}[chapter]
\newtheorem{lemma}[theorem]{Lemma}
\newtheorem{proposition}[theorem]{Proposition}

\newtheorem{notation}[theorem]{Notation}

\theoremstyle{definition}
\newtheorem{definition}[theorem]{Definition}
\newtheorem{example}[theorem]{Example}

\theoremstyle{remark}
\newtheorem{remark}[theorem]{Remark}

\numberwithin{section}{chapter}
\numberwithin{equation}{chapter}

\makeindex

\begin{document}

\frontmatter

\title{Ellipsitomic associators}

\author{Damien Calaque}
\address{IMAG, Univ Montpellier, CNRS, Montpellier, France}
\email{damien.calaque@umontpellier.fr}

\author{Martin Gonzalez}
\address{Institut de Recherche Technologique SystemX, Palaiseau, France}
\email{martin.gonzalez@irt-systemx.fr}

\date{}



\begin{abstract}
We develop a notion of ellipsitomic associators by means of operad theory. 
We take this opportunity to review the operadic point-of-view on Drinfeld associators and to provide such an operadic 
approach for elliptic associators too. 
We then show that ellipsitomic associators do exist, using the monodromy of the universal ellipsitomic KZB connection, 
that we introduced in a previous work. 
We finally relate the KZB ellipsitomic associators to certain Eisenstein series associated with congruence subgroups of 
$\mathrm{SL}_2(\mathbb{Z})$, and to twisted elliptic multiple zeta values. 

\bigskip

\noindent\textit{\textbf{R\'esum\'e (Associateurs ellipsitomiques).}} --- Nous développons la notion d'associateur ellipsitomique 
au moyen de la théorie des opérades. Nous saisissons cette opportunité pour revoir le point de vue opéradique sur les associateurs
de Drinfeld, et pour fournir également une telle approche opéradique pour les associateurs elliptiques. 
Nous montrons ensuite que les associateurs ellipsitomiques existent, en utilisant la monodromie de la connexion KZB ellipsitomique 
universelle, que nous avions introduite dans un travail précédent. 
Nous relions pour finir les associateurs ellipsitomiques KZB à certaines séries d'Eisenstein associées aux sous-groupes de congruence 
de $\mathrm{SL}_2(\mathbb{Z})$, et aux valeurs zêta multiples elliptiques tordues. 
\end{abstract}

\maketitle

\tableofcontents

\chapter*{Introduction}

The torsor of associators was introduced by Drinfeld \cite{DrGal} in the early nineties, 
in the context of quantum groups and prounipotent Grothendieck--Teichmuller theory. 
Since then, it has proven to have deep connections with several areas of mathematics (and physics): 
number theory \cite{LeMu}, deformation quantization \cite{EK,Konts-op,Tam1}, Chern--Simons theory and low-dimensional 
topology \cite{Konts-knots}, algebraic topology and the little disks operad \cite{Tam2}, Lie theory and the 
Kashiwara--Vergne conjecture \cite{AET,AT} etc... In this paper we are mostly 
interested in the operadic and also number theoretic aspects. For instance, 
\begin{itemize}
\item[(a)] The torsor of associators can be seen as the torsor of isomorphisms between two operads in (prounipotent)
groupoids related to the little disks operad, denoted $\PaB$ and $\PaCD$ (for 
\textit{parenthesized braids} and \textit{parenthesized chord diagrams}). These can be understood as 
the Betti and de Rham fundamental groupoids of an operad of suitably compactified configuration spaces 
of points in the plane. See chapter \ref{Section2} for more details, and accurate references. 
\item[(b)] It is expected that associators can be seen as generating series for (variations on 
motivic) multiple zeta values (MZVs), as was observed for the KZ associator \cite{LeMu} and the Deligne 
associator \cite{Brown-on-Deligne}. 
\end{itemize}

\medskip

The first example of an associator was produced by Drinfeld as the renormalized holonomy of a universal 
version of the so-called Knizhnik--Zamolodchikov (KZ) connection \cite{DrGal}, which is defined on a trivial 
principal bundle over the configuration space of points in the plane. The defining equations 
of an associator can be deduced from intuitive geometric 
reasonings about paths on configuration spaces, and they lead to representations of braid groups. 

\medskip

Enriquez, Etingof and the first author \cite{CEE} 
introduced a universal version of an elliptic variation on the KZ connection (known as 
Knizhnik--Zamolodchikov--Bernard, or KZB, connection, as the extension to higher genus is due to Bernard 
\cite{B1,B2}). It is a holomorphic connection defined on a non trivial principal bundle over configuration spaces 
of points on an elliptic curve. They showed that 
\begin{itemize}
\item The holonomy of the universal KZB connection along fundamental cycles of an ellitpic curve 
satisfy relations which lead to representations of braid groups on the ($2$-)torus. 
\item They also satisfy a modularity property, that is a consequence of the fact that the (universal) KZB 
connection extends from configuration spaces of points on an elliptic curve to moduli spaces of marked 
elliptic curves (see also \cite{LR} for when there are at most $2$ marked points). 
\end{itemize}
Enriquez later introduced the notion of an elliptic associator \cite{En2}, and proved that the holonomy 
of the universal elliptic KZB connection does produce, for every elliptic curve, an example of elliptic 
associator. The class of elliptic associators that are obtained \textit{via} this procedure are called 
\textit{KZB associators}. In another work \cite{En3}, Enriquez defined and studied an 
elliptic version of MZVs; he showed that KZB associators are generating series for elliptic MZVs (eMZVs). 

\medskip

In a recent paper \cite{CaGo} we introduced a generalization of the universal elliptic KZB connection: 
the universal \textit{ellipsitomic} KZB connection. It is defined over twisted configuration spaces, where 
the twisting is by a finite quotient $\Gamma$ of the fundamental group of the elliptic curve. 
When $\Gamma=1$ is trivial, one recovers the universal elliptic KZB connection. 

\medskip

The aim of the present paper is two-fold. 
\begin{itemize}
\item[(a)] First we provide an operadic interpretation of elliptic associators. 
We extend this approach to the ellipsitomic case, use the language of operads to define ellipsitomic associators, and 
sketch the rudiments of an ellipsitomic Grothendieck--Teichm\"uller theory. 
\item[(b)] Then we show that holonomies of the universal ellipsitomic KZB connection along suitable paths produce 
examples of ellipsitomic associators, and are generating series for elliptic multiple polylogarithms at 
$\Gamma$-torsion points, that are similar to the twisted elliptic MZVs (teMZVs) studied in \cite{M2} by 
Broedel--Matthes--Richter--Schlotterer. 
\end{itemize}

\medskip

Our work fits in a more general program that aims at studying associators for an oriented  surface together with a finite 
group acting on it. We summarize in the following table the contributions to this program 
that we are aware of: 

\begin{center}
\begin{tabular}{|c|c|p{2.1cm}|p{1.5cm}|p{3.6cm}|p{1.7cm}|}
\hline
gen. & group & associators	& operadic approach & Universal connection / existence proof & coefficients	
\\ \hline 
$0$ & trivial & \cite{DrGal} & \cite{BN,Fresse}	& rational KZ \cite{DrGal} / \textit{ibid.} & MZVs \cite{LeMu}	
\\ \hline 
$0$ & $\Z/N\Z$ & cyclotomic associators \cite{En} & \cite{CG-cyclo}	& trigonometric KZ \cite{En} / \textit{ibid.} & colored MZVs \cite{En}
\\ \hline 
$0$ & fin.~$\subset PSU_2(\mathbb{C})$ & unknown & unknown & \cite{Maa} / unknown & unknown	
\\ \hline 
$1$ & trivial & elliptic associators \cite{En2} & this paper (Sec.~3) & elliptic KZB \cite{CEE} / \cite{En2} & eMZVs \cite{En3}  
\\ \hline 
$1$ & $\Z/M\Z\times \Z/N\Z$ & ellipsitomic associators (this paper) & this paper (Sec.~4 \& 5) & ellipsitomic KZB \cite{CaGo} / this paper (Sec.~6) & this paper (Sec.~7)
\\ \hline 
$>1$ & trivial & \cite{GAssoc} & \cite{GAssoc} & KZB \cite{En4} / conj.~in \cite{GAssoc} & maybe \cite{Go} 
\\ \hline
\end{tabular}
\end{center}

\medskip

\section*{Description of the paper}
The first chapter is devoted to some recollection on operads and operadic modules, 
with some emphasis on specific features when the underlying category is the one of groupoids. 
Chapter 2 also recollects known results, about the operadic approach to (genuine) associators and 
to various Grothendieck--Teichm\"uller groups. The main results we state are taken from the recent book 
\cite{Fresse}. 

The main goal of chapter 3 is to provide a similar treatment of elliptic associators, using 
operadic modules in place of sole operads. We show in particular that (a variant of) the universal 
elliptic structure $\PaB_{e\ell\ell}$ (resp.~its graded/de Rham counterpart $G\PaCD_{e\ell\ell}$) 
from \cite{En2} carries the structure of an operadic module in groupoids over the operad in groupoid 
$\PaB$ (resp.~$G\PaCD$). 
We provide a generators and relations presentation for $\PaB_{e\ell\ell}$ (Theorem \ref{PaBell}), 
and deduce from it the following 
\begin{theorem*}[Theorem \ref{thm-intro1}] 
The torsor of elliptic associators from \cite{En2} coincides with the torsor of isomorphisms from (a variant of)
$\PaB_{e\ell\ell}$ to $G\PaCD_{e\ell\ell}$ that are the identity on objects. Similarly, the elliptic 
Grothendieck--Teichm\"uller group (resp.~its graded version) is isomorphic to the group of automorphisms 
of $\PaB_{e\ell\ell}$ (resp.~of $G\PaCD_{e\ell\ell}$) that are the identity on objects. 
\end{theorem*}

The fourth chapter introduces a generalization of $\PaB_{e\ell\ell}$, with an additional labelling/twisting by 
elements of $\Gamma$ (recall that $\Gamma$ is the group of deck transformations of a finite cover of the 
torus by another torus). We give a geometric definition of the operadic module $\PaB_{e\ell\ell}^\Gamma$ 
of parenthesized ellipsitomic braids, and then provide a presentation by generators and relations for it 
(Theorem \ref{PaB:ell:G}). In the fifth chapter we define an operadic module of ellipsitomic chord diagrams, 
that mixes features of $\PaCD_{e\ell\ell}$ from chapter 3, and of the moperad of cyclotomic chord diagrams 
from \cite{CG-cyclo}. This allows us to identify ellipsitomic associators, which we define in purely operadic 
terms, with series satisfying certain algebraic equations (Theorem \ref{th:ass:tell}). 

Chapter 6 is devoted to the proof of the following
\begin{theorem*}[Theorem \ref{theorem:twistedKZBass}]
The set of ellipsitomic associators over $\C$ is non-empty. 
\end{theorem*}
The proof makes crucial use of the ellipsitomic KZB connection, introduced in our previous work \cite{CaGo}, 
and relies on a careful analysis of its monodromy. We actually prove that one can associate an ellipsitomic 
associator with every element of the upper half-plane (Theorem \ref{theorem:twistedKZBass}). 
In the last chapter we quickly explore some number theoretic and modular aspects of the coefficients of 
the ``KZB produced'' ellipsitomic associators from the previous chapter. 

Finally, in an appendix we provide an alternative presentation for $\PaB_{e\ell\ell}^{\mathbf{\Gamma}}$. 

\section*{Acknowledgements}
Both authors are grateful to Adrien Brochier, Benjamin Enriquez, and Pierre Lochak, as well as Richard Hain, Nils Matthes, and Eric Hopper, 
for numerous conversations and suggestions. 
We also thank the referee for their comments and careful reading, that helped a lot improving this paper. 

The first author has received funding from the Institut Universitaire de France, and from the European Research Council 
(ERC) under the European Union's Horizon 2020 research and innovation programme (Grant Agreement No. 768679). 

This paper is extracted from the second author's PhD thesis \cite{G1} at Sorbonne Universit\'e, and part of this 
work has been done while the second author was visiting the Institut Montpelli\'erain Alexander Grothendieck,
thanks to the financial support of the Institut Universitaire de France. The second author warmly thanks the 
Max-Planck Institute for Mathematics in Bonn and Universit\'e d'Aix-Marseille, for their hospitality and excellent 
working conditions.

\mainmatter
\chapter{Background material on operads and groupoids}
\label{Section1}

In this chapter we fix a symmetric monoidal category $(\mathcal C,\otimes,\mathbf{1})$ 
having small colimits. Let us assume for simplicity of exposition that $\otimes$ commutes with these\footnote{This latter assumption is not necessary 
(and we will have to get rid of it when considering the monoidal structure given by the direct sum of Lie algebras): if the monoidal product does not commute 
with colimits, the category of $\mathfrak{S}$-module still has enough structure so that one can define monoids and modules in it. Characterizations in terms of 
partial compositions remain unchanged. We refer to \cite{Ching} for more details. }. 


\section{$\mathfrak{S}$-modules}

An \textit{$\mathfrak{S}$-module} (in $\mathcal C$) is a 
functor $S:\mathbf{Bij}\to \mathcal C$, where $\mathbf{Bij}$ denotes the category of 
finite sets with bijections as morphisms. It can also be defined as a collection 
$\left(S(n)\right)_{n\geq0}$ of objects of $\mathcal C$ such that $S(n)$ is endowed 
with a right action of the symmetric group $\mathfrak{S}_n$ for every $n$; 
one has $S(n):=S(\{1,\dots,n\})$. 
A \textit{morphism} of $\mathfrak{S}$-modules $\varphi:S\to T$ is a natural transformation. 
It is determined by the data of a collection $\varphi(n):S(n)\to T(n)$ of 
$\mathfrak S_n$-equivariant morphisms in $\mathcal C$. 

\medskip

The category $\mathfrak S$-mod of $\mathfrak S$-modules is naturally endowed with a 
symmetric monoidal product $\otimes$ defined as follows: 
\[
(S\otimes T)(n):=\coprod_{p+q=n}\left(S(p)\otimes T(q)\right)_{\mathfrak S_p\times\mathfrak S_q}^{\mathfrak S_n}\,.
\]
Here, if $H\subset G$ is a group inclusion, then $(-)_H^G$ is left adjoint to the restriction functor 
from the category of objects carrying a $G$-action to the category of objects carrying an $H$-action. 

The symmetric sequence $\mathbf{1}_{\otimes}$ defined by 
\[
\mathbf{1}_{\otimes}(n):=
\begin{cases} 
\mathbf{1} & \mathrm{if}~n=0 \\
\emptyset & \mathrm{else}
\end{cases}
\]
is a monoidal unit for $\otimes$. 

\medskip

There is another (non-symmetric) monoidal product $\circ$ on $\mathfrak S$-mod, defined as follows: 
\[
(S\circ T)(n):=\coprod_{k\geq0}T(k)\underset{\mathfrak S_k}{\otimes}\left(S^{\otimes k}(n)\right)\,.
\]
Here, if $H$ is a group and $X,Y$ are objects carrying an $H$-action, then 
\[
X\underset{H}{\otimes}Y:={\rm coeq}\left(
\underset{h\in H}{\coprod}X\otimes Y\doublerightarrow{h\otimes{\rm id}}{{\rm id}\otimes h} X\otimes Y\right)\,.
\] 

The symmetric sequence $\mathbf{1}_{\circ}$ defined by 
\[
\mathbf{1}_{\circ}(n):=
\begin{cases} 
\mathbf{1} & \mathrm{if}~n=1 \\
\emptyset & \mathrm{else}
\end{cases}
\]
is a monoidal unit for $\circ$. 


\section{Operads}

An \textit{operad} (in $\mathcal C$) is a unital monoid in $(\mathfrak S\textrm{-mod},\circ,\mathbf{1}_{\circ})$. 
The category of operads in $\mathcal C$ will be denoted $\tmop{Op}\mathcal C$. 

More explicitly, an operad structure on a $\mathfrak S$-module $\mathcal O$ is the data:
\begin{itemize}
\item of a unit map $e:\mathbf{1}\to\mathcal O(1)$. 
\item for every sets $I,J$ and any element $i\in I$, of a \textit{partial composition} 
\[
\circ_i:\mathcal O(I)\otimes\mathcal O(J)\longrightarrow \mathcal O\left(J\sqcup I-\{i\}\right)
\]
\end{itemize}
satisfying the following constraints: 
\begin{itemize}
\item for every sets $I,J,K$, with elements $i\in I$, $j\in J$, the following diagram commutes: 
\[
\xymatrix{
\mathcal O(I)\otimes\mathcal O(J)\otimes\mathcal O(K) \ar[d]^{{\rm id}\otimes\circ_j}\ar[rr]^{\circ_i\otimes{\rm id}} &~& 
\mathcal O\left(J\sqcup I-\{i\}\right)\otimes\mathcal O(K) \ar[d]^{\circ_j} \\
\mathcal O(I)\otimes\mathcal O\left(K\sqcup J-\{j\}\right) \ar[rr]^{\circ_i} &~& \mathcal O\left(K\sqcup J\sqcup I-\{i,j\}\right)
}
\]
\item for every sets $I,J_1,J_2$, with elements $i_1,i_2\in I$, the following diagram commutes: 
\[
\xymatrix{
\mathcal O(I)\otimes\mathcal O(J_1)\otimes\mathcal O(J_2) \ar[d]^{(\circ_{i_2}\otimes{\rm id})(23)}\ar[rr]^{\circ_{i_1}\otimes{\rm id}} &~& 
\mathcal O\left(J_1\sqcup I-\{i_1\}\right)\otimes\mathcal O(J_2) \ar[d]^{\circ_{i_2}} \\
\mathcal O\left(J_2\sqcup I-\{i_2\}\right)\otimes\mathcal O(J_1) \ar[rr]^{\circ_{i_1}} &~& \mathcal O\left(J_2\sqcup J_1\sqcup I-\{i_1,i_2\}\right)
}
\]
\item for every sets $I,I',J$, $i\in I$, with a bijection $\sigma:I\to I'$, the following diagram commutes: 
\[
\xymatrix{
\mathcal O(I)\otimes\mathcal O(J) \ar[d]^{\circ_i}\ar[rr]^{\mathcal O(\sigma)} &~& \mathcal O(I')\otimes\mathcal O(J)\ar[d]^{\circ_{\sigma(i)}}
 \\\mathcal O\left(J\sqcup I-\{i\}\right) \ar[rr]^{\mathcal O({\rm id}\sqcup\sigma_{|I-\{i\}})} &~& \mathcal O\left(J\sqcup I'-\{\sigma(i)\}\right)
}
\]
\item for every set $I$,with $i\in I$, the following diagrams commute: 
\[
\xymatrix{
\mathbf{1}\otimes\mathcal O(I)\ar[rd]_{\simeq}\ar[r]^{e\otimes\mathrm{id}} & \mathcal O(\{1\})\otimes\mathcal O(I) \ar[d]^{\circ_{1}} \\
 & \mathcal O(I)
}
\quad
\xymatrix{
\mathcal O(I)\otimes\mathbf{1}\ar[d]_{\simeq}\ar[r]^{\mathrm{id}\otimes e} & \mathcal O(I)\otimes \mathcal O(\{1\}) \ar[d]^{\circ_i} \\
\mathcal O(I) \ar[r]_-{i\mapsto 1}^-{\simeq} & \mathcal O\left(I\sqcup \{1\}-\{i\}\right)
}
\]
\end{itemize}
\begin{example}
Let $X$ be an object of $\mathcal C$. Then we define, for any finite set $I$, 
the set $\underline{\rm End}(X)(I):={\rm Hom}_{\mathcal C}(X^{\otimes I},X)$. Composition of tensor products of maps provide 
$\underline{\rm End}(X)$ with the structure of an operad in sets. 
\end{example}
Given an operad in sets $\mathcal O$, an \textit{$\mathcal O$-algebra in $\mathcal C$} is an object $X$ of $\mathcal C$ together 
with a morphism of operads $\mathcal O\to \underline{\rm End}(X)$. 


\section{Example of an operad: Stasheff polytopes}

To any finite set $I$ we associate the configuration space 
$\textrm{Conf}(\mathbb{R},I)=\{\mathbf{x}=(x_i)_{i\in I}\in \mathbb{R}^I|x_i\neq x_j\textrm{ if }i\neq j\}$
and its reduced version 
\[
\textrm{C}(\mathbb{R},I):=\textrm{Conf}(\mathbb{R},I)/\mathbb{R}\rtimes\mathbb{R}_{>0}\,.
\]
The Axelrod--Singer--Fulton--MacPherson compactification\footnote{We are using the differential geometric 
compactification from \cite{AS}, which is an analog of the algebro-geometric one from \cite{FMP}. } 
$\overline{\textrm{C}}(\mathbb{R},I)$ of $\textrm{C}(\mathbb{R},I)$ is a disjoint union of $|I|$-th Stasheff 
polytopes \cite{Sta}, indexed by $\mathfrak S_I$.  
The boundary $\partial\overline{\textrm{C}}(\mathbb{R},I):=\overline{\textrm{C}}(\mathbb{R},I)-\textrm{C}(\mathbb{R},I)$ 
is the union, over all partitions $I=J_1\coprod\cdots\coprod J_k$, of 
\[
\partial_{J_1,\cdots,J_k}\overline{\textrm{C}}(\mathbb{R},I):=\prod_{i=1}^k\overline{\rm C}(\mathbb{R},J_i)\times\overline{\rm C}(\mathbb{R},k)\,.
\]
The inclusion of boundary components provides $\overline{\rm C}(\mathbb{R},-)$ with the structure of an 
operad in topological spaces (where the monoidal structure is given by the cartesian product). 

One can see that $\overline{\textrm{C}}(\mathbb{R},I)$ is actually a manifold with corners, and that, 
considering only zero-dimensional strata of our configuration spaces, we get a suboperad 
$\mathbf{Pa}\subset\overline{\textrm{C}}(\mathbb{R},-)$ that can be shortly described as follows: 
\begin{itemize}
\item $\mathbf{Pa}(I)$ is the set of pairs $(\sigma,p)$ with $\sigma$ is a linear order on $I$ and $p$ 
a maximal parenthesization of $\underbrace{\bullet\cdots\bullet}_{|I|~{\rm times}}$,
\item the operad structure is given by substitution. 
\end{itemize}
Notice that $\mathbf{Pa}$ is actually an operad in sets, and that $\mathbf{Pa}$-algebras are nothing else than \textit{magmas}. 


\section{Modules over an operad: Bott-Taubes polytopes}

A \textit{module} over an operad $\mathcal O$ (in $\mathcal C$) is a right $\mathcal O$-module in 
$(\mathfrak S\textrm{-mod},\circ,\mathbf{1}_{\circ})$. 
Notice that any operad is a module over itself. We let the reader find the very explicit description 
of a module in terms of partial compositions, as for operads. 

\medskip

To any finite set $I$ we associate the configuration space 
$\textrm{Conf}(\mathbb{S}^1,I)=\{\mathbf{x}=(x_i)_{i\in I}\in (\mathbb{S}^1)^I|x_i\neq x_j\textrm{ if }i\neq j\}$
and its reduced version 
\[
\textrm{C}(\mathbb{S}^1,I):=\textrm{Conf}(\mathbb{S}^1,I)/\mathbb{S}^1\,.
\]
The Axelrod--Singer--Fulton--MacPherson compactification $\overline{\textrm{C}}(\mathbb{S}^1,I)$ of 
$\textrm{C}(\mathbb{S}^1,I)$ is a disjoint union of $|I|$-th Bott--Taubes polytopes \cite{BT}, indexed by $\mathfrak S_I$.  
The boundary $\partial\overline{\textrm{C}}(\mathbb{S}^1,I):=\overline{\textrm{C}}(\mathbb{S}^1,I)-\textrm{C}(\mathbb{S}^1,I)$ 
is the union, over all partitions $I=J_1\coprod\cdots\coprod J_k$, of 
\[
\partial_{J_1,\cdots,J_k}\overline{\textrm{C}}(\mathbb{S}^1,I):=\prod_{i=1}^k\overline{\rm C}(\mathbb{R},J_i)\times\overline{\rm C}(\mathbb{S}^1,k)\,.
\]
The inclusion of boundary components provides $\overline{\rm C}(\mathbb{S}^1,-)$ with the structure of a module over the operad 
$\overline{\rm C}(\mathbb{R},-)$ in topological spaces. 

One can see that $\overline{\textrm{C}}(\mathbb{S}^1,I)$ is actually a manifold with corners, and that, considering 
only zero-dimensional strata of our configuration spaces, we get $\mathbf{Pa}\subset\overline{\textrm{C}}(\mathbb{S}^1,-)$, 
which is a module over $\mathbf{Pa}\subset\overline{\textrm{C}}(\mathbb{R},-)$. 


\section{Convention: pointed versions}\label{sec-pointings}

Observe that there is an operad $Unit$ defined by 
\[
Unit(n)=\begin{cases}
\mathbf{1} &\mathrm{if}~n=0,1 \\
\emptyset &\mathrm{else}
\end{cases}
\]
By convention, all our operads $\mathcal O$ will be $Unit$-pointed and reduced, in the sense that they will come equipped 
with a specific operad morphism $Unit\to\mathcal O$ that is an isomorphism in arity $\leq 1$: $\mathcal O(n)\simeq\mathbf{1}$ 
if $n=0,1$. 
Morphisms of operads are required to be compatible with this pointing. 

\medskip

Now, if $\mathcal P$ is an $\mathcal O$-module, then it naturally becomes a $Unit$-module as well, by restriction. 
By convention, all our modules will be pointed as well, in the sense that they will come equipped with a specific 
$Unit$-module morphism $Unit\to\mathcal P$ that is an isomorphism in arity $\leq 1$. 
Morphisms of modules are also required to be compatible with the pointing. 

\medskip

The main reason for this convention is that we need the following features, that we have in the case of compactified 
configuration spaces: 
\begin{itemize}
\item For operads and modules, we want to have ``deleting operations'' $\mathcal O(n)\to\mathcal O(n-1)$ that decrease arity. 
\item For modules, we want to be able to see the operad ``inside'' them, i.e.~we want to have distinguished morphism 
$\mathcal O\to\mathcal P$ of $\mathfrak{S}$-modules. 
\end{itemize}


\section{Group actions}\label{sec-gpaction}

Let $G$ be a $*$-module in group, where $*$ is the terminal operad: the partial composition $\circ_i$ 
is a group morphism $G(n)\to G(n+m-1)$. 
\begin{example}
Let $\Gamma$ be a group, we consider the $\mathfrak{S}$-module in groups 
$\overline{\Gamma}:=\{\Gamma^n/\Gamma^{diag}\}_{n\geq0}$, 
where $\Gamma^{diag}$ denotes the normal closure of the diagonal subgroup in each $\Gamma^n$. 
It is equipped with the following $*$-module structure: the $i$-th partial composition is given by the partial diagonal morphism
\begin{eqnarray*}
\Gamma^n/\Gamma & \longrightarrow & \Gamma^{n+m-1}/\Gamma \\
~[\gamma_1,\dots,\gamma_n] & \longmapsto & 
[\gamma_1,\dots,\gamma_{i-1},\underbrace{\gamma_i,\dots,\gamma_i}_{m~\mathrm{times}},\gamma_{i+1},\dots,\gamma_n]
\end{eqnarray*}
\end{example}
Given an operad $\mathcal O$ in $\mathcal C$, we say that an $\mathcal O$-module $\mathcal P$ carries a \textit{$G$-action} if 
\begin{itemize}
\item for every $n\geq0$, there is an $\mathfrak{S}_n$-equivariant left action $G(n)\times\mathcal P(n)\to\mathcal P(n)$. 
\item for every $m\geq0$, $n\geq0$, and $1\leq i\leq n$, the partial composition 
\[
\circ_i:\mathcal P(n)\otimes\mathcal O(m)\longrightarrow \mathcal P(n+m-1)
\]
is equivariant along the above group morphism $G(n)\to G(n+m-1)$. 
\end{itemize}
A morphism $\mathcal P\to\mathcal Q$ of $\mathcal O$-modules with $G$-action 
is said \textit{$G$-equivariant} if, for every $n\geq0$, the map $\mathcal P(n)\to\mathcal Q(n)$ is $G(n)$-equivariant. 

\medskip

Given a group $\Gamma$, we say that an $\mathcal O$-module $\mathcal P$ carries a \textit{diagonally trivial action of} 
$\Gamma$ if it carries a $\overline{\Gamma}$-action. 

\medskip

The quotient $G\backslash\mathcal P$ of an $\mathcal O$-module $\mathcal P$ with a $G$-action is defined as follows: 
\begin{itemize}
\item For every $n\geq 0$, $\big(G(n)\backslash\mathcal P\big)(n):=G(n)\backslash\mathcal P(n)$; 
\item The equivariance of the partial composition $\circ_i$ tels us that it descends to the quotient
\[
\big(G(n)\backslash\mathcal P(n)\big)\otimes\mathcal O(m)\longrightarrow G(n+m-1)\backslash\mathcal P(n+m-1)\,.
\] 
\end{itemize}


\section{Semi-direct products and fake pull-backs}\label{sec-semifake}

Let $\mathbf{Grpd}$ be the category of groupoids. 
For a group $G$, we denote $G-\mathbf{Grpd}$ the category of groupoids equipped with a $G$-action. 
There is a \textit{semi-direct product} functor 
\begin{eqnarray*}
G-\mathbf{Grpd} & \longrightarrow	& \mathbf{Grpd}_{/G}	\\
\mathcal P 		& \longmapsto		& \mathcal P\rtimes G
\end{eqnarray*}
where the group $G$ is viewed as a groupoid with a single object, and where $\mathcal P\rtimes G$ is defined as follows: 
\begin{itemize}
\item Objects of $\mathcal P\rtimes G$ are just objects of $\mathcal P$; 
\item In addition to the arrows of $\mathcal P$, for every $g\in G$, and for every object $\mathbf{p}$ of $P$, 
there is an arrow $g\cdot\mathbf{p}\overset{g}{\to}\mathbf{p}$; 
\item These new arrows multiply together via the group multiplication of $G$;
\item For every morphism $f$ in $P$, and every $g\in G$, the relation $gfg^{-1}=g\cdot f$ holds. 
\end{itemize}

\begin{notation}
We warn the reader that we use all along the paper the following 
rather unusual convention for arrows in a groupoid, and more generally in a category: 
we often concatenate arrows rather than composing them. 
In other words, $f_1 f_2 = f_2 \circ f_1$. 
\end{notation}

\medskip

There is also a functor $\mathcal G$ going in the other direction
\begin{eqnarray*}
\mathbf{Grpd}_{/G}						& \longrightarrow	& G-\mathbf{Grpd}		\\
(\mathcal Q\overset{\varphi}{\to} G)	& \longmapsto		& \mathcal G(\varphi)
\end{eqnarray*}
that one can describe as follows: 
\begin{itemize}
\item The $G$-set of objects of $\mathcal G(\varphi)$ is the free $G$-set generated by $\on{Ob}(\mathcal Q)$; 
\item A morphism $(g,x)\to (h,y)$ in $\mathcal G(\varphi)$ is a morphism $x\overset{f}{\to} y$ in $\mathcal Q$ 
such that $g\varphi(f)=h$. 
\end{itemize}
\begin{example}\label{Ex:cobraid}
The groupoid $\mathcal G(\on{B}_n\to \SG_n)$ is the colored braid groupoid 
$\mathbf{CoB}(n)$ from \cite[\S5.2.8]{Fresse}. 
\end{example}
\begin{remark}
Given an object $q$ of $\mathcal Q$, $\on{Aut}_{\mathcal G(\varphi)}(g,q)$ is the kernel of the morphism 
$\on{Aut}_{\mathcal Q}(q)\to G$ for every $g\in G$. 
\end{remark}
\medskip

These constructions still make sense for modules over a given operad $\mathcal O$ whenever $G$ is an operadic $*$-module in groups. 
 
\bigskip

Let $\mathcal P, \mathcal Q$ be two operads (resp.~modules) in groupoids. If we are given a morphism 
$f:\on{Ob}(\mathcal P)\to\on{Ob}(\mathcal Q)$ between the operads (resp. operad modules) of objects of $\mathcal P$ 
and $\mathcal Q$, then (following \cite{Fresse}) we can define an operad (resp. operad module) $ f^\star \mathcal Q$ 
in the following way:
\begin{itemize}
\item $\on{Ob}(f^\star \mathcal Q) := \on{Ob}(\mathcal P)$,
\item $\on{Hom}_{(f^\star \mathcal Q)(n)}(p,q):=\on{Hom}_{\mathcal Q(n)}(f(p),f(q))$.
\end{itemize}
In particular, $f^\star \mathcal Q$, which we call the \textit{fake pull-back} of $\mathcal Q$ along $f$, inherits 
the operad structure of $\mathcal P$ for its operad of objects and that of $\mathcal Q$ for the morphisms. 
\begin{remark}
Notice that this is not a pull-back in the category of operads in groupoids.
\end{remark}


\section{Prounipotent completion}\label{sec-procomp}

Let $\kk$ be a $\Q$-ring. We denote by $\mathbf{CoAlg_{\KK}}$ the symmetric monoidal category of complete filtered topological 
coassociative cocommutative counital $\KK$-coalgebras, where the monoidal product is given by the completed tensor product 
$\hat{\otimes}_{\KK}$ over $\KK$. 

Let $\mathbf{Cat(CoAlg_{\kk})}$ be the category of small $\mathbf{CoAlg_{\KK}}$-enriched categories. It is symmetric monoidal as well, 
with monoidal product $\otimes$ defined as follows: 
\begin{itemize}
\item $\on{Ob}(C \otimes C'):=\on{Ob}(C) \times \on{Ob}(C')$. 
\item $\on{Hom}_{C \otimes C'}\big((c,c'),(d,d')\big):=\on{Hom}_{C}(c,d) \hat\otimes_{\KK} \on{Hom}_{C'}(c',d')$. 
\end{itemize}

All the constructions of the previous section still make sense, at the cost of replacing the group $G$ with 
its completed group algebra $\widehat{\kk G}$ (which is a Hopf algebra) in the semi-direct product construction. 

\medskip

Considering the cartesian symmetric monoidal structure on $\mathbf{Grpd}$, there is a symmetric monoidal functor 
\begin{eqnarray*}
\mathbf{Grpd} & \longrightarrow & \mathbf{Cat(CoAlg_{\kk})} \\
\mathcal G & \longmapsto & \mathcal G(\KK)
\end{eqnarray*}
defined as follows: 
\begin{itemize}
\item Objects of $\mathcal P(\KK)$ are objects of $\mathcal P$. 
\item For $a,b \in Ob(\mathcal P)$, 
\[
\on{Hom}_{\mathcal P(\KK)}(a,b)=\widehat{\KK\cdot\on{Hom}_{\mathcal P}(a,b)}\,.
\]
Here $\KK\cdot\on{Hom}_{\mathcal P}(a,b)$ is equipped with the unique coalgebra structure such that the elements 
of $\on{Hom}_{\mathcal P}(a,b)$ are grouplike (meaning that they are diagonal for the coproduct and that their 
counit is $1$), and the ``$~\widehat{~}$~'' refers to the completion with respect to the topology defined by the
 sequence $(\on{Hom}_{\mathcal I^k}(a,b)\big)_{k\geq0}$, where $\mathcal I^k$ is the category having the same objects 
 as $\mathcal P$ and morphisms lying in the $k$-th power (for the composition of morphisms) of kernels of the counits 
 of $\KK\cdot\on{Hom}_{\mathcal P}(a,b)$'s. 
\item For a functor $F:\mathcal P\to \mathcal Q$, $F(\KK):\mathcal P(\KK)\to \mathcal Q(\KK)$ is the functor given by $F$ 
on objects and by $\KK$-linearly extending $F$ on morphisms. 
\end{itemize}
Being symmetric monoidal, this functor sends operads in groupoids to operads in $\mathbf{Cat(CoAlg_{\kk})}$. 
\begin{example}
For instance, viewing $\Pa$ as an operad in groupoid (with only identities as morphisms), then $\Pa(\KK)$ is the operad 
in $\mathbf{Cat(CoAlg_{\kk})}$ with same objects as $\Pa$, and whose morphisms are 
\[
\on{Hom}_{\Pa(\KK)(n)}(a,b)=\begin{cases}
\KK & \mathrm{if}~a=b \\
0 & \mathrm{else}
\end{cases}
\]
with $\KK$ being equipped with the coproduct $\Delta(1)=1 \otimes 1$ and counit $\epsilon(1)=1$. 
\end{example}

\medskip

The functor we have just defined has a right adjoint 
\[
G:\mathbf{Cat(CoAlg_{\kk})}\longrightarrow \mathbf{Grpd}\,,
\]
that we can describe as follows: 
\begin{itemize}
\item For $C$ in $\mathbf{Cat(CoAlg_{\kk})}$, objects of $G(C)$ are objects of $C$. 
\item For $a,b \in Ob(\mathcal G)$, $\on{Hom}_{G(C)}(a,b)$ is the subset of grouplike elements in $\on{Hom}_{C}(a,b)$. 
\end{itemize}
Being right adjoint to a symmetric monoidal functor, it is lax symmetric monoidal, and thus it sends operads (resp.~modules) 
to operads (resp.~modules). 

We thus get a \textit{$\KK$-prounipotent completion} functor 
$\mathcal G\mapsto \hat{\mathcal G}(\KK):=G\big(\mathcal G(\KK)\big)$ for (operads and modules in) groupoids. 

\begin{remark}
Let $\varphi:G\to S$ be a surjective group morphism, and assume that $S$ is finite. 
One can prove that the prounipotent completion $\hat{\mathcal G}(\varphi)(\kk)$ of the construction from the previous 
section is isomorphic to $\mathcal G(\varphi(\kk))$, where $\varphi(\kk):G(\varphi,\kk)\to S$ is Hain's relative 
completion \cite{HainMalcev}. 
This essentially follows from that, when $S$ is finite, the kernel of the relative completion is the completion of the kernel. 
\end{remark}

\chapter{Operads associated with configuration spaces (associators)}
\label{Section2}


\section{Compactified configuration space of the plane}

To any finite set $I$ we associate a configuration space 
$$
\textrm{Conf}(\mathbb{C},I)=\{\mathbf{z}=(z_i)_{i\in I}\in \mathbb{C}^I|z_i\neq z_j\textrm{ if }i\neq j\}\,.
$$
We also consider its reduced version 
$$
\text{\gls{CCn}}:=\textrm{Conf}(\mathbb{C},I)/\mathbb{C}\rtimes\mathbb{R}_{>0}.
$$
We then consider the Axelrod--Singer--Fulton--MacPherson compactification \gls{FMCn} of $\textrm{C}(\mathbb{C},I)$. 
The boundary $\partial\overline{\textrm{C}}(\mathbb{C},I)=\overline{\textrm{C}}(\mathbb{C},I)-\textrm{C}(\mathbb{C},I)$ is made 
of the following irreducible components: for any partition 
$I=J_1\coprod\cdots\coprod J_k$ there is a component 
$$
\partial_{J_1,\cdots,J_k}\overline{\textrm{C}}(\mathbb{C},I)\cong \overline{\rm C}(\mathbb{C},k)\times
\prod_{i=1}^k\overline{\rm C}(\mathbb{C},J_i)\,.
$$

The inclusion of boundary components provides $\overline{\rm C}(\mathbb{C},-)$ with the structure of an operad 
in topological spaces. One can picture the partial operadic composition morphisms as follows:
\begin{center}
\includegraphics[width=143mm]{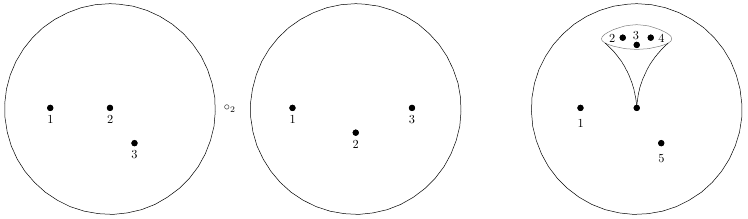}
\end{center}


\section{A presentation for the pure braid group}\label{paragraphe2.2}

The pure braid group \gls{PBn} is generated by elementary pure braids $P_{ij}$, $1\leq i<j\leq n$, 
which satisfy the following relations: 
\begin{flalign}
& (P_{ij},P_{kl})=1\quad\textrm{if }\{i,j\}\textrm{ and }\{k,l\}~\textrm{are non crossing}\,,  	\tag{PB1}\label{eqn:PB1} \\
& (P_{kj}P_{ij}P_{kj}^{-1},P_{kl})=1\quad\textrm{if }i<k<j<l\,, 																\tag{PB2}\label{eqn:PB2} \\
& (P_{ij},P_{ik}P_{jk})=(P_{jk},P_{ij}P_{ik})=(P_{ik},P_{jk}P_{ij})=1\quad\textrm{if }i<j<k\,. &\tag{PB3}\label{eqn:PB3}  
\end{flalign}
In this article we will represent the generator $P_{ij}$ in the following two equivalent ways: 
\begin{center}
\begin{tikzpicture}[baseline=(current bounding box.center)]
\tikzstyle point=[circle, fill=black, inner sep=0.05cm]
 \node[point, label=above:$1$] at (0,1) {};
 \node[point, label=below:$1$] at (0,-1.5) {};
 \draw[->,thick, postaction={decorate}] (0,1) -- (0,-1.45);
 \node[point, label=above:$i$] at (1,1) {};
 \node[point, label=below:$i$] at (1,-1.5) {};
  \node[point, label=above:$...$] at (2,1) {};
 \node[point, label=below:$...$] at (2,-1.5) {};
 \draw[thick] (1,1) .. controls (1,0.25) and (3.5,0.5).. (3.5,-0.25);
 \node[point, ,white] at (3,0.2) {};
 \node[point, label=above:$j$] at (3,1) {};
 \node[point, label=below:$j$] at (3,-1.5) {};
   \draw[->,thick, postaction={decorate}] (3,1) -- (3,-1.45);
   \draw[->,thick, postaction={decorate}] (3,0) -- (3,-1.45);
  \node[point, ,white] at (3,-0.7) {};
 \draw[->,thick, postaction={decorate}] (3.5,-0.25) .. controls (3.5,-1) and (1,-0.75).. (1,-1.45);
 \node[point, label=above:$n$] at (4,1) {};
 \node[point, label=below:$n$] at (4,-1.5) {};
 \draw[->,thick, postaction={decorate}] (4,1) -- (4,-1.45);
   \node[point, ,white] at (2,0.4) {};
    \node[point, ,white] at (2,-0.9) {};
   \draw[->,thick, postaction={decorate}] (2,1) -- (2,-1.45);
\end{tikzpicture}
$\qquad\longleftrightarrow\quad\sphericalangle\quad$
\begin{tikzpicture}[baseline=(current bounding box.center)] 
\tikzstyle point=[circle, fill=black, inner sep=0.05cm]
 \draw[loosely dotted] (1.5,-1.5) -- (-1.5,1.5); 
 \node[point, label=above:$n$] at (1.5,-1.5) {}; 
  \node[point, label=left:$i$] at (-0.5,0.5) {};
 \node[point, label=below:$j$] at (0.5,-0.5) {};
 \node[point, label=below:$1$] at (-1.5,1.5) {};
 \draw[->, thick, postaction={decorate}] (-0.5,0.5) .. controls (-0.5,0.5) and (0.65,0.7).. (0.65,-0.5); 
 \draw[thick] (0.65,-0.49) .. controls (0.5,-0.7)  .. (0.4,-0.5) ;
 \draw[thick, postaction={decorate}] (0.4,-0.5) .. controls (0.5,0.25) .. (-0.5,0.5);
\end{tikzpicture}
\end{center}
There is another elementary braid $\reflectbox{P}_{i,j}$ conjugated to $P_{i,j}$.  We can represent 
two incarnations of the generator $\reflectbox{P}_{i,j}$ in the following way
\begin{center}
\begin{tikzpicture}[baseline=(current bounding box.center)]
\tikzstyle point=[circle, fill=black, inner sep=0.05cm]
	\node[point, label=above:$1$] at (0,1) {};
	\node[point, label=below:$1$] at (0,-1.5) {};
	\draw[->,thick, postaction={decorate}] (0,1) -- (0,-1.45);
	\node[point, label=above:$i$] at (1,1) {};
	\node[point, label=below:$i$] at (1,-1.5) {};
	\node[point, label=above:$...$] at (2,1) {};
	\node[point, label=below:$...$] at (2,-1.5) {};
	\draw[->,thick, postaction={decorate}] (2,1) -- (2,-1.45);
	\node[point, ,white] at (2,0.4) {};
	\node[point, ,white] at (2,-0.9) {};
	\node[point, label=above:$j$] at (3,1) {};
	\node[point, label=below:$j$] at (3,-1.5) {};
	\draw[thick] (1,1) .. controls (1,0.25) and (3.5,0.5).. (3.5,-0.25);
	\node[point, ,white] at (3,0.2) {};
	\draw[->,thick, postaction={decorate}] (3,1) -- (3,-1.45);
	\draw[->,thick, postaction={decorate}] (3,0) -- (3,-1.45);
	\node[point, ,white] at (3,-0.7) {};
	\draw[->,thick, postaction={decorate}] (3.5,-0.25) .. controls (3.5,-1) and (1,-0.75).. (1,-1.45);
	\node[point, label=above:$n$] at (4,1) {};
	\node[point, label=below:$n$] at (4,-1.5) {};
	\draw[->,thick, postaction={decorate}] (4,1) -- (4,-1.45);
\end{tikzpicture}
$\qquad\longleftrightarrow\quad\sphericalangle\quad$
\begin{tikzpicture}[baseline=(current bounding box.center)] 
\tikzstyle point=[circle, fill=black, inner sep=0.05cm]
	\draw[loosely dotted] (-1.5,1.5) -- (1.5,-1.5); 
	\node[point, label=below:1] at (-1.5,1.5) {}; 
	\node[point, label=above:$i$] at (-0.5,0.5) {};
	\node[point, label=above:$j$] at (0.5,-0.5) {};
	\node[point, label=above:$n$] at (1.5,-1.5) {};
	\draw[thick, postaction={decorate}] (-0.5,0.5) .. controls (0,-0.5) and (-0.25,-0.2).. (0.6,-0.4); 
	\draw[->,thick] (0.6,-0.4) .. controls (0.7,-0.5)  .. (0.6,-0.6) ;
	\draw[thick, postaction={decorate}] (0.6,-0.6) .. controls (0,-0.75) .. (-0.5,0.5);
\end{tikzpicture}
\end{center}

Indeed, one can define $O_{ij}:=P_{ij}\cdots P_{i(i+2)}P_{i(i+1)}$. 
In other words, $P_{ij}=O_{ij}O_{i(j-1)}^{-1}$. 
And we define $\reflectbox{P}_{ij}:=O_{i(j-1)}^{-1}O_{ij}= O_{i(j-1)}^{-1}P_{ij} O_{i(j-1)}$. 


\section{The operad of parenthesized braids}\label{sec-pab}

There are inclusions of topological operads 
$$
\mathbf{Pa}\,\subset\,\overline{\textrm{C}}(\mathbb{R},-)\,\subset\,\overline{\textrm{C}}(\mathbb{C},-)\,.
$$
Then it makes sense to define 
$$
\text{\gls{PaB}}:=\pi_1\left(\overline{\textrm{C}}(\mathbb{C},-),\mathbf{Pa}\right)\,,
$$
which is an operad in groupoids.

\begin{example}[Description of $\mathbf{PaB}(2)$]\label{ex2.1}
Let us first recall that $\mathbf{Pa}(2)=\mathfrak{S}_2$, and that $\overline{\textrm{C}}(\mathbb{C},2)\simeq \mathbb{S}^1$. 
Besides the identity morphism in $\mathbf{PaB}(2)$ going from $(12)$ to $(12)$, there is an arrow 
$R^{1,2}$ in $\mathbf{PaB}(2)$ going from $(12)$ to $(21)$ which can be depicted as 
follows\footnote{We actually have another arrow, that can be obtained from the first one 
as $(R^{2,1})^{-1}$ according to the notation that is explained after Theorem \ref{thm2.3}, 
and which can be depicted as an undercrossing braid. }: 
\begin{center}
\begin{tikzpicture}[baseline=(current bounding box.center)]
\tikzstyle point=[circle, fill=black, inner sep=0.05cm]
 \node[point, label=above:$1$] at (1,1) {};
 \node[point, label=below:$2$] at (1,-0.25) {};
 \node[point, label=above:$2$] at (2,1) {};
 \node[point, label=below:$1$] at (2,-0.25) {};
  \draw[->,thick] (1,1) .. controls (1,0.25) and (2,0.5).. (2,-0.20);
 \node[point, ,white] at (1.5,0.4) {};
 \draw[->,thick] (2,1) .. controls (2,0.25) and (1,0.5).. (1,-0.20); 
\end{tikzpicture}
\qquad\qquad\qquad
\begin{tikzpicture}[baseline=(current bounding box.center)]
\tikzstyle point=[circle, fill=black, inner sep=0.05cm]
 \node[point, label=right:$2$] at (0.5,-0.5) {};
 \node[point, label=left:$1$] at (-0.5,0.5) {};
 \draw[->,thick] (0.5,-0.5) .. controls (-0.25,-0.25) .. (-0.45,0.45); 
 \draw[->,thick] (-0.5,0.5) .. controls (0.25,0.25) .. (0.45,-0.45); 
\end{tikzpicture}
\\ \text{Two incarnations of $R^{1,2}$}
\end{center}
We will denote $\tilde R^{1,2}:=(R^{2,1})^{-1}$.
\end{example}
\begin{example}[Notable arrows in $\mathbf{PaB}(3)$]\label{ex2.1bis}
Let us first recall that $\mathbf{Pa}(3)=\mathfrak{S}_3\times\{(\bullet\bullet)\bullet,\bullet(\bullet\bullet)\}$ 
and that $\overline{\textrm{C}}(\mathbb{R},3)\cong\mathfrak{S}_3\times[0,1]$. 
Therefore,there is an arrow $\Phi^{1,2,3}$ (the identity path in $[0,1]$) from $(12)3$ to $1(23)$ in $\PaB(3)$. 
It can be depicted as follows: 
\begin{center}
\begin{tikzpicture}[baseline=(current bounding box.center)]
\tikzstyle point=[circle, fill=black, inner sep=0.05cm]
 \node[point, label=above:$(1$] at (1,1) {};
 \node[point, label=below:$1$] at (1,-0.25) {};
 \node[point, label=above:$2)$] at (1.5,1) {};
 \node[point, label=below:$(2$] at (3.5,-0.25) {};
 \node[point, label=above:$3$] at (4,1) {};
 \node[point, label=below:$3)$] at (4,-0.25) {};
 \draw[->,thick] (1,1) .. controls (1,0) and (1,0).. (1,-0.20); 
 \draw[->,thick] (1.5,1) .. controls (1.5,0.25) and (3.5,0.5).. (3.5,-0.20);
 \draw[->,thick] (4,1) .. controls (4,0) and (4,0).. (4,-0.20);
\end{tikzpicture}
\qquad\qquad\qquad
\begin{tikzpicture}[baseline=(current bounding box.center)] 
\tikzstyle point=[circle, fill=black, inner sep=0.05cm]
 \node[point, label=below:$1$] at (0.5,0.5) {};
 \node[point, label=below:$2$] at (1,0.5) {};
 \node[point, label=below:$3$] at (2.5,0.5) {};
 \draw[->,thick] (1,0.5) .. controls (1,0.5) .. (2,0.5); 
\end{tikzpicture}
\\ \text{Two incarnations of $\Phi^{1,2,3}$}
\end{center}
\end{example}
The following result is borrowed from \cite[Theorem 6.2.4]{Fresse}, even though it perhaps 
already appeared in \cite{BN} in a different form. 
\begin{theorem}\label{thm2.3}
As an operad in groupoids having $\mathbf{Pa}$ as operad of objects, $\mathbf{PaB}$ is 
freely generated by $R:=R^{1,2}$ and 
$\Phi:=\Phi^{1,2,3}$ together with the following relations: 
\begin{flalign}
& \Phi^{\emptyset,1,2}=\Phi^{1,\emptyset,2}=\Phi^{1,2,\emptyset}=\on{Id}_{1,2}									\quad 
\Big(\mathrm{in}~\mathrm{Hom}_{\mathbf{PaB}(2)}\big(12,12\big)\Big)\,, 				\tag{U1}\label{eqn:U} \\
& R^{1,2}\Phi^{2,1,3}R^{1,3}=\Phi^{1,2,3}R^{1,23}\Phi^{2,3,1} 											\quad 
\Big(\mathrm{in}~\mathrm{Hom}_{\mathbf{PaB}(3)}\big((12)3,2(31)\big)\Big)\,, 				\tag{H1}\label{eqn:H1} \\
& \tilde R^{1,2}\Phi^{2,1,3}\tilde R^{1,3}=\Phi^{1,2,3}\tilde R^{1,23}\Phi^{2,3,1} 											\quad 
\Big(\mathrm{in}~\mathrm{Hom}_{\mathbf{PaB}(3)}\big((12)3,2(31)\big)\Big)\,, 					\tag{H2}\label{eqn:H2} \\
& \Phi^{12,3,4}\Phi^{1,2,34}=\Phi^{1,2,3}\Phi^{1,23,4}\Phi^{2,3,4} 									\quad
\Big(\mathrm{in}~\mathrm{Hom}_{\mathbf{PaB}(4)}\big(((12)3)4,1(2(34))\big)\Big)\,. &	\tag{P}\label{eqn:P}
\end{flalign}
\end{theorem}
We now briefly explain the notation we have been using in the above statement, which is quite standard. 

\begin{notation}
In this article, we write the composition of paths from left to right (and we draw the braids from top to bottom). 
If $X$ is an arrow from $\mathbf{p}$ to $\mathbf{q}$ in $\mathbf{PaB}(n)$, then 
\begin{itemize}
\item for any $\mathbf{r}\in \mathbf{Pa}(k)$, the identity of $\mathbf{r}$ in $\mathbf{PaB}(k)$ is also denoted $\mathbf{r}$,
\item for any $\mathbf{r}\in \mathbf{Pa}(k)$, we write $X^{1,\dots,n}$ for $\mathbf{r}\circ_1X\in \mathbf{PaB}(n+k-1)$,
\item we write $X^{\emptyset,2,\dots,n}\in \mathbf{PaB}(n+k-2)$ for the image of $X^{1,\dots,n}$ by the first braid deleting operation,
\item for any $\sigma\in\mathfrak{S}_{n+k-1}$ we define $X^{\sigma_1,\dots,\sigma_n}:=(X^{1,\dots,n})\cdot\sigma$,
\item for any $\mathbf{r}\in \mathbf{Pa}(k)$, $X^{\mathbf{r},k+1,\dots,k+n-1}:=X\circ_1\mathbf{r}\in \mathbf{PaB}(n+k-1)$,
\item we allow ourselves to combine these in an obvious way. 
\end{itemize}
This notation is unambiguous as soon as we specify the starting object of our arrows. 
\end{notation}
For example, the pentagon \eqref{eqn:P} and the first hexagon \eqref{eqn:H1} relations can be respectively depicted as  
\begin{center}
\begin{align}\tag{\ref{eqn:P}}
\begin{tikzpicture}[baseline=(current bounding box.center)] 
\tikzstyle point=[circle, fill=black, inner sep=0.05cm]
\tikzstyle point2=[circle, fill=black, inner sep=0.08cm]
\node[point, label=above:$((1$] at (1,2) {};
\node[point, label=above:$2)$] at (1.5,2) {};
\node[point, label=above:$3)$] at (2.5,2) {};
\node[point, label=above:$4$] at (5,2) {};
\node[point, label=below:$1$] at (1,0) {};
\node[point, label=below:$(2$] at (3.5,0) {};
\node[point, label=below:$(3$] at (4.5,0) {};
\node[point, label=below:$4))$] at (5,0) {};
 \draw[->,thick] (1,2) .. controls (1,0.5) .. (1,0.05);
 \draw[-,thick] (1.5,2) .. controls (1.5,1.5) .. (1.5,1);
 \draw[-,thick] (2.5,2) .. controls (2.5,1.5) and (4.5,1.5).. (4.5,1);
 \draw[->,thick] (1.5,1) .. controls (1.5,0.5) and (3.5,0.5).. (3.5,0.05);
 \draw[->,thick] (4.5,1) .. controls (4.5,0.5) .. (4.5,0.05);
 \draw[->,thick] (5,2) .. controls (5,0.5) .. (5,0.05);
\end{tikzpicture}
=
\begin{tikzpicture}[baseline=(current bounding box.center)] 
\tikzstyle point=[circle, fill=black, inner sep=0.05cm]
\tikzstyle point2=[circle, fill=black, inner sep=0.08cm]
\node[point, label=above:$((1$] at (1,2) {};
\node[point, label=above:$2)$] at (1.5,2) {};
\node[point, label=above:$3)$] at (2.5,2) {};
\node[point, label=above:$4$] at (5,2) {};
\node[point, label=below:$1$] at (1,0) {};
\node[point, label=below:$(2$] at (3.5,0) {};
\node[point, label=below:$(3$] at (4.5,0) {};
\node[point, label=below:$4))$] at (5,0) {};
 \draw[->,thick] (1,2) .. controls (1,0.5) .. (1,0.05);
 \draw[-,thick] (2.5,2) .. controls (2.5,1.5) .. (2.5,1.3);
 \draw[-,thick] (1.5,2) .. controls (1.5,1.7) and (2,1.7).. (2,1.3);
 \draw[-,thick] (2,1.3) .. controls (2,0.9) and (3.5,0.9).. (3.5,0.65);
 \draw[-,thick] (2.5,1.3) .. controls (2.5,0.9) and (4,0.9).. (4,0.65);
 \draw[->,thick] (4,0.65) .. controls (4,0.4) and (4.5,0.4).. (4.5,0.05);
 \draw[->,thick] (3.5,0.65) .. controls (3.5,0.5) .. (3.5,0.05);
 \draw[->,thick] (5,2) .. controls (5,0.5) .. (5,0.05);
\end{tikzpicture}
\end{align}
\end{center}
and
\begin{center}
\begin{align}\tag{\ref{eqn:H2}}
\begin{tikzpicture}[baseline=(current bounding box.center)]
\tikzstyle point=[circle, fill=black, inner sep=0.05cm]
 \node[point, label=above:$(1$] at (1,2) {};
 \node[point, label=below:$2$] at (1,-1) {};
 \node[point, label=above:$2)$] at (1.5,2) {};
 \node[point, label=below:$(3$] at (3.5,-1) {};
 \node[point, label=above:$3$] at (4,2) {};
 \node[point, label=below:$1)$] at (4,-1) {};
\draw[-,thick] (1.5,2) .. controls (1.5,1.5) and (1,1.5).. (1,1); 
 \node[point, ,white] at (1.25,1.5) {};
 \draw[-,thick] (1,2) .. controls (1,1.5) and (1.5,1.5).. (1.5,1);
 \draw[->,thick] (1,1) .. controls (1,0) and (1,0).. (1,-0.95); 
 \draw[-,thick] (1.5,1) .. controls (1.5,0.25) and (3.5,0.5).. (3.5,0);
 \draw[-,thick] (4,2) .. controls (4,0) and (4,0).. (4,0);
\draw[->,thick] (4,0) .. controls (4,-0.5) and (3.5,-0.5).. (3.5,-0.95); 
 \node[point, ,white] at (3.75,-0.5) {};
 \draw[->,thick] (3.5,0) .. controls (3.5,-0.5) and (4,-0.5).. (4,-0.95);
\end{tikzpicture}
 = 
\begin{tikzpicture}[baseline=(current bounding box.center)] 
\tikzstyle point=[circle, fill=black, inner sep=0.05cm]
\tikzstyle point2=[circle, fill=black, inner sep=0.08cm]
 \node[point, label=above:$(1$] at (1,2) {};
 \node[point, label=below:$2$] at (1,-1) {};
 \node[point, label=above:$2)$] at (1.5,2) {};
 \node[point, label=below:$(3$] at (3.5,-1) {};
 \node[point, label=above:$3$] at (4,2) {};
 \node[point, label=below:$1)$] at (4,-1) {};
 \draw[-,thick] (1,2) .. controls (1,1) and (1,1).. (1,1); 
 \draw[-,thick] (1.5,2) .. controls (1.5,1.25) and (3.5,1.5).. (3.5,1);
 \draw[-,thick] (4,2) .. controls (4,1) and (4,1).. (4,1);
\draw[-,thick] (4,1) .. controls (4,0.5) and (1.5,0.5).. (1.5,0); 
\draw[-,thick] (3.5,1) .. controls (3.5,0.5) and (1,0.5).. (1,0); 
 \node[point2, ,white] at (2.37,0.55) {};
 \node[point2, ,white] at (2.64,0.45) {};
 \draw[-,thick] (1,1) .. controls (1,0.5) and (4,0.5).. (4,0);
 \draw[->,thick] (1,0) .. controls (1,-0.5) and (1,-0.5).. (1,-0.95); 
 \draw[->,thick] (1.5,0) .. controls (1.5,-0.5) and (3.5,-0.5).. (3.5,-0.95);
 \draw[->,thick] (4,0) .. controls (4,-0.5) and (4,-0.5).. (4,-0.95);
\end{tikzpicture}
\end{align}
\end{center}
or, as commuting diagrams (giving the name of the relations)	
$$
\xymatrix@!0 @R=3pc @C=4.5pc{
& (1 2)(34) \ar[dl]_{\Phi^{1,2,34}}   & & & & (12)3 \ar[r]^{\Phi^{1,2,3}} 
\ar[dl]_{(R^{2,1})^{-1}} & 1(23) \ar[dr]^{(R^{23,1})^{-1}} & \\
1(2(34))  & & ((12)3)4 \ar[ul]_{\Phi^{12,3,4}} \ar[d]_{\Phi^{1,2,3}}& 
\text{and} & (21)3 \ar[dr]^{\Phi^{2,1,3}} & & & (23)1 \ar[dl]_{\Phi^{2,3,1}}  \\
1((23)4) \ar[u]_{\Phi^{2,3,4}}  & & (1(23))4 \ar[ll]_{\Phi^{1,23,4}} 
& & & 2(13) \ar[r]^{(R^{3,1})^{-1}} & 2(31)  &
} 
$$


\section{The operad of chord diagrams}\label{sec-cd}

The holonomy Lie algebra of the configuration space 
\[
\text{\gls{ConfCn}}:=\{\zz=(z_1,\dots,z_n)\in \C^n|z_i\neq z_j\textrm{ if }i\neq j\}
\]
of $n$ points on the complex line is isomorphic to the graded Lie $\mathbb{C}$-algebra \gls{tn} 
generated by $t_{ij}$, $1\leq i\neq j\leq n$, with relations
\begin{flalign}
& t_{ij}=t_{ji}\,,                			                 			\tag{S} \label{eqn:S} \\
& [t_{ij},t_{kl}]=0\quad\textrm{if }\#\{i,j,k,l\}=4 \,,				\tag{L} \label{eqn:L} \\
& [t_{ij},t_{ik}+t_{jk}]=0\quad\textrm{if }\#\{i,j,k\}=3\,. &	\tag{4T} \label{eqn:4T} 
\end{flalign}
It is known as the Kohno--Drinfled Lie algebra. 

In \cite{BN,Fresse} it is shown that the collection of Lie $\kk$-algebras $\t_n(\KK)$ 
is provided with the structure of an operad in the category $grLie_\KK$ 
of positively graded finite dimensional Lie algebras over $\kk$, with symmetric monoidal 
strucure given by the direct sum $\oplus$. 
Partial compositions are described as follows: 
$$
\begin{array}{cccc}
\circ_k : & \t_I(\KK) \oplus \t_J(\KK)  & \longrightarrow & \t_{J\sqcup I-\{k\}}(\KK) \\
    & (0,t_{\alpha \beta}) & \longmapsto & t_{\alpha\beta} \\
 & (t_{ij},0) & \longmapsto & 
 \begin{cases}
  \begin{tabular}{ccccc}
  $t_{ij}$ & if & $ k\notin\{i,j\} $ \\
  $\sum\limits_{p\in J} t_{pj}$ & if & $k=i$ \\
  $\sum\limits_{p\in J} t_{ip}$ & if & $k=j$ 
  \end{tabular}
  \end{cases}
\end{array}
$$

\medskip

Observe that there is a lax symmetric monoidal functor 
$$
\hat{\mathcal{U}}:grLie_\kk\longrightarrow \mathbf{Cat(CoAlg_\KK)}
$$
sending a positively graded Lie algebra to the degree completion of its universal envelopping 
algebra, which is a complete filtered cocommutative Hopf algebra, 
viewed as a $\mathbf{CoAlg_\KK}$-enriched category with only one object. 

We then consider the operad of {\it chord diagrams} 
$\CD(\KK) := \hat{\mathcal{U}}(\t({\kk}))$ in $\mathbf{Cat(CoAlg_\KK)}$. 
\begin{remark}
This denomination comes from the fact that morphisms in $\CD(\KK)(n)$ can be represented 
as linear combinations of diagrams of chords on $n$ vertical strands, 
where the chord diagram corresponding to $t_{ij}$ can be represented as 
\begin{align*}
\tik{ \hori[->]{0}{1}{1} \node[point, label=above:$i$] at (0,-1) {}; \node[point, label=above:$j$] at (1,-1) {};
\node[point, label=above:$1$] at (-1,-1) {}; \node[point, label=above:$n$] at (2,-1) {};
\node[point, label=below:$1$] at (-1,-2) {}; \node[point, label=below:$n$] at (2,-2) {};
\node[point, label=below:$i$] at (0,-2) {}; \node[point, label=below:$j$] at (1,-2) {};
 \draw[->,thick] (-1,-1) to (-1,-2);  \draw[->,thick] (2,-1) to (2,-2);  }
\end{align*}
and the composition is given by vertical concatenation of diagrams. Partial compositions can 
easily be understood as ``cabling and removal operations'' on strands (see \cite{BN,Fresse}). 
Relations \eqref{eqn:L} and \eqref{eqn:4T} defining each $\t_n(\KK)$ can be represented as follows:
\begin{align}\tag{\ref{eqn:L}}
\tik{ \hori{0}{0}{1}\straight[->]{0}{1}; \hori[->]{2}{1}{1}\straight[->]{1}{1}\straight{2}{0}
\straight{3}{0}\node[point, label=above:$j$] at (1,0) {}; \node[point, label=above:$k$] at (2,0) {};
\node[point, label=above:$i$] at (0,0) {}; \node[point, label=above:$l$] at (3,0) {};
\node[point, label=below:$i$] at (0,-2) {}; \node[point, label=below:$l$] at (3,-2) {};
\node[point, label=below:$j$] at (1,-2) {}; \node[point, label=below:$k$] at (2,-2) {};}
=\tik{\hori[->]{0}{1}{1}\straight{0}{0}; \hori{2}{0}{1}\straight{1}{0}\straight[->]{2}{1}
\straight[->]{3}{1}\node[point, label=above:$j$] at (1,0) {}; \node[point, label=above:$k$] at (2,0) {};
\node[point, label=above:$i$] at (0,0) {}; \node[point, label=above:$l$] at (3,0) {};
\node[point, label=below:$i$] at (0,-2) {}; \node[point, label=below:$l$] at (3,-2) {};
\node[point, label=below:$j$] at (1,-2) {}; \node[point, label=below:$k$] at (2,-2) {};}
\end{align}

\begin{align}\tag{\ref{eqn:4T}}
\tik{\hori{0}{0}{1} \straight{2}{0}\hori[->]{0}{1}{2}\straight[->]{1}{1}
\node[point, label=above:$i$] at (0,0) {};\node[point, label=above:$j$] at (1,0) {}; 
\node[point, label=above:$k$] at (2,0) {};\node[point, label=below:$i$] at (0,-2) {}; 
\node[point, label=below:$j$] at (1,-2) {}; \node[point, label=below:$k$] at (2,-2) {};}
+\tik{\hori{0}{0}{1} \straight{2}{0}\hori[->]{1}{1}{1}\straight[->]{0}{1}
\node[point, label=above:$i$] at (0,0) {};\node[point, label=above:$j$] at (1,0) {}; 
\node[point, label=above:$k$] at (2,0) {};\node[point, label=below:$i$] at (0,-2) {}; 
\node[point, label=below:$j$] at (1,-2) {}; \node[point, label=below:$k$] at (2,-2) {};}
=\tik{\hori[->]{0}{1}{1}\straight[->]{2}{1}\hori{0}{0}{2}\straight{1}{0}
\node[point, label=above:$i$] at (0,0) {};\node[point, label=above:$j$] at (1,0) {}; 
\node[point, label=above:$k$] at (2,0) {};\node[point, label=below:$i$] at (0,-2) {}; 
\node[point, label=below:$j$] at (1,-2) {}; \node[point, label=below:$k$] at (2,-2) {};}
+\tik{\hori[->]{0}{1}{1} \straight[->]{2}{1}\hori{1}{0}{1}\straight{0}{0}
\node[point, label=above:$i$] at (0,0) {};\node[point, label=above:$j$] at (1,0) {};
 \node[point, label=above:$k$] at (2,0) {};\node[point, label=below:$i$] at (0,-2) {}; 
\node[point, label=below:$j$] at (1,-2) {}; \node[point, label=below:$k$] at (2,-2) {};}
\end{align}
\end{remark}


\section{The operad of parenthesized chord diagrams}\label{sec-pacd}

Recall that the operad $\CD(\KK)$ has only one object in each arity. 
Hence we can define the operad
\[
\text{\gls{PaCD}}:=\omega_1^\star \CD(\kk)
\]
of \textit{parenthesized chord diagrams}, where 
$\omega_1:\Pa=\on{Ob}(\Pa(\KK))\to\on{Ob}(\CD(\kk))$ is the terminal morphism. 
Here is a self-explanatory example of how to depict a morphism in $\mathbf{PaCD}(\kk)(3)$: 
\[
f \cdot \tik{\node[point, label=above:$(i$] at (0,0) {};
\node[point, label=above:$j)$] at (0.5,0) {}; \node[point, label=above:$k$] at (2.5,0) {};
\node[point, label=below:$i$] at (0,-2) {}; 
\node[point, label=below:$(k$] at (2,-2) {}; \node[point, label=below:$j)$] at (2.5,-2) {}; 
\straight{0}{0}\straight[->]{0}{1}  \draw[->,thick] (0.5,0) to (2.5,-2);\draw[->,thick] (2.5,0) to (2,-2);}
\]
where $f\in\CD(\kk)(3)$.
\begin{example}[Notable arrows of $\mathbf{PaCD}(\kk)$]
There are the following arrows in $\mathbf{PaCD}(\kk)(2)$:
\begin{center}
$H^{1,2}:=t_{12} \cdot \on{Id}_{1,2} = t_{12} \cdot $
\begin{tikzpicture}[baseline=(current bounding box.center)]
\tikzstyle point=[circle, fill=black, inner sep=0.05cm]
 \node[point, label=above:$1$] at (1,1) {};
 \node[point, label=below:$1$] at (1,-0.25) {};
 \node[point, label=above:$2$] at (2,1) {};
 \node[point, label=below:$2$] at (2,-0.25) {};
 \draw[->,thick] (2,1) to (2,-0.20); 
 \draw[->,thick] (1,1) to (1,-0.20);
\end{tikzpicture} =:
\begin{tikzpicture}[baseline=(current bounding box.center)]
\tikzstyle point=[circle, fill=black, inner sep=0.05cm]
 \node[point, label=above:$1$] at (1,1) {};
 \node[point, label=below:$1$] at (1,-0.25) {};
 \node[point, label=above:$2$] at (2,1) {};
 \node[point, label=below:$2$] at (2,-0.25) {};
 \draw[->,thick] (2,1) to (2,-0.20); 
  \draw[densely dotted, thick] (1,0.5) to (2,0.5); 
 \draw[->,thick] (1,1) to (1,-0.20);
\end{tikzpicture}
\qquad
$X^{1,2}=1 \cdot$ 
\begin{tikzpicture}[baseline=(current bounding box.center)]
\tikzstyle point=[circle, fill=black, inner sep=0.05cm]
 \node[point, label=above:$1$] at (1,1) {};
 \node[point, label=below:$2$] at (1,-0.25) {};
 \node[point, label=above:$2$] at (2,1) {};
 \node[point, label=below:$1$] at (2,-0.25) {};
 \draw[->,thick] (2,1) to (1,-0.20); 
 \draw[->,thick] (1,1) to (2,-0.20);
\end{tikzpicture}
\end{center} 
We also have the following arrow in $\mathbf{PaCD}(\kk)(3)$:
\begin{center}
$a^{1,2,3}=1 \cdot$
\begin{tikzpicture}[baseline=(current bounding box.center)]
\tikzstyle point=[circle, fill=black, inner sep=0.05cm]
 \node[point, label=above:$(1$] at (1,1) {};
 \node[point, label=below:$1$] at (1,-0.25) {};
 \node[point, label=above:$2)$] at (1.5,1) {};
 \node[point, label=below:$(2$] at (3.5,-0.25) {};
 \node[point, label=above:$3$] at (4,1) {};
 \node[point, label=below:$3)$] at (4,-0.25) {};
 \draw[->,thick] (1,1) to (1,-0.20); 
 \draw[->,thick] (1.5,1) to (3.5,-0.20);
 \draw[->,thick] (4,1) to (4,-0.20);
\end{tikzpicture}
\end{center}
\end{example}
\begin{remark}
The elements $H^{1,2}, X^{1,2}$ and $ a^{1,2,3}$ are generators of the operad ${\PaCD}(\kk)$ and satisfy the 
following relations:
\begin{itemize}
\item $X^{2,1}=(X^{1,2})^{-1}$,
\item $a^{12,3,4}a^{1,2,34} = a^{1,2,3}a^{1,23,4}a^{2,3,4}$,
\item $X^{12,3}=a^{1,2,3}X^{2,3}(a^{1,3,2})^{-1}X^{1,3}a^{3,1,2}$,
\item $H^{1,2}=X^{1,2}H^{2,1}(X^{1,2})^{-1}$,
\item $H^{12,3}=a^{1,2,3}\big(H^{2,3}+X^{2,3}(a^{1,3,2})^{-1}H^{1,3}a^{1,3,2}X^{3,2}\big)(a^{1,2,3})^{-1}$.
\end{itemize}
In particular, even if ${\PaCD}(\kk)$ does not have a presentation in terms of generators and relations 
(as is the case for $\PaB$), one can show that ${\PaCD}(\kk)$ has a universal property with respect to 
the generators $H^{1,2}, X^{1,2}$ and $ a^{1,2,3}$ and the above relations (see \cite[Theorem 10.3.4]{Fresse} for details).
\end{remark}


\section{Drinfeld associators}\label{sec:2.8assoc}

Let us first introduce some terminology that we use in this paragraph, as well as later in the paper: 
\begin{itemize}
\item Let $\mathbf{Grpd}_\kk$ denote the (symmetric monoidal) category of $\kk$-prounipotent 
groupoids (which is the image of the completion functor $\mathcal G\mapsto \hat{\mathcal G}(\kk)$); 
\item For $\mathcal C$ being $\mathbf{Grpd}$, $\mathbf{Grpd}_\kk$, or $\mathbf{Cat(CoAlg_{\kk})}$, the notation 
\[
\on{Aut}_{\tmop{Op}\mathcal C}^+\quad(\mathrm{resp.}~\on{Iso}_{\tmop{Op}\mathcal C}^+)
\] 
refers to those automorphisms (resp.~isomorphisms) which are the identity on objects.
\end{itemize}

\medskip

In the remainder of this section we recall some well known results on the operadic point of view on associators and 
Grothendieck-Teichm\"uller groups, which will be useful later on. 
Even though the statements and proofs of all the results in this section can be found in \cite{Fresse}, it is worth 
mentionning that a ``pre-operadic'' approach was initiated by Bar-Natan in \cite{BN}. 

\medskip

\begin{definition}
A \textit{Drinfeld $\KK$-associator} is an isomorphism between the operads $\widehat{\PaB}(\KK)$ and $G \PaCD(\KK)$ in 
$\mathbf{Grpd}_{\kk}$, which is the identity on objects. We denote by
\[
\text{\gls{Ass}}:=\on{Iso}^+_{\tmop{Op}\mathbf{Grpd}_{\kk}}(\widehat{\PaB}(\KK),G \PaCD(\KK))
\]
the set of $\KK$-associators.
\end{definition}

\begin{theorem}[Drinfeld \cite{DrGal}, Bar-Natan \cite{BN}, Fresse \cite{Fresse}]\label{Th:Ass}
There is a one-to-one correspondence between the set of Drinfeld $\kk$-associators and the set $\text{\gls{bAss}}$ 
of couples $(\mu,\varphi)\in\kk^\times \times \on{exp}(\hat\f_2(\kk))$, such that 
\begin{itemize}
\item $\varphi^{3,2,1}=(\varphi^{1,2,3})^{-1} \quad $ in $\on{exp}(\hat\t_{3}(\kk))$,
\item $\varphi^{1,2,3}e^{\mu t_{23}/2}\varphi^{2,3,1}e^{\mu t_{31}/2}\varphi^{3,1,2}e^{\mu t_{12}/2}
=e^{\mu(t_{12}+t_{13}+t_{23})/2} \quad $ in $\on{exp}(\hat\t_{3}(\kk))$,
\item $\varphi^{1,2,3}\varphi^{1,23,4}\varphi^{2,3,4}=
\varphi^{12,3,4}\varphi^{1,2,34} \quad$ in $\on{exp}(\hat\t_{4}(\kk))$,
\end{itemize}
where $\varphi^{1,2,3}=\varphi(t_{12},t_{23})$ is viewed as an element of $\on{exp}(\hat\t_{3}(\kk))$ via the 
inclusion $\hat\f_2(\kk)\subset \hat\t_3(\kk)$ sending $x$ to $t_{12}$ and $y$ to $t_{23}$. 
\end{theorem}
Three observations are in order: 
\begin{itemize}
\item The free Lie $\kk$-algebra $\f_2(\kk)$ in two generators $x,y$ is graded, with generators having degree $1$, 
and its degree completion is denoted by $\hat{\f}_2(\kk)$; 
\item The $\kk$-prounipotent group $\on{exp}(\hat\f_2(\kk))$ is thus isomorphic to the $\kk$-prounipotent completion 
$\widehat{\on {F}}_2(\KK)$ of the free group ${\on {F}}_2$ on two generators; 
\item The quotient $\hat{\bar{\t}}_{3}(\kk)$of the Lie algebra $\hat\t_{3}(\kk)$ by its center, generated by 
$t_{12}+t_{13}+t_{23}$, is isomorphic to $\hat{\f}_2(\kk)$. Thus, the second relation in the above theorem is equivalent to
\[
\varphi^{1,2,3}e^{\mu y/2}\varphi^{2,3,1}e^{\mu z/2}\varphi^{3,1,2}e^{\mu x/2}=1
\]
in $\on{exp}(\hat\f_2(\kk))$, where $x,y,z$ are variables subject to relation $ x+y+z=0$.
\end{itemize}

This Theorem was originally proven by Drinfeld in \cite{DrGal}, though it was phrased without the operadic language. 
As stated, it can be found in \cite[Theorem 10.2.9]{Fresse}, and its proof relies on the universal property 
of $\PaB$ from Theorem \ref{thm2.3}. 
In particular, a morphism $F:\widehat{\PaB}(\KK) \longrightarrow G \PaCD(\KK)$ is uniquely determined by 
a scalar parameter $\mu\in\kk$ and $\varphi\in\on{exp}(\hat\f_2(\kk))$: 
\begin{itemize}
\item $F(R^{1,2})=e^{\mu t_{12}/2} \cdot X^{1,2}$,
\item $F(\Phi^{1,2,3})=\varphi(t_{12},t_{23}) \cdot a^{1,2,3}$\,,
\end{itemize}
where $R$ and $\Phi$ are the ones from Examples \ref{ex2.1} and \ref{ex2.1bis}.

\begin{example}[The KZ Associator]\label{exampleKZassoc}
The first associator was constructed by Drinfeld using the 
KZ connection and is known as the KZ associator $\Phi_{\on{KZ}}$.
It is defined as the the renormalized holonomy
from $0$ to $1$ of $G'(z) = (\frac{t_{12}}{z} + \frac{t_{23}}{z-1})G(z)$, 
i.e., $\text{\gls{KZAss}} := G_{0^{+}}^{-1}G_{1^{-}}\in\on{exp}(\hat\t_{3}(\C))$, where 
$G_{0^{+}},G_{1^{-}}$ are the solutions such that $G_{0^{+}}(z)\sim z^{t_{12}}$ 
when $z\to 0^{+}$ and $G_{1^{-}}(z)\sim (1-z)^{t_{23}}$ when $z\to 1^{-}$. 
We have
\[
\Phi_{\on{KZ}}(V,U)=\Phi_{\on{KZ}}(U,V)^{-1}, \; 
\Phi_{\on{KZ}}(U,V)e^{\pi\i V}\Phi_{\on{KZ}}(V,W)e^{\pi\i W}\Phi_{\on{KZ}}(W,U)e^{\pi\i U}=1, 
\]
where $U=\overline{t}_{12}\in \f_2(\C)\simeq \bar\t_3(\C) := \t_3(\C)/(t_{12}+t_{13}+t_{23})$, 
$V=\overline{t}_{23}\in \bar\t_3(\C)$ and $U+V+W=0$, and  
\[
\Phi^{12,3,4}_{\on{KZ}}\Phi^{1,2,34}_{\on{KZ}}=\Phi^{1,2,3}_{\on{KZ}} \Phi^{1,23,4}_{\on{KZ}}\Phi^{2,3,4}_{\on{KZ}}\,,
\]
hence $(2\pi\i,\Phi_{\on{KZ}})$ is an element of $\on{Ass}(\C)$.
\end{example}


\section{Grothendieck--Teichmuller group}\label{sec2-gt}

\begin{definition}
The \textit{Grothendieck--Teichm\"uller group} is defined as the group 
\[
\text{\gls{GT}}:= \on{Aut}_{\on{Op} \mathbf{Grpd}}^{+}(\PaB)
\]
of automorphisms of the operad in groupoids $\PaB$ which are the identity of objects. 
One defines similarly its $\kk$-pro-unipotent version 
\[
\widehat{\GT}(\kk):= \on{Aut}_{\on{Op}\mathbf{Grpd}_\kk}^{+}\big(\widehat{\PaB}(\kk)\big)\,.
\]
There are also pro-$\ell$ and profinite versions, denoted $\GT_{\ell}$ and $\widehat \GT$, 
that we do not consider in this paper. 
\end{definition}
We can also characterize elements of $\GT$ and $\widehat\GT(\KK)$ as solutions of certain explicit 
algebraic equations. 
This characterization proves that the above operadic definition of $\GT$ coincides with the one given 
by Drinfeld in his original paper \cite{DrGal}. 
In this article we will focus on the $\kk$-pro-unipotent version of this group in genus $0$ and $1$, 
and twisted situations. 
\begin{definition}\label{Def-GT}
Drinfeld's Grothendieck--Teichm\"uller group $\text{\gls{bGT}}$ consists of pairs 
\[
(\lambda,f) \in \KK^{\times} \times\widehat{\on{F}}_2(\KK)
\]
which satisfy the following equations:
\begin{itemize}
\item $f(x,y)=f(y,x)^{-1}$ in $\widehat{\on{F}}_2(\KK)$, 
\item $x_1^{\nu}f(x_1,x_2)x_2^{\nu}f(x_2,x_3)x_3^{\nu}f(x_3,x_1)=1$ in $\widehat{\on{F}}_2(\KK)$, 
\item $f(x_{13}x_{23}, x_{34})f(x_{12}, x_{23}x_{24}) =f(x_{12}, x_{23}) 
f(x_{12}x_{13}, x_{24}x_{34})f(x_{23}, x_{34})$ in $\widehat{\on{PB}}_4(\KK)$,
\end{itemize}
where $x_1,x_2,x_3$ are 3 variables subject only to $x_1x_2x_3=1$, $\nu=\frac{\lambda-1}{2}$, and $x_{ij}$ 
is the elementary pure braid $\reflectbox{P}_{ij}$ from \S\ref{paragraphe2.2}. 
The multiplication law is given by
\begin{equation*}\label{GT:LCI}
(\lambda_1, f_1)(\lambda_2, f_2)=
(\lambda_1\lambda_2, 
 f_1(x^{\lambda_2},f_2(x, y)y^{\lambda_2} f_2(x, y)^{-1}) f_2(x, y)) .
\end{equation*}
\end{definition}
\begin{theorem}\label{GTGT}
There is an isomorphism between the groups $\widehat{\GT}(\KK)$ and $\widehat{\on{GT}}(\KK)$. 
\end{theorem}
This was first implicitely shown by Drinfeld in \cite{DrGal}. An explicit proof of this theorem 
can be found for example in \cite[Theorem 11.1.7]{Fresse}. In particular, one obtains the 
couple $(\lambda,f)$ from an automorphism $F\in \widehat{\GT}(\KK)$ as follows. We have
\begin{align}\label{FR}
F\left(\begin{tikzpicture}[baseline=(current bounding box.center)]
\tikzstyle point=[circle, fill=black, inner sep=0.05cm]
 \node[point, label=above:$1$] at (1,1) {};
 \node[point, label=below:$2$] at (1,-0.25) {};
 \node[point, label=above:$2$] at (2,1) {};
 \node[point, label=below:$1$] at (2,-0.25) {};
 \draw[->,thick] (1,1) .. controls (1,0.25) and (2,0.5).. (2,-0.20);
 \node[point, ,white] at (1.5,0.4) {};
  \draw[->,thick] (2,1) .. controls (2,0.25) and (1,0.5).. (1,-0.20); 
\end{tikzpicture}\right) 
= 
\left(\begin{tikzpicture}[baseline=(current bounding box.center)]
\tikzstyle point=[circle, fill=black, inner sep=0.05cm]
 \node[point, label=above:$1$] at (1,1) {};
 \node[point, label=below:$1$] at (1,-0.25) {};
 \node[point, label=above:$2$] at (2,1) {};
 \node[point, label=below:$2$] at (2,-0.25) {};
  \draw[->,thick] (1,1) .. controls (2,0.25) and (2,0.5).. (1,-0.20);
  \node[point, ,white] at (1.5,0.66) {};
            \draw[->,thick] (2,1) .. controls (1,0.25) and (1,0.5).. (2,-0.20); 
       \node[point, ,white] at (1.5,0.15) {};
            \draw[thick] (1.59,0.2) .. controls (1.59,0.2) and (1.59,0.2).. (1,-0.20);
\end{tikzpicture}
\right)^{\nu} \cdot 
\begin{tikzpicture}[baseline=(current bounding box.center)]
\tikzstyle point=[circle, fill=black, inner sep=0.05cm]
 \node[point, label=above:$1$] at (1,1) {};
 \node[point, label=below:$2$] at (1,-0.25) {};
 \node[point, label=above:$2$] at (2,1) {};
 \node[point, label=below:$1$] at (2,-0.25) {};
 \draw[->,thick] (1,1) .. controls (1,0.25) and (2,0.5).. (2,-0.20);
 \node[point, ,white] at (1.5,0.4) {};
  \draw[->,thick] (2,1) .. controls (2,0.25) and (1,0.5).. (1,-0.20); 
\end{tikzpicture} 
= 
\left(\begin{tikzpicture}[baseline=(current bounding box.center)]
\tikzstyle point=[circle, fill=black, inner sep=0.05cm]
 \node[point, label=above:$1$] at (1,1) {};
 \node[point, label=below:$2$] at (1,-0.25) {};
 \node[point, label=above:$2$] at (2,1) {};
 \node[point, label=below:$1$] at (2,-0.25) {};
 \draw[->,thick] (1,1) .. controls (1,0.25) and (2,0.5).. (2,-0.20);
 \node[point, ,white] at (1.5,0.4) {};
  \draw[->,thick] (2,1) .. controls (2,0.25) and (1,0.5).. (1,-0.20); 
\end{tikzpicture}\right)^{2\nu+1}
\end{align}
\begin{align}\label{FPhi}
F\left(\begin{tikzpicture}[baseline=(current bounding box.center)]
\tikzstyle point=[circle, fill=black, inner sep=0.05cm]
 \node[point, label=above:$(1$] at (1,1) {};
 \node[point, label=below:$1$] at (1,-0.25) {};
 \node[point, label=above:$2)$] at (1.5,1) {};
 \node[point, label=below:$(2$] at (3.5,-0.25) {};
 \node[point, label=above:$3$] at (4,1) {};
 \node[point, label=below:$3)$] at (4,-0.25) {};
 \draw[->,thick] (1,1) .. controls (1,0) and (1,0).. (1,-0.20); 
 \draw[->,thick] (1.5,1) .. controls (1.5,0.25) and (3.5,0.5).. (3.5,-0.20);
 \draw[->,thick] (4,1) .. controls (4,0) and (4,0).. (4,-0.20);
\end{tikzpicture}\right) 
= 
f\left(\begin{tikzpicture}[baseline=(current bounding box.center)]
\tikzstyle point=[circle, fill=black, inner sep=0.05cm]
 \node[point, label=above:$(1$] at (1,1) {};
 \node[point, label=below:$(1$] at (1,-0.25) {};
 \node[point, label=above:$2)$] at (2,1) {};
 \node[point, label=below:$2)$] at (2,-0.25) {};
   \node[point, label=above:$3$] at (3,1) {};
 \node[point, label=below:$3$] at (3,-0.25) {};
 \draw[->,thick] (3,1) .. controls (3,0) and (3,0).. (3,-0.20);
  \draw[->,thick] (1,1) .. controls (2,0.25) and (2,0.5).. (1,-0.20);
  \node[point, ,white] at (1.5,0.66) {};
            \draw[->,thick] (2,1) .. controls (1,0.25) and (1,0.5).. (2,-0.20); 
       \node[point, ,white] at (1.5,0.15) {};
            \draw[thick] (1.59,0.2) .. controls (1.59,0.2) and (1.59,0.2).. (1,-0.20);
\end{tikzpicture},\begin{tikzpicture}[baseline=(current bounding box.center)]
\tikzstyle point=[circle, fill=black, inner sep=0.05cm]
 \node[point, label=above:$(1$] at (1,1) {};
 \node[point, label=below:$(1$] at (1,-0.25) {};
 \node[point, label=above:$2)$] at (2,1) {};
 \node[point, label=below:$2)$] at (2,-0.25) {};
  \node[point, label=above:$3$] at (3,1) {};
 \node[point, label=below:$3$] at (3,-0.25) {};
 \draw[->,thick] (1+1,1) .. controls (2+1,0.25) and (2+1,0.5).. (1+1,-0.20);
  \node[point, ,white] at (1.5+1,0.66) {};
            \draw[->,thick] (2+1,1) .. controls (1+1,0.25) and (1+1,0.5).. (2+1,-0.20); 
       \node[point, ,white] at (1.5+1,0.15) {};
            \draw[thick] (1.59+1,0.2) .. controls (1.59+1,0.2) and (1.59+1,0.2).. (1+1,-0.20);
    \draw[->,thick] (1,1) .. controls (1,0) and (1,0).. (1,-0.20);
\end{tikzpicture}\right)
\cdot 
\begin{tikzpicture}[baseline=(current bounding box.center)]
\tikzstyle point=[circle, fill=black, inner sep=0.05cm]
 \node[point, label=above:$(1$] at (1,1) {};
 \node[point, label=below:$1$] at (1,-0.25) {};
 \node[point, label=above:$2)$] at (1.5,1) {};
 \node[point, label=below:$(2$] at (3.5,-0.25) {};
 \node[point, label=above:$3$] at (4,1) {};
 \node[point, label=below:$3)$] at (4,-0.25) {};
 \draw[->,thick] (1,1) .. controls (1,0) and (1,0).. (1,-0.20); 
 \draw[->,thick] (1.5,1) .. controls (1.5,0.25) and (3.5,0.5).. (3.5,-0.20);
 \draw[->,thick] (4,1) .. controls (4,0) and (4,0).. (4,-0.20);
\end{tikzpicture}
\end{align}
In other words, if we set $\lambda=2\nu+1$, we get the assignment 
\begin{itemize}
\item $F(R^{1,2})=(R^{1,2})^\lambda $,
\item $F(\Phi^{1,2,3})=  f(x_{12},x_{23}) \cdot \Phi^{1,2,3}$.
\end{itemize}

Next, one obtains the composition law of $\widehat{\on{GT}}(\KK)$ from the composition of automorphisms 
$F_1$ and $F_2$ in $\on{Aut}_{\on{Op} \mathbf{Grpd}_\kk}^+(\widehat{\PaB}(\KK))$
as follows: the associated couples $(\lambda_1,f_1)$ and $(\lambda_2,f_2)$ in $\kk^\times \times \hat{F}_2(\kk)$ satisfy 
$(F_1 F_2)(R^{1,2})=(F_2\circ F_1)(R^{1,2})=(R^{1,2})^{\lambda_1\lambda_2}$, and 
\begin{eqnarray*}
(F_1 F_2)(\Phi^{1,2,3}) & = &(F_2\circ F_1)(\Phi^{1,2,3})= F_2( f_1(x_{12},x_{23})\Phi^{1,2,3} )  \\
& = & F_2( f_1(x_{12},x_{23}))  F_2(\Phi^{1,2,3} )  \\
& = & f_1( F_2(x_{12}), F_2(x_{23}))f_2(x_{12},x_{23})\Phi^{1,2,3} \\
& = & f_1(x_{12}^{\lambda_2},f_2(x_{12},x_{23})x_{23}^{\lambda_2} f_2(x_{12},x_{23})^{-1})f_2(x_{12},x_{23})\Phi^{1,2,3}.
\end{eqnarray*}

\begin{remark}
There are also profinite and pro-${\ell}$ versions of the Grothendieck--Teichm\"uller group, denoted $\widehat{\GT}$ and 
$\GT_{\ell}$, respectively. There are morphisms
\[
\GT \longrightarrow \widehat{\GT} \twoheadrightarrow \GT_{\ell} \hookrightarrow \widehat\GT(\Q_{\ell})\quad\mathrm{and}\quad 
\GT \longrightarrow \widehat\GT(\KK)\,.
\] 
It is important to keep in mind that the profinite, pro-${\ell}$, $\kk$-pro-unipotent versions of the 
Grothendieck--Teichm\"uller group do not coincide with the profinite, pro-${\ell}$, $\kk$-pro-unipotent 
completions of the ``thin'' Grothendieck--Teichm\"uller group $\GT$ which only consists of the pairs $(1,1)$ and $(-1,1)$.
\end{remark}


\section{Graded Grothendieck--Teichmuller group}

\begin{definition}
The graded Grothendieck--Teichm\"uller group is the group 
\[
\text{\gls{GRT}}:=\on{Aut}_{\on{Op} \mathbf{Grpd}_{\kk}}^+(G\PaCD(\KK))
\] 
of automorphisms of $G\PaCD(\KK)$ that are the identity on objects.
\end{definition}
\begin{remark}
When restricted to the full subcategory $\mathbf{Cat(CoAlg^{conn}_{\kk})}$ of $\mathbf{CoAlg_{\kk}}$-enriched categories for 
which the hom-coalgebras are connected, the functor $G$ leads to an equivalence between $\mathbf{Cat(CoAlg^{conn}_{\kk})}$ 
and $\mathbf{Grpd}_\kk$. Hence there is an isomorphism 
\[
\GRT(\kk)\simeq \on{Aut}_{\on{Op} \mathbf{Cat(CoAlg_{\kk})}}^+(\PaCD(\KK))\,.
\]
\end{remark}
Again, the operadic definition of $\GRT(\kk)$ happens to coincide with the one originally given by Drinfeld. 
\begin{definition}
Let $\on{GRT}_1$ be the set of elements in $g\in\on{exp}(\hat\f_2(\kk))\subset \on{exp}(\hat\t_3(\kk))$ such that 
\begin{itemize}
\item $g^{3,2,1}=g^{-1}$ and $g^{1,2,3}g^{2,3,1}g^{3,1,2}=1$, in $\on{exp}(\hat\t_{3}(\kk))$,
\item $g^{1,2,3}g^{1,23,4}g^{2,3,4}=g^{12,3,4}g^{1,2,34}$, in $\on{exp}(\hat\t_{4}(\kk))$,
\end{itemize}
One has the following multiplication law on $\on{GRT}_1$: 
\[
(g_1*g_2)(t_{12},t_{23}):= g_1(t_{12},\on{Ad}(g_2(t_{12},t_{23}))(t_{23}))g_2(t_{12},t_{23})\,.
\]
\end{definition}
Drinfeld showed in \cite{DrGal} that $\on{GRT}_1$ is stable under $*$, which defines a group 
structure on it, and that rescaling transformations $g(x,y)\mapsto \lambda\cdot g(x,y)=g(\lambda x,\lambda y)$ 
define an action of $\kk^\times$ of $\on{\GRT}_1$ by automorphisms. 
\begin{theorem}\label{GRTGRT}
There is a group isomorphism $\GRT ({\kk}) \cong {\kk}^{\times}\rtimes \on{GRT}_1=:\text{\gls{bGRT}}$. 
\end{theorem}

This was first implicitely shown by Drinfeld in \cite{DrGal}. An explicit proof of this theorem 
can be found for example in \cite[Theorem 10.3.10]{Fresse}.
In particular, we obtain the couple $(\lambda,g)$ from an automorphism $G\in\GRT(\kk)$ by the assignment
\begin{itemize}
\item $G(X^{1,2})=X^{1,2}$,
\item $G(H^{1,2})=e^{\lambda t_{12}} H^{1,2}$,
\item $G(a^{1,2,3})= g(t_{12},t_{23}) \cdot a^{1,2,3}$.
\end{itemize}
The composition of automorphisms $G_1$ and $G_2$ in
$\on{Aut}_{\on{Op} \mathcal{\hat{G}}}^+(G\PaCD(\KK))$ is given as follows:  the associated 
couples $(\lambda,g_1)$ and $(\mu,g_2)$ in $\kk^\times \times \on{exp}(\hat{\bar{\t}}_{3}(\kk))$ satisfy 
\[
(G_1G_2)(H^{1,2})=(G_2 \circ G_1)(H^{1,2})=\lambda\mu H^{1,2}\,,
\]
\[
(G_1  G_2)(a^{1,2,3})=(G_2 \circ G_1)(a^{1,2,3})=
g_1\big(\mu t_{12}, g_2(t_{12},t_{23})(\mu t_{23})g_2(t_{12},t_{23})^{-1}\big) g_2(t_{12},t_{23})\cdot a^{1,2,3}\,.
\]


\section{Bitorsors}\label{torsors}

Recall first that there is a free and transitive left action of $\widehat{\on{GT}}(\kk)$
on $\on{Ass}(\kk)$, defined, for $(\lambda,f)\in\widehat{\on{GT}}(\kk)$ and $(\mu,\varphi)\in \on{Ass}(\kk)$, by
\[
((\lambda,f)*(\mu,\varphi))(t_{12},t_{23}):= (\lambda\mu,f(e^{\mu t_{12}},\on{Ad}(\varphi(t_{12},t_{23}))
(e^{\mu t_{23}}))\varphi(t_{12},t_{23})),
\]
where $\on{Ad}(f)(g):=fgf^{-1}$, for any symbols $f,g$.

Recall that there is also a free and transitive right action of $\on{GRT}(\kk)$ on $\on{Ass}(\kk)$ defined 
as follows: for $(\lambda,g)\in \on{GRT}(\kk)$ and $(\mu,\varphi)\in \on{Ass}(\kk)$, given by
\[
((\mu,\varphi)*(\lambda,g))(t_{12},t_{23}):= (\lambda\mu,\varphi(\lambda t_{12},\on{Ad}(g)(\lambda t_{23}))g(t_{12},t_{23}) ).
\]
The two actions commute making $\big(\widehat{\on{GT}}(\kk),\on{Ass}(\kk),\on{GRT}(\kk)\big)$ into a bitorsor.

\begin{theorem}
There is a torsor isomorphism
\begin{equation}\label{bitorsor:cl}
(\widehat{\GT}(\kk),\Assoc(\kk),\GRT(\kk)) \longrightarrow (\widehat{\on{GT}}(\kk),\on{Ass}(\kk),\on{GRT}(\kk))
\end{equation}
\end{theorem}
\begin{proof}
On the one hand, in \cite[Theorem 10.3.13]{Fresse} it is shown that the natural free and transitive left action 
of $\widehat{\GT}(\kk)$ on $\Assoc(\kk)$ coincides with the action of $\on{GT}(\kk)$ on $\on{Ass}(\kk)$ via 
the correspondence of Theorem \ref{GTGT}. On the other hand, in \cite[Theorem 11.2.1]{Fresse}, it is shown that 
the natural free and transitive right action of $\GRT(\kk)$ on $\Assoc(\kk)$ coincides with the action of $\on{GRT}(\kk)$ 
over $\on{Ass}(\kk)$ via the correspondence of Theorem \ref{GRTGRT}. 
\end{proof}

\chapter{Modules associated with configuration spaces (elliptic associators)}
\label{Section3}


\section{Compactified configuration space of the torus}

Let $\mathbb{T}$ be the topological ($2$-)torus. To any finite set $I$ we associate a configuration space 
\[
\textrm{Conf}(\mathbb{T},I)=\{\mathbf{z}=(z_i)_{i\in I}\in \mathbb{T}^I|z_i\neq z_j\textrm{ if }i\neq j\}\,.
\]
We also consider its reduced version 
\[
\text{\gls{CTn}}:=\textrm{Conf}(\mathbb{T},I)/\mathbb{T}\,.
\]
We then consider the Axelrod--Singer--Fulton--MacPherson compactification \gls{FMTn} of $\textrm{C}(\mathbb{T},I)$. 
The boundary $\partial\overline{\textrm{C}}(\mathbb{T},I)=\overline{\textrm{C}}(\mathbb{T},I)-\textrm{C}(\mathbb{T},I)$ 
is made up of the following irreducible components: for any partition 
$I=J_1\coprod\cdots\coprod J_k$ there is a component 
\[
\partial_{J_1,\cdots,J_k}\overline{\textrm{C}}(\mathbb{T},I)
\cong\overline{\rm C}(\mathbb{T},k) \times \prod_{i=1}^k\overline{\rm C}(\mathbb{C},J_i)\,.
\]
The inclusion of boundary components provide $\overline{\rm C}(\mathbb{T},-)$ with the structure of a module over the 
operad $\overline{\rm C}(\mathbb{C},-)$ in topological spaces. 

\begin{center}
\includegraphics[scale=1]{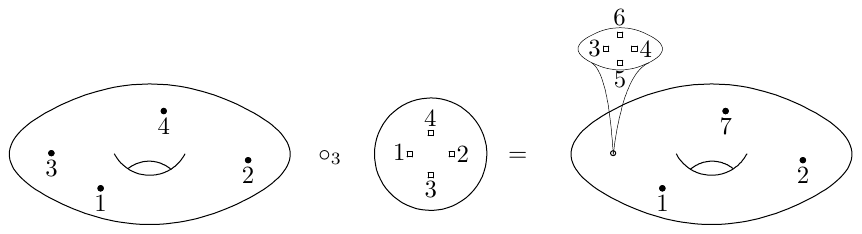}
\end{center}


\section{The pure braid group on the torus}

The reduced pure braid group \gls{PB1n} with $n$ strands on the torus (that 
is the fundamental group of $\textrm{C}(\mathbb{T},n)$) is generated by paths $X_i$'s and $Y_i$'s 
($i=1,\dots,n$), which can be represented as follows
\begin{center}
\includegraphics[scale=1]{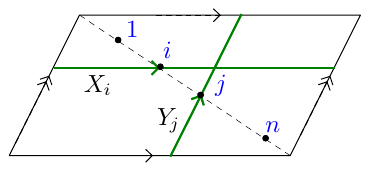}
\end{center}
Moreover, the following relations are satisfied in $\overline{\on{PB}}_{1,n}$: 
\begin{flalign}
& (X_i,X_j)=1=(Y_i,Y_j)\,,\quad\textrm{for }i< j\,, 							\tag{T1}\label{eqn:T1} \\
& (X_i,Y_j)=P_{ij}\textrm{ and }(X_j^{-1},Y_i^{-1})=\reflectbox{P}_{ij}\,,\quad\textrm{for }i< j\,,\tag{T2}\label{eqn:T2} \\ 
& (X_1,Y_1^{-1})=P_{1n}\cdots P_{12} \,,							\tag{T3}\label{eqn:T3} \\
& (X_i,P_{jk})=1=(Y_i,P_{jk})\,,\quad\textrm{for all }i, j< k\,,		\tag{T4}\label{eqn:T4} \\ 
& (X_iX_j,P_{ij})=1=(Y_iY_j,P_{ij})\,,\quad\textrm{for }i< j\,, 		\tag{T5}\label{eqn:T5} \\
& X_1\cdots X_n=1=Y_1\cdots Y_n \,, 	&		\tag{TR}\label{eqn:TR} 
\end{flalign}

There are also the following relations, satisfied in the fundamental group $\overline{\on{B}}_{1,n}$ of 
$\textrm{C}(\mathbb{T},n)/\mathfrak{S}_n$:
\begin{equation}\tag{N}\label{eqn:N}
X_{i+1}=\sigma^{-1}_i X_i \sigma^{-1}_i, \quad Y_{i+1}=\sigma^{-1}_i Y_i \sigma^{-1}_i\,,  
\end{equation}
where $\sigma_i$ are the generators of the braid group $\on{B}_n$ with geometric convention as follows:
\begin{center}
\includegraphics[scale=1]{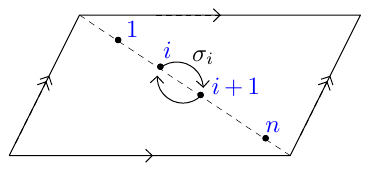}
\end{center}


\section{The $\mathbf{PaB}$-module $\mathbf{PaB}_{e\ell\ell}$ of parenthesized elliptic (or beak) braids}\label{sec-pabell}

In a similar manner as in \S\ref{sec-pab}, there are inclusions of topological 
modules\footnote{The second one depends on the choice 
of an embedding $\mathbb{S}^1\hookrightarrow\mathbb{T}$: we choose by 
convention the ``horizontal embedding'', which corresponds to 
$\mathbb{S}^1\times\{*\}$. }
\[
\mathbf{Pa}\,\subset\,\overline{\textrm{C}}(\mathbb{S}^1,-)\,\subset\,\overline{\textrm{C}}(\mathbb{T},-)\,.
\]
Then it makes sense to define 
\[
\text{\gls{PaBe}}:=\pi_1\left(\overline{\textrm{C}}(\mathbb{T},-),\mathbf{Pa}\right)\,,
\]
which is a $\mathbf{PaB}$-module in groupoids. 

As said in section \ref{sec-pointings}, there is a map of $\mathfrak{S}$-modules $\PaB \longrightarrow \PaB_{e\ell\ell}$ 
and we abusively denote $R^{1,2}$ and $\Phi^{1,2,3}$ the images in $\PaB_{e\ell\ell}$ of the corresponding arrows in $\PaB$. 

\begin{example}[Structure of $\mathbf{PaB}_{e\ell\ell}(2)$]\label{ex2.2}
As in Example \ref{ex2.1} there is an arrow $R^{1,2}$ going from $(12)$ to $(21)$. 
Additionnally, we also have two automorphisms of $(12)$, denoted $A^{1,2}$ and $B^{1,2}$, corresponding to the following 
loops on $\overline{\on{C}}(\mathbb{T},2)$:

\[
\includegraphics[width=90mm]{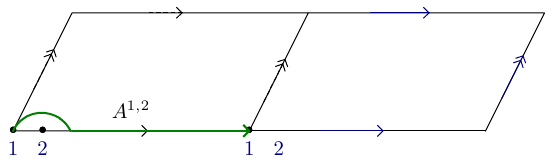} 
\qquad \includegraphics[width=50mm]{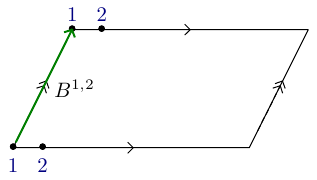}
\]
By global translation of the torus, these are the same loops as the following
\[
\includegraphics[width=90mm]{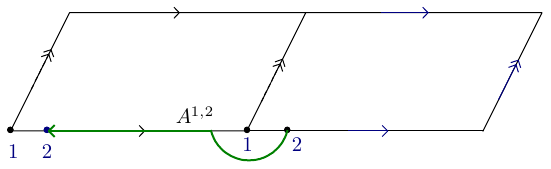} 
\qquad \includegraphics[width=53mm]{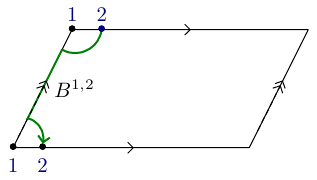}
\]
In particular, $A^{1,2}\tilde R^{1,2}$ and $B^{1,2}\tilde R^{1,2}$, which are morphisms from $(12)$ to $(21)$, 
correspond to the following paths $\overline{\on{C}}(\mathbb{T},2)$:
\[
\includegraphics[width=90mm]{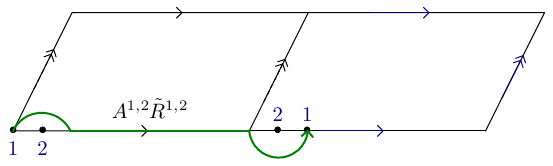} 
\qquad \includegraphics[width=50mm]{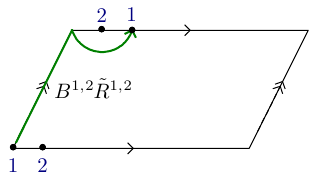}
\]
\end{example}
\begin{remark}
The arrows $A^{1,2}$ and $B^{1,2}$ correspond to $A^{\pm}_{1,2}$ in \cite[\S1.3]{En2}. 
Thus we will respectively depict $A^{1,2}$ and $B^{1,2}$ as 
\begin{center}
\begin{tikzpicture}[baseline=(current bounding box.center)]
\tikzstyle point=[circle, fill=black, inner sep=0.05cm]
 \node[point, label=above:$1$] at (1,1) {};
 \node[point, label=below:$1$] at (1,0) {};
 \node[point, label=above:$2$] at (2,1) {};
 \node[point, label=below:$2$] at (2,0) {};
 \draw[-,thick] (1,1) .. controls (1,0.5) and (1,0.5).. (1.5,0.5); 
 \draw[->,thick] (1.5,0.5) .. controls (1,0.5) and (1,0.5).. (1,0.05); 
\node[point, white, label=left:$+$] at (1,0.5) {};
 \draw[->,thick] (2,1) .. controls (2,0.5) and (2,0.5).. (2,0.05);
\end{tikzpicture}
\quad and \quad
\begin{tikzpicture}[baseline=(current bounding box.center)]
\tikzstyle point=[circle, fill=black, inner sep=0.05cm]
 \node[point, label=above:$1$] at (1,1) {};
 \node[point, label=below:$1$] at (1,0) {};
 \node[point, label=above:$2$] at (2,1) {};
 \node[point, label=below:$2$] at (2,0) {};
 \draw[-,thick] (1,1) .. controls (1,0.5) and (1,0.5).. (1.5,0.5); 
 \draw[->,thick] (1.5,0.5) .. controls (1,0.5) and (1,0.5).. (1,0.05); 
\node[point, white, label=left:$-$] at (1,0.5) {};
 \draw[->,thick] (2,1) .. controls (2,0.5) and (2,0.5).. (2,0.05);
\end{tikzpicture}
\end{center}
\end{remark}

The images of $R^{1,2}$ and $\Phi^{1,2,3}$ by the $\mathfrak{S}$--module morphism 
$\PaB \longrightarrow \PaB_{e\ell\ell}$ will still be denoted the same way. 
One can rephrase \cite[Proposition 1.3]{En2} in the following way: 
\begin{theorem}\label{PaBell}
As a $\mathbf{PaB}$-module in groupoids having $\mathbf{Pa}$ as $\mathbf{Pa}$-module 
of objects, $\mathbf{PaB}_{e\ell\ell}$ is freely 
generated by $A:=A^{1,2}$ and $B:=B^{1,2}$, together with the following relations: 
\begin{flalign}
& \Phi^{1,2,3}A^{1,23}\tilde R^{1,23}\Phi^{2,3,1}A^{2,31}\tilde R^{2,31}\Phi^{3,1,2}A^{3,12}\tilde R^{3,12}=\on{Id}_{(12)3}\,, 
\tag{N1}\label{eqn:N1} \\
& \Phi^{1,2,3}B^{1,23}\tilde R^{1,23}\Phi^{2,3,1}B^{2,31}\tilde R^{2,31}\Phi^{3,1,2}B^{3,12}\tilde R^{3,12}=\on{Id}_{(12)3}\,,
\tag{N2}\label{eqn:N2} \\ 
& R^{1,2} R^{2,1}=\left(\Phi^{1,2,3}A^{1,23}(\Phi^{1,2,3})^{-1}, \tilde R^{1,2}\Phi^{2,1,3}B^{2,13}(\Phi^{2,1,3})^{-1}\tilde R^{2,1}\right)\,. &
\tag{E}\label{eqn:E} 
\end{flalign}
All these relations hold in the automorphism group of $(12)3$ in $\PaB_{e\ell\ell}(3)$.
\end{theorem}
\begin{proof}
Let $\mathcal Q_{e\ell\ell}$ be the $\PaB$-module with the above presentation. 
We first show that there is a morphism of $\PaB$-modules $\mathcal Q_{e\ell\ell}\to\PaB_{e\ell\ell}$. 
We have already seen that there are two automorphisms $A,B$ of $(12)$ in $\mathbf{PaB}_{e\ell\ell}(2)$ 
(see Example \ref{ex2.2}). We have to prove that they indeed 
satisfy the relations \eqref{eqn:N1}, \eqref{eqn:N2} and \eqref{eqn:E}. 

\medskip

\noindent\underline{Relations \eqref{eqn:N1} and \eqref{eqn:N2} are satistfied:} the two \textit{nonagon relations} 
\eqref{eqn:N1} and \eqref{eqn:N2} can be depicted as 
\begin{center}
\begin{align}\tag{\ref{eqn:N1},\ref{eqn:N2}}
\begin{tikzpicture}[baseline=(current bounding box.center)]
\tikzstyle point=[circle, fill=black, inner sep=0.05cm] 
\node[point, label=above:$(1$] at (1,1) {};
 \node[point, label=below:$(1$] at (1,0) {};
 \node[point, label=above:$2)$] at (1.5,1) {};
 \node[point, label=below:$2)$] at (1.5,0) {};
 \node[point, label=above:$3$] at (3,1) {};
 \node[point, label=below:$3$] at (3,0) {};
  \draw[->,thick] (1.5,1) .. controls (1.5,0.5) and (1.5,0.05).. (1.5,0.05);
 \draw[->,thick] (1,1) .. controls (1,0.5) and (1,0.5).. (1,0.05);
 \draw[->,thick] (3,1) .. controls (3,0.5) and (3,0.5).. (3,0.05);
\end{tikzpicture} 
\qquad = \qquad
\begin{tikzpicture}[baseline=(current bounding box.center)] 
\tikzstyle point=[circle, fill=black, inner sep=0.05cm]
 \node[point, label=above:$(1$] at (1,4) {};
 \node[point, label=below:$(1$] at (1,-2) {};
 \node[point, label=above:$2)$] at (1.5,4) {};
 \node[point, label=below:$2)$] at (1.5,-2) {};
 \node[point, label=above:$3$] at (3,4) {};
 \node[point, label=below:$3$] at (3,-2) {};
 \draw[-,thick] (1,4) .. controls (1,3.75) and (1,3.75).. (1,3.5);
 \draw[-,thick] (1.5,4) .. controls (1.5,3.75) and (2.5,3.75).. (2.5,3.5);
 \draw[-,thick] (1,3.5) .. controls (1,3.25) and (1,3.25).. (1.5,3.25); 
\node[point, white, label=left:$\pm$] at (1,3.25) {};
 \draw[-,thick] (1.5,3.25) .. controls (1,3.25) and (1,3.25).. (1,3); 
 \draw[-,thick] (3,4) .. controls (3,3.75) and (3,3.75).. (3,3);
 \draw[-,thick] (3,3) .. controls (3,3.75) and (3,3.75).. (3,3);
 \draw[-,thick] (2.5,3.5) .. controls (2.5,3.25) and (2.5,3.25).. (2.5,3); 
\draw[-,thick] (2.5,3) .. controls (2.5,2.5) and (1,2.5).. (1,2);
\draw[-,thick] (3,3) .. controls (3,2.5) and (1.5,2.5).. (1.5,2);
\node[point, ,white] at (2.15,2.45) {};
\node[point, ,white] at (1.85,2.55) {};
\draw[-,thick] (1,3) .. controls (1.1,2.5) and (2.9,2.5).. (3,2);
\draw[-,thick] (1,2) .. controls (1,1.75) and (1,1.75).. (1,1.5);
\draw[-,thick] (1.5,2) .. controls (1.5,1.75) and (2.5,1.75).. (2.5,1.5);
 \draw[-,thick] (1,1.5) .. controls (1,1.25) and (1,1.25).. (1.5,1.25); 
\node[point, white, label=left:$\pm$] at (1,1.25) {};
 \draw[-,thick] (1.5,1.25) .. controls (1,1.25) and (1,1.25).. (1,1); 
 \draw[-,thick] (3,2) .. controls (3,1.75) and (3,1.75).. (3,1);
 \draw[-,thick] (3,1) .. controls (3,1.75) and (3,1.75).. (3,1);
 \draw[-,thick] (2.5,1.5) .. controls (2.5,1.25) and (2.5,1.25).. (2.5,1); 
\draw[-,thick] (2.5,1) .. controls (2.5,0.5) and (1,0.5).. (1,0);
\draw[-,thick] (3,1) .. controls (3,0.5) and (1.5,0.5).. (1.5,0);
\node[point, ,white] at (2.15,0.45) {};
\node[point, ,white] at (1.85,0.55) {};
\draw[-,thick] (1,1) .. controls (1.1,0.5) and (2.9,0.5).. (3,0);
\draw[-,thick] (1,0) .. controls (1,-0.25) and (1,-0.25).. (1,-0.5);
\draw[-,thick] (1.5,0) .. controls (1.5,-0.25) and (2.5,-0.25).. (2.5,-0.5);
 \draw[-,thick] (1,-0.5) .. controls (1,-0.75) and (1,-0.75).. (1.5,-0.75); 
\node[point, white, label=left:$\pm$] at (1,-0.75) {};
 \draw[-,thick] (1.5,-0.75) .. controls (1,-0.75) and (1,-0.75).. (1,-1); 
 \draw[-,thick] (3,0) .. controls (3,-0.25) and (3,-0.25).. (3,-1);
 \draw[-,thick] (3,-1) .. controls (3,-0.25) and (3,-0.25).. (3,-1);
 \draw[-,thick] (2.5,-0.5) .. controls (2.5,-0.75) and (2.5,-0.75).. (2.5,-1); 
\draw[-,thick] (2.5,-1) .. controls (2.5,-1.5) and (1,-1.5).. (1,-2);
\draw[-,thick] (3,-1) .. controls (3,-1.5) and (1.5,-1.5).. (1.5,-2);
\node[point, ,white] at (2.15,-1.55) {};
\node[point, ,white] at (1.85,-1.45) {};
\draw[-,thick] (1,-1) .. controls (1.1,-1.5) and (2.9,-1.5).. (3,-2);
\end{tikzpicture} \end{align}
\end{center}
It is satisfied in $\PaB_{e\ell\ell}$, expressing the fact that when all (here, three) 
points move in the same direction on the torus, this corresponds to a constant 
path in the reduced configuration space of points on the torus. 
The same is true with the second nonagon relation \eqref{eqn:N2}. 

\medskip

\noindent\underline{Relation \eqref{eqn:E} is satisfied:} below one sees the path that is 
obtained from the right-hand-side of the \textit{mixed relation} \eqref{eqn:E}:
\begin{itemize}
\item $\Phi^{1,2,3}A^{1,23}(\Phi^{1,2,3})^{-1}$ is the path 
$$
\includegraphics[scale=1]{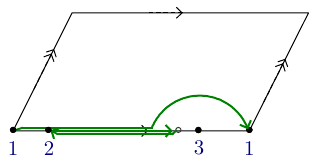}
$$
\item $\tilde R^{1,2}\Phi^{2,1,3}B^{2,13}(\Phi^{2,1,3})^{-1}\tilde R^{2,1}$ is the path
$$
\includegraphics[scale=1]{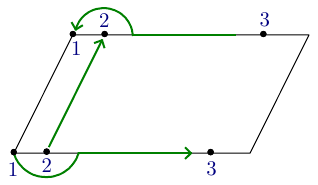}
$$
\end{itemize}

One can see that the commutator of these loops is homotopic to the pure braiding 
of the first two points in the clockwise direction, that is $R^{1,2} R^{2,1}$, by means of the following picture: 
$$
\includegraphics[scale=1]{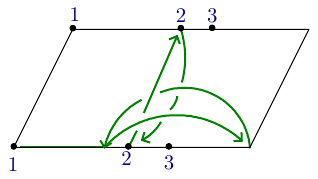}
$$

\medskip

Thus, by the universal property of $\mathcal Q_{e\ell\ell}$, there is a morphism of 
$\PaB$-modules $\mathcal Q_{e\ell\ell} \to \mathbf{PaB}_{e\ell\ell}$, which is the identity on objects.
To show that this map is in fact an isomorphism, it suffices to show that it is an 
isomorphism at the level of automorphism groups of objects arity-wise, as all groupoids are connected. 
Let $n\geq 0$, and $p$ be the object $(\cdots((12)3)\cdots\cdots)n$ of 
$\mathcal Q_{e\ell\ell}(n)$ and $\mathbf{PaB}_{e\ell\ell}(n)$. We want to show that the induced morphism 
\[
\on{Aut}_{\mathcal Q_{e\ell\ell}(n)}(p)  \longrightarrow 
\on{Aut}_{\mathbf{PaB}_{e\ell\ell}(n)}(p)=\pi_1\left(\overline{\textrm{C}}(\mathbb{T},n),p\right)
\]
is an isomorphism. 

On the one hand, as $\bar{\textrm{C}}(\mathbb{T},n)$ is a manifold with corners, we are allowed to move the basepoint $p$ 
to a point $p_{reg}$ which is included in the simply connected subset obtained as the image of\footnote{We have already 
done so for the proof of relation \eqref{eqn:E}. } 
\[
D_{n,\tau} \assign \left\{\mathbf{z} \in \mathbb{C}^n\,\big|\,
z_j = a_j + b_j \tau, a_j, b_j \in \mathbb{R}, 0 < a_1 < a_2 < ... <a_n < a_1 + 1, 0 < b_n  < ... < b_1 < b_n + 1 \right\}
\]
in $\textrm{C}(\mathbb{T},n)$, where $\mathbb{T}=\mathbb{C}/(\mathbb{Z}+\tau\mathbb{Z})$. 
We then have an isomorphism of fundamental groups $\pi_1(\bar{\textrm{C}}(\mathbb{T},n),p) \simeq \pi_1(\textrm{C}(\mathbb{T},n),p_{reg})$. 

\medskip

On the other hand, in \cite[Proposition 1.4]{En2}, Enriquez constructs a universal elliptic 
structure $\PaB_{e\ell\ell}^{En}$, that by definition carries an action of the (algebraic version of the) 
reduced braid group on the torus $\overline{\on{B}}_{1,n}$ in the following sense: 
\begin{itemize}
\item $\PaB_{e\ell\ell}^{En}$ is a category; 
\item for every object $p$ of $\mathbf{Pa}(n)$, there is a corresponding object $[p]$ 
in $\PaB_{e\ell\ell}^{En}$, and $[p]=[q]$ if $p$ and $q$ only differ by a permutation (but have the same underlying parenthesization); 
\item there are group morphisms $\overline{\on{B}}_{1,n}\tilde{\to} \on{Aut}_{\PaB_{e\ell\ell}^{En}}(p)\to\mathfrak S_n$.
\end{itemize}
Moreover, one has by constuction of $\PaB_{e\ell\ell}^{En}$ that 
$\on{Aut}_{\mathcal Q_{e\ell\ell}(n)}(p)$ is the kernel of the map $\on{Aut}_{\PaB^{En}_{e\ell\ell}}([p]) \to \mathfrak{S}_n$. 
One can actually show that there is a commuting diagram 
\[
\xymatrix{
 \on{\overline{PB}}_{1,n} \ar[r]^-{\simeq} \ar[d] &  \on{Aut}_{\mathcal Q_{e\ell\ell}(n)}(p) \ar[r] \ar[d] & 
\pi_1\left(\overline{\textrm{C}}(\mathbb{T},n),p\right) \ar[d] & \ar[l]_-{\simeq}
 \pi_1\left(\textrm{C}(\mathbb{T},n),p_{reg}\right)\ar[d] \\ 
 \on{\overline{B}}_{1,n} \ar[r]^-{\simeq} \ar[d] &  \on{Aut}_{\PaB^{En}_{e\ell\ell}}(p) \ar[r] \ar[d] & 
\pi_1\left(\overline{\textrm{C}}(\mathbb{T},n)/\mathfrak S_n,[p] \right)  \ar[d] & \ar[l]_-{\simeq}
\pi_1\left(\textrm{C}(\mathbb{T},n)/\mathfrak S_n,[p_{reg}]\right) \ar[d] \\
\mathfrak{S}_n \ar@{=}[r]& \mathfrak{S}_n\ar@{=}[r] & \mathfrak{S}_n  \ar@{=}[r]  & \mathfrak{S}_n 
}
\]
where all vertical sequences are short exact sequences. 
Thus, in order to show that the map $\on{Aut}_{\mathcal Q_{e\ell\ell}(n)}(p) \to 
\pi_1\left(\overline{\textrm{C}}(\mathbb{T},n),p\right)$ is an isomorphism, we are left to show that 
\[
\overline{ \on{B}}_{1,n} \longrightarrow \pi_1(\textrm{C}(\mathbb{T},n),p_{reg})
\]
is indeed an isomorphism. But this map is nothing else than a conjugate of the map 
constructed in \cite[Theorem 5]{Bir}, identifying the algebraic and topological versions of the braid group on the torus. 
\end{proof}


\section{The $\CD(\kk)$-module of elliptic chord diagrams}\label{sec-cdell}

For any $n\geq0$, recall that $\text{\gls{t1n}}$ is defined as the bigraded Lie $\KK$-algebra freely generated by 
$x_1,\dots,x_n$ in degree $(1,0)$, $y_1,\dots,y_n$ in degree $(0,1)$ (for $i=1,...,n$), and $t_{ij}$ in degree $(1,1)$ 
(for $1\leq i\neq j\leq n$), together with the relations \eqref{eqn:S}, \eqref{eqn:L}, \eqref{eqn:4T}, and the 
following additional elliptic relations as well:
\begin{flalign}
& [x_i,y_j]=t_{ij}\quad\textrm{for }i\neq j\,, 														\tag{S$_{e\ell\ell}$}\label{eqn:Sell} \\
& [x_i,x_j] = [y_i,y_j] = 0\quad\textrm{for }i\neq j\,, 									\tag{N$_{e\ell\ell}$}\label{eqn:Nell} \\ 
& [x_i,y_i] = -\sum_{j|j\neq i} t_{ij}\,,																	\tag{T$_{e\ell\ell}$}\label{eqn:Tell} \\
& [x_i,t_{jk}] = [y_i,t_{jk}] = 0\quad\textrm{if }\#\{i,j,k\}=3\,,				\tag{L$_{e\ell\ell}$}\label{eqn:Lell} \\ 
& [x_i+x_j,t_{ij}] = [y_i+y_j,t_{ij}]=0\quad\textrm{for }i\neq j\,.  	&		\tag{4T$_{e\ell\ell}$}\label{eqn:4Tell}
\end{flalign}

The $\sum_i x_i$ and $\sum_i y_i$ are central in $\t_{1,n}({\KK})$, and we also consider the quotient 
\[
\bar\t_{1,n}({\KK}):=\t_{1,n}({\KK})/(\sum_i x_i,\sum_i y_i)\,.
\] 
\begin{example}
$\bar\t_{1,2}(\KK)$ is equal to the free Lie $\kk$-algebra $\f_2(\kk)$ on two generators $x=x_1$ and $y=y_2$. 
\end{example}
Both $\t_{1,n}$ and $\bar\t_{1,n}$ are acted on by the symmetric group $\mathfrak{S}_n$, and one can show that the 
$\mathfrak{S}$-modules in $grLie_\kk$ 
\[
\t_{e\ell\ell}({\KK}):=\{\t_{1,n}(\KK) \}_{n \geq 0}\quad\mathrm{and}\quad\bar\t_{e\ell\ell}({\KK}):=\{\bar\t_{1,n}(\KK) \}_{n \geq 0}
\]
actually are $\t(\KK)$-modules in $grLie_{\KK}$. Partial compositions are defined as follows: for $I$ a finite set and $i\in I$, 
\[
\begin{array}{ccccccc}
\circ_k : & \t_{1,I}(\KK) \oplus \t_J(\KK)  & \longrightarrow & \t_{1,J\sqcup I-\{i\}}(\KK) \\
    & (0,t_{\alpha \beta}) & \longmapsto & t_{\alpha\beta} \\
 & (t_{ij},0) & \longmapsto & 
 \begin{cases}
  \begin{tabular}{ccccc}
  $t_{ij}$ & if & $ k\notin\{i,j\} $ \\
  $\sum\limits_{p\in J} t_{pj}$ & if & $k=i$ \\
  $\sum\limits_{p\in J} t_{ip}$ & if & $j=k$ 
  \end{tabular}
  \end{cases}\\
   & (x_i,0) & \longmapsto & 
 \begin{cases}
  \begin{tabular}{ccccc}
  $x_i$ & if & $ k \neq i $ \\
  $\sum\limits_{p\in J} x_{p}$ & if & $k=i$ 
  \end{tabular}
  \end{cases}\\
     & (y_i,0) & \longmapsto & 
 \begin{cases}
  \begin{tabular}{ccccc}
  $y_i$ & if & $ k \neq i $ \\
  $\sum\limits_{p\in J} y_{p}$ & if & $k=i$ 
  \end{tabular}
  \end{cases}
\end{array}
\]
We call $\t_{e\ell\ell}({\KK})$, resp.~$\bar{\t}_{e\ell\ell}({\KK})$, the module of \textit{infinitesimal elliptic braids}, 
resp.~of \textit{infinitesimal reduced elliptic braids}. 

\medskip

We finally define the $\CD(\kk)$-module $\CD_{e\ell\ell}(\kk) := \mathcal{\hat U}(\bar\t_{e\ell\ell}(\kk))$ of 
\textit{elliptic chord diagrams}. 
Similarly to the genus $0$ situation, morphisms in $\CD_{e\ell\ell}(\kk)(n)$ can be represented as chords on $n$ vertical 
strands, with extra chords correponding to the generators $x_i$ and $y_i$ as in the following picture: 
\begin{center}
\tik{
\tell{0}{1}{+}
\node[point, label=above:$i$] at (1,0) {};
\node[point, label=below:$i$] at (1,-1) {}; 
}
\text{ 	and }
\tik{
\tell{0}{1}{-}
\node[point, label=above:$i$] at (1,0) {};
\node[point, label=below:$i$] at (1,-1) {};
}
\end{center}
The  elliptic relations introduced above can be represented as follows, analogously as for the genus 0 case: 
\begin{center}

\begin{align}\tag{\ref{eqn:Sell}}
\tik{
\tell{0}{1}{{-}}
\tell[->]{1}{2}{{+}}
\straight[->]{1}{1}
\straight{2}{0}
\node[point, label=above:$i$] at (1,0) {};
\node[point, label=below:$i$] at (1,-2) {};
\node[point, label=above:$j$] at (2,0) {};
\node[point, label=below:$j$] at (2,-2) {};
}
-
\tik{\tell{0}{2}{{+}}
\tell[->]{1}{1}{{-}}
\straight[->]{2}{1}
\straight{1}{0}
\node[point, label=above:$i$] at (1,0) {};
\node[point, label=below:$i$] at (1,-2) {};
\node[point, label=above:$j$] at (2,0) {};
\node[point, label=below:$j$] at (2,-2) {};
}
=
\tik{\tell{0}{1}{{+}} \tell[->]{1}{2}{{-}}\straight[->]{1}{1}\straight{2}{0}
\node[point, label=above:$i$] at (1,0) {};\node[point, label=below:$i$] at (1,-2) {}; 
\node[point, label=above:$j$] at (2,0) {};\node[point, label=below:$j$] at (2,-2) {};}
-\tik{\tell{0}{2}{{-}} \tell[->]{1}{1}{{+}}\straight[->]{2}{1}\straight{1}{0}
\node[point, label=above:$i$] at (1,0) {};\node[point, label=below:$i$] at (1,-2) {}; 
\node[point, label=above:$j$] at (2,0) {};\node[point, label=below:$j$] at (2,-2) {};}
 =\ \tik{ \hori[->]{0}{1}{1}{}{}\node[point, label=above:$i$] at (0,-1) {};
 \node[point, label=below:$i$] at (0,-2) {}; \node[point, label=above:$j$] at (1,-1) {};
 \node[point, label=below:$j$] at (1,-2) {};}
\end{align}

\begin{align}\tag{\ref{eqn:Nell}}
\tik{\tell{0}{1}{{\pm}} \tell[->]{1}{2}{{\pm}}\straight[->]{1}{1}\straight{2}{0}
\node[point, label=above:$i$] at (1,0) {};\node[point, label=below:$i$] at (1,-2) {}; 
\node[point, label=above:$j$] at (2,0) {};\node[point, label=below:$j$] at (2,-2) {};}=
\tik{\tell{0}{2}{{\pm}} \tell[->]{1}{1}{{\pm}}\straight[->]{2}{1}\straight{1}{0}
\node[point, label=above:$i$] at (1,0) {};\node[point, label=below:$i$] at (1,-2) {}; 
\node[point, label=above:$j$] at (2,0) {};
\node[point, label=below:$j$] at (2,-2) {};}
\end{align}

\begin{align}\tag{\ref{eqn:Tell}}
\tik{\tell{0}{1}{{+}} \tell[->]{1}{1}{{-}}\straight[->]{1}{1} 
\node[point, label=above:$i$] at (1,0) {};\node[point, label=below:$i$] at (1,-2) {};}
-\tik{\tell{0}{1}{A^{-}} \tell[->]{1}{1}{{+}}\node[point, label=above:$i$] at (1,0) {};
\node[point, label=below:$i$] at (1,-2) {};}
=\ - \sum_{j;j\neq i} \tik{ \hori[->]{0}{1}{1} \node[point, label=above:$i$] at (0,-1) {}; 
\node[point, label=above:$j$] at (1,-1) {};\node[point, label=below:$i$] at (0,-2) {}; 
\node[point, label=below:$j$] at (1,-2) {};}
\end{align}

\begin{align}\tag{\ref{eqn:Lell}}
\tik{\tell{0}{1}{{\pm}}\draw[zell] (0,-1)--(0,-2); \hori[->]{2}{1}{1}\straight[->]{1}{1}
\straight{2}{0}\straight{3}{0}\node[point, label=above:$i$] at (1,0) {};
\node[point, label=below:$i$] at (1,-2) {}; \node[point, label=above:$j$] at (2,0) {};
\node[point, label=below:$j$] at (2,-2) {};\node[point, label=above:$k$] at (3,0) {};
\node[point, label=below:$k$] at (3,-2) {};}=\tik{\tell[->]{1}{1}{{\pm}}\draw[zell] (0,0)--(0,-1); 
\hori{2}{0}{1}\straight{1}{0}\straight[->]{2}{1}\straight[->]{3}{1}
\node[point, label=above:$i$] at (1,0) {};\node[point, label=below:$i$] at (1,-2) {}; 
\node[point, label=above:$j$] at (2,0) {};\node[point, label=below:$j$] at (2,-2) {};
\node[point, label=above:$k$] at (3,0) {};\node[point, label=below:$k$] at (3,-2) {};}
\end{align}

\begin{align}\tag{\ref{eqn:4Tell}}
\tik{\tell{0}{1}{{\pm}} \draw[zell] (0,-1)--(0,-2);\hori[->]{1}{1}{1}\straight[->]{1}{1}\straight{2}{0}
\node[point, label=above:$i$] at (1,0) {};\node[point, label=below:$i$] at (1,-2) {}; 
\node[point, label=above:$j$] at (2,0) {};\node[point, label=below:$j$] at (2,-2) {};}
+
\tik{\tell{0}{2}{{\pm}} \draw[zell] (0,-1)--(0,-2);\hori[->]{1}{1}{1}\straight[->]{1}{1}\straight{2}{0}\straight{1}{0}
\node[point, label=above:$i$] at (1,0) {};\node[point, label=below:$i$] at (1,-2) {}; 
\node[point, label=above:$j$] at (2,0) {};\node[point, label=below:$j$] at (2,-2) {};}
=
\tik{\tell[->]{1}{1}{{\pm}} \draw[zell] (0,0)--(0,-1);\hori{1}{0}{1} \straight[->]{2}{1}
\node[point, label=above:$i$] at (1,0) {};\node[point, label=below:$i$] at (1,-2) {}; 
\node[point, label=above:$j$] at (2,0) {};\node[point, label=below:$j$] at (2,-2) {};}
+
\tik{\tell[->]{1}{2}{{\pm}} \draw[zell] (0,0)--(0,-1);
\hori{1}{0}{1} \straight[->]{2}{1}\straight[->]{1}{1}\node[point, label=above:$i$] at (1,0) {};
\node[point, label=below:$i$] at (1,-2) {}; \node[point, label=above:$j$] at (2,0) {};
\node[point, label=below:$j$] at (2,-2) {};}
\end{align}

\end{center}

\begin{remark}
The relation between (a closely related version of) $\CD_{e\ell\ell}(\kk)$ and the elliptic 
Kontsevich integral was studied in Philippe Humbert's thesis \cite{HPh}.
\end{remark}


\section{The $\PaCD(\kk)$-module of parenthesized elliptic chord diagrams}\label{sec-pacdell}

As in the genus zero case, the module of objects $\on{Ob}(\CD_{e\ell\ell}(\kk))$ of $\CD_{e\ell\ell}(\kk)$ is terminal. 
Hence there is a morphism of modules 
$\omega_2:\Pa=\on{Ob}(\mathbf{Pa}(\kk))\to\on{Ob}(\CD_{e\ell\ell}(\kk))$ over the 
morphism of operads $\omega_1$ from \S\ref{sec-pacd}, and thus we can define 
the $\PaCD(\kk)$-module\footnote{Recall that $\PaCD(\kk)$ is defined as $\omega_1^\star \CD(\kk)$. }
\[
\text{\gls{PaCDe}}:=\omega_2^\star \CD_{e\ell\ell}(\kk)\,,
\]
in $\mathbf{Cat(CoAss_\KK)}$, of so-called \textit{parenthesized elliptic chord diagrams}. 

\begin{example}[Notable arrows in $\mathbf{PaCD}_{e\ell\ell}(\kk)(2)$]
There are the following arrows $X^{1,2}_{e\ell\ell}$, $Y^{1,2}_{e\ell\ell}$ in $\mathbf{PaCD}_{e\ell\ell}(\kk)(2)$: 
\begin{center}
$X^{1,2}_{e\ell\ell}=x_1 \cdot$
\begin{tikzpicture}[baseline=(current bounding box.center)]
\tikzstyle point=[circle, fill=black, inner sep=0.05cm]
 \node[point, label=above:$1$] at (1,1) {};
 \node[point, label=below:$1$] at (1,-0.25) {};
 \node[point, label=above:$2$] at (2,1) {};
 \node[point, label=below:$2$] at (2,-0.25) {};
 \draw[->,thick] (1,1) .. controls (1,0) and (1,0).. (1,-0.20); 
 \draw[->,thick] (2,1) .. controls (2,0.25) and (2,0.5).. (2,-0.20);
\end{tikzpicture}
$=$
\tik{
\draw[->,thick] (1,0) .. controls (1,0) and (1,-1).. (1,-0.20-1);
\tell{0.1}{1}{+}
\node[point, label=above:$1$] at (1,0) {};
\node[point, label=below:$1$] at (1,-0.25-1) {};
\node[point, label=above:$2$] at (2,1-1) {};
\node[point, label=below:$2$] at (2,-0.25-1) {};
\draw[->,thick] (2,1-1) .. controls (2,0.25-1) and (2,0.5-1).. (2,-0.20-1);
}
\qquad\qquad
$Y^{1,2}_{e\ell\ell}=y_1 \cdot$
\begin{tikzpicture}[baseline=(current bounding box.center)]
\tikzstyle point=[circle, fill=black, inner sep=0.05cm]
 \node[point, label=above:$1$] at (1,1) {};
 \node[point, label=below:$1$] at (1,-0.25) {};
 \node[point, label=above:$2$] at (2,1) {};
 \node[point, label=below:$2$] at (2,-0.25) {};
 \draw[->,thick] (1,1) .. controls (1,0) and (1,0).. (1,-0.20); 
 \draw[->,thick] (2,1) .. controls (2,0.25) and (2,0.5).. (2,-0.20);
\end{tikzpicture}
$=$
\tik{
\draw[->,thick] (1,0) .. controls (1,0) and (1,-1).. (1,-0.20-1);
\tell{0.1}{1}{-}
\node[point, label=above:$1$] at (1,0) {};
\node[point, label=below:$1$] at (1,-0.25-1) {};
\node[point, label=above:$2$] at (2,1-1) {};
\node[point, label=below:$2$] at (2,-0.25-1) {};
\draw[->,thick] (2,1-1) .. controls (2,0.25-1) and (2,0.5-1).. (2,-0.20-1);
}
\end{center}
\end{example}

\begin{remark}\label{PaCD:ell:rel} 
As said in section \ref{sec-pointings}, there is a map of $\mathfrak{S}$-modules 
$\PaCD(\kk) \longrightarrow \PaCD_{e\ell\ell}(\kk)$ and we abusively denote $X^{1,2}$, 
$H^{1,2}$ and $a^{1,2,3}$ the images in $\PaCD_{e\ell\ell}(\kk)$ of the corresponding arrows in $\PaCD(\kk)$. 
The elements $X^{1,2}_{e\ell\ell}, Y^{1,2}_{e\ell\ell}$ are generators of the ${\PaCD}(\kk)$-module 
$\mathbf{PaCD}_{e\ell\ell}(\kk)$ and satisfy the following relations in 
$\on{End}_{\mathbf{PaCD}_{e\ell\ell}(\kk)(2)}(12)$:
\begin{itemize}
\item $X_{e\ell\ell}^{1,2}+X^{1,2}X_{e\ell\ell}^{2,1}(X^{1,2})^{-1}=0$,
\item $Y_{e\ell\ell}^{1,2}+X^{1,2}Y_{e\ell\ell}^{2,1}(X^{1,2})^{-1}=0$.
\end{itemize}
They also satisfy the following relations in $\on{End}_{\mathbf{PaCD}_{e\ell\ell}(\kk)(3)}((12)3)$:
\begin{itemize}
\item $X_{e\ell\ell}^{12,3}+a^{1,2,3}X^{1,23}X_{e\ell\ell}^{23,1}(a^{1,2,3}X^{1,23})^{-1}
+X^{12,3}(a^{3,1,2})^{-1}X_{e\ell\ell}^{31,2}(X^{12,3}(a^{3,1,2})^{-1})^{-1}=0$,
\item $Y_{e\ell\ell}^{12,3}+a^{1,2,3}X^{1,23}Y_{e\ell\ell}^{23,1}(a^{1,2,3}X^{1,23})^{-1}
+X^{12,3}(a^{3,1,2})^{-1}Y_{e\ell\ell}^{31,2}(X^{12,3}(a^{3,1,2})^{-1})^{-1}=0$,
\item $H^{1,2}=[a^{1,2,3}X_{e\ell\ell}^{1,23}(a^{1,2,3})^{-1},
X^{1,2}a^{2,1,3}Y_{e\ell\ell}^{2,13}(a^{2,1,3})^{-1}(X^{1,2})^{-1}]$.
\end{itemize}
\end{remark}


\section{Elliptic associators}\label{sec:3.5ellasoc}

Let us introduce some terminology, complementing the one of \S\ref{sec:2.8assoc}. 
Let us write $\tmop{OpR}\mathcal C$ for the category of pairs $(\mathcal P,\mathcal M)$, where $\mathcal P$ 
is an operad and $\mathcal M$ is a right $\mathcal O$-module, in $\mathcal C$. 
A morphism $(\mathcal P,\mathcal M)\to (\mathcal Q,\mathcal N)$ is a pair $(f,g)$, where 
$f:\mathcal{P} \to \mathcal{Q}$ is a morphism between operads and $g:\mathcal{M} \to \mathcal{N}$ is a morphism 
of $\mathcal{P}$-modules. 

The superscript ``$+$'' still indicates that we consider morphisms that are the identity on objects. 

\begin{definition}
An elliptic associator over $\KK$ is a couple $(F,G)$ where $F$ is a $\kk$-associator and $G$ is an 
isomorphism between the $\widehat{\PaB}(\KK)$-module 
$\widehat{\PaB}_{e\ell\ell}(\KK)$ and the $G \PaCD(\KK)$-module $G \PaCD_{e\ell\ell}(\KK)$ which is the 
identity on objects and which is compatible with $F$:
\[
\text{\gls{Asse}} := \on{Iso}^+_{\tmop{OpR}\mathbf{Grpd}_\kk}
\Big(\big(\widehat{\PaB}(\KK),\widehat{\PaB}_{e\ell\ell}(\KK)\big),\big(G \PaCD(\KK),G \PaCD_{e\ell\ell}(\KK)\big)\Big).
\] 
\end{definition}

The following theorem identifies our definition of elliptic associators with the original one introduced by Enriquez in \cite{En2}.

\begin{theorem}\label{thm:bij-ell-assoc}
There is a one-to-one correspondence between the set $\Ell(\KK)$ and
the set $\text{\gls{bAsse}}$ of quadruples $(\mu,\varphi,A_+,A_{-})$, 
where $(\mu,\varphi)\in \on{Ass}(\KK)$ and $A_\pm\in 
\on{exp}(\hat{\bar\t}_{1,2}(\KK))$, such that: 
\begin{equation}\label{def:ell:ass:1}
\alpha_\pm^{1,2,3}\alpha_\pm^{2,3,1} \alpha_\pm^{3,1,2}= 1, 
\text{ where }\alpha_\pm =\varphi^{1,2,3}
A_\pm^{1,23} e^{-\mu(t_{12}+t_{13})/2}, 
\end{equation}
\begin{equation}\label{def:ell:ass:2}
e^{\mu t_{12}}= \left(\varphi A_+^{1,23}\varphi^{-1},
e^{-\mu t_{12}/2}\varphi^{2,1,3} A_-^{2,13}(\varphi^{2,1,3})^{-1}e^{-\mu t_{12}/2}
\right).  
\end{equation}
All these relations hold in the group $\on{exp}(\hat{\bar\t}_{1,3}(\KK))$.
\end{theorem}

\begin{proof}
An associator $F$ corresponds uniquely to a couple $(\mu,\varphi)\in\on{Ass}(\kk)$ and an isomorphism 
$G$ between $\widehat{\PaB}_{e\ell\ell}(\KK)$ and $G \PaCD_{e\ell\ell}(\KK)$ sends the arrows $A^{1,2}$ 
and $B^{1,2}$ of $\on{End}_{\widehat{\mathbf{PaB}}_{e\ell\ell}(\kk)(2)}(12)$  to $A_+ \cdot X_{e\ell\ell}^{1,2}$ 
and $A_{-} \cdot Y_{e\ell\ell}^{1,2}$ with $A_{\pm}\in\on{exp}(\hat{\bar\t}_{1,2}(\kk))$ (recall that 
$\hat{\bar\t}_{1,2}(\kk)$ is the completed free Lie algebra over $\kk$ in two generators). The image of 
relations \eqref{eqn:N1}, \eqref{eqn:N2} and \eqref{eqn:E} are precisely the relations \eqref{def:ell:ass:1} 
and \eqref{def:ell:ass:2}.
\end{proof}

\begin{example}[Elliptic KZB Associators]\label{example-KZB-assoc}
Let us fix $\tau \in \h$. Recall that the Lie algebra $\bar\t_{1,2}(\C)$ is isomorphic to the free 
Lie algebra $\mathfrak{f}_2(\C)$ generated by two elements $x:=x_1$ and $y:=y_1$.
We define the elliptic KZB associators $\text{\gls{KZAsse}}:=(A(\tau),B(\tau))$ as the renormalized holonomies 
from $0$ to $1$ and $0$ to $\tau$ of the differential equation
 \begin{equation} \label{eq:dept2}
 G'(z) =  -{\frac{\theta_{\tau}(z+\on{ad}x)\on{ad}x}{\theta_{\tau}(z)
 \theta_{\tau}(\on{ad}x)}}(y)\cdot G(z), 
 \end{equation}
with values in the group $\on{exp}(\hat{\bar\t}_{1,2}(\C))$
More precisely, this equation has a unique solution $G(z)$ defined over $\{a+b\tau\,|\, a, b\in(0,1)\}$
such that $G(z)\simeq (-2\pi\on{i}z)^{-[x,y]}$ at $z\to 0$. In this case,
\[
A(\tau):= G(z)G(z+1)^{-1}, \quad  B(\tau):= 
G(z)G(z+\tau)^{-1}e^{-2\pi\on{i}x}. 
\]
These are elements of the group $\on{exp}(\hat{\bar\t}_{1,2}(\C))$. More precisely, 
Enriquez showed in \cite{En2} that the element $(2\pi\i,\Phi_{\on{KZ}}, A(\tau),B(\tau))$ is in $\on{Ell}(\C)$.
\end{example}


\section{Elliptic Grothendieck--Teichm\"uller group} 

\begin{definition}
The ($\kk$-prounipotent version of the) \textit{elliptic Grothendieck--Teichm\"uller group} is defined as the group 
\[
\text{\gls{kGTe}}:=\on{Aut}_{\tmop{OpR}\mathbf{Grpd}_\kk}^+
\big({\widehat\PaB(\KK)},\widehat\PaB_{e\ell\ell}(\KK)\big)\,.
\]
\end{definition}

Again, we now show that our definition coincides with the original one defined by Enriquez in \cite{En2}.
Recall that the set $\text{\gls{bkGTe}}$ is the set of tuples $(\lambda,f,g_\pm)$, where
$(\lambda,f)\in \widehat{\on{GT}}(\KK)$, $g_\pm\in \widehat{\on{F}}_2(\KK)$ such that, 
in $\widehat{\bar{\on{B}}}_{1,3}(\kk)$, 
\begin{equation} \label{def:GTell:1}
(f(\sigma_1^2,\sigma_2^2)
g_\pm(A,B) (\sigma_1\sigma_2^2\sigma_1)^{- \frac{\lambda-1}{2}}\sigma_1^{- 1}\sigma_2^{- 1})^3=1\,, 
\end{equation}
\begin{equation} \label{def:GTell:3}
u^2= (g_+,u^{-1}g_-u^{-1})\,,
\end{equation}
where $u=f(\sigma_1^2,\sigma_2^2)^{-1}\sigma_1^{\lambda} f(\sigma_1^2,\sigma_2^2)$ 
and $g_\pm=g_\pm(A,B)$. 

For $(\lambda,f,g_{\pm}),(\lambda',f',g'_{\pm})\in \widehat{\on{GT}}_{e\ell\ell}(\KK)$, 
we set 
\[
(\lambda,f,g_\pm)*(\lambda',f',g'_\pm):=(\lambda'',f'',g''_\pm)\,,
\]
where $g''_\pm(A,B) = g_\pm(g'_+(A,B),g'_-(A,B))$. 
This gives $\widehat{\on{GT}}_{e\ell\ell}(\KK)$ a group structure.
Moreover, for $(\lambda,f,g_+,g_{-})\in\widehat{\on{GT}}_{ell}(\kk)$ 
and $(\mu,\varphi,A_+,A_{-})\in\on{Ell}(\kk)$, we set 
\[
(\lambda,f,g_+,g_{-})*(\mu,\varphi,A_+,A_{-}):= (\mu',\varphi',A'_+,A'_{-})\,,
\]
where $A'_\pm := g_\pm(A_+,A_-)$. In \cite{En2}, it is shown that this defines a free and transitive 
left action of $\widehat{\on{GT}}_{e\ell\ell}(\KK)$ on $\on{Ell}(\kk)$. 

\begin{proposition}\label{Ell:GT}
There is an isomorphism $\widehat{\GT}_{e\ell\ell}(\KK)\longrightarrow\widehat{\on{GT}}_{e\ell\ell}(\KK)$ 
such that the bijection $\Ell(\KK)\tilde\longrightarrow\on{Ell}(\kk)$ becomes a torsor isomorphism. 
\end{proposition}
\begin{proof}
Suppose that we are given an automorphism $(F,G)$ of $\big(\widehat\PaB(\KK),\widehat\PaB_{e\ell\ell}(\KK)\big)$ 
which is the identity on objects. We already know (see \S\ref{sec2-gt}) that $F$ is determined by a pair 
$(\lambda, f)\in\widehat\GT(\kk)$, and that any such pair determines an $F$. Moreover, the images of the two 
generators $A^{1,2}, B^{1,2} \in \on{Aut}_{\widehat\PaB_{e\ell\ell}(\KK)(2)} (12)=\widehat{\bar{\PB}}_{1,2}(\KK)$ are 
\[
G(A^{1,2})=g_+(A^{1,2},B^{1,2})\qquad\textrm{and}\qquad G(B^{1,2})=g_-(A^{1,2},B^{1,2})\,,
\]
with $g_\pm\in\widehat{\on{F}}_2(\KK) \simeq \widehat{\bar{\PB}}_{1,2}(\KK)$. 
It therefore follows from Theorem \ref{PaBell} that $(\lambda,f,g_\pm)$ satisfies relations 
\eqref{def:GTell:1} and \eqref{def:GTell:3} if and only it determines an automorphism $(F,G)$. 

Let us then prove that the bijective assignement $(F,G)\mapsto (\lambda,f,g_\pm)$ that we just described is a group morphism. 
For this we show that the composition of automorphisms corresponds to the group law of $\on{GT}_{e\ell\ell}(\kk)$. 
We already know (see \S\ref{sec2-gt}) that the composition of automorphisms of $\widehat{\PaB}(\kk)$ corresponds 
to the group law in $\on{GT}(\kk)$. Now, given automorphisms $(F_1,G)$ and $(F_2,H)$, and there respective images 
$(\lambda_1,f_1,g_\pm)$ and $(\lambda_2,f_2,h_\pm)$, we have that 
\[
(H \circ G)(A)=H(g_+(A,B))=g_+(H(A), H(B))=g_+(h_+(A,B), h_-(A,B))\,,
\]
and, likewise, $(H \circ G)(B)=g_-(h_+(A,B), h_-(A,B))$. 

We finally prove the equivariance statement. Let $(F,G)\in\GT_{e\ell\ell}(\KK)$, with image 
$(\lambda,f,g_{\pm})\in\on{GT}_{e\ell\ell}(\kk)$, and let $(K,H)\in\Ell_{e\ell\ell}(\KK)$, with image 
$(\mu,\varphi,A_\pm)$. It is known (see \S\ref{torsors}) that $K\circ F$ is sent to $(\mu,\varpi)*(\lambda,f)$. 
It remains to compute: 
\[
(H\circ G)(A)=H(g_+(A,B))=g_+(H(A), H(B))=g_+(A_+, A_-)\,,
\]
and, similarly, $(H \circ G)(B)=g_-(A_+, A_-)$. 
\end{proof}


\section{Graded elliptic Grothendieck--Teichm\"uller group} 

\begin{definition}
The graded elliptic Grothendieck-Teichm\"uller group is the group 
\[
\text{\gls{GRTe}}:=\on{Aut}^+_{\tmop{OpR}\mathbf{Cat}(\mathbf{CoAlg}_\kk)}\big(\PaCD(\kk),\PaCD_{e\ell\ell}(\KK)\big)\,. 
\]
\end{definition}
Notice that there is an isomorphism 
\[
\on{Aut}^+_{\tmop{OpR}\mathbf{Cat}(\mathbf{CoAlg}_\kk)}\big(\PaCD(\kk),\PaCD_{e\ell\ell}(\KK)\big)\simeq 
\on{Aut}^+_{\tmop{OpR}\mathbf{Grpd}_\kk}\big(G\PaCD(\kk),G\PaCD_{e\ell\ell}(\KK)\big)\,.
\]

As before, our goal in this paragraph is to show that our definition coincides with the one of Enriquez \cite{En2}. 
Recall that he defines $\on{GRT}_1^{ell}(\kk)$ as the set of triples 
$(g,u_+,u_{-})\in\on{GRT}_1(\kk)\times\left(\hat{\bar{\t}}_{1,2}(\kk)\right)^{\times 2}$, 
satisfying  
\begin{equation} \label{def:grt:ell:1}
\on{Ad}(g^{1,2,3})(u_\pm^{1,23})+
\on{Ad}(g^{2,1,3})(u_\pm^{2,13})+u_\pm^{3,12}=0\,, 
\end{equation}
\begin{equation} \label{def:grt:ell:2}
[\on{Ad}(g^{1,2,3})(u_\pm^{1,23}),u_\pm^{3,12}]=0\,, 
\end{equation}
\begin{equation} \label{def:grt:ell:3}
[\on{Ad}(g^{1,2,3})(u_+^{1,23}),
\on{Ad}(g^{2,1,3})(u_-^{2,13})]=t_{12}\,,
\end{equation}
as relations in $\hat{\bar{\t}}_{1,3}({\kk})$. 
He defines a group structure as follows: 
\[
(g,u_+,u_{-})*(h,v_+,v_{-}):= (g*h,w_+,w_{-})
\,,\quad\textrm{where}\quad 
w_\pm(x_1,y_1):= u_\pm(v_+(x_1,y_1),v_-(x_1,y_1))\,.
\]
The group $\kk^\times$ acts on $\on{GRT}_1^{ell}(\kk)$ by rescaling: $c\cdot (g,u_\pm):= (c\cdot g,c\cdot u_\pm)$, 
where $c\cdot g$ is as before, and
\begin{itemize}
\item $(c\cdot u_+)(x_1,y_1):= u_+(x_1,c^{-1}y_1)$,
\item $(c\cdot u_-)(x_1,y_1):= cu_-(x_1,c^{-1}y_1)$.
\end{itemize}
We then set $\text{\gls{bGRTe}}:= \on{GRT}_1^{ell}(\kk)\rtimes \kk^\times$. 

Moreover, there is a right action of $\on{GRT}_1^{ell}(\kk)$ on $\on{Ell}(\kk)$: 
for $(g,u_\pm)\in\on{GRT}_1^{ell}(\kk)$ and $(\mu,\varphi,A_\pm)\in \on{Ell}(\kk)$, we set 
$(\mu,\varphi,A_\pm)*(g,u_\pm):= (\mu,\tilde\varphi,\tilde A_\pm)$, where 
\[
\tilde A_\pm(x_1,y_{1}):= A_\pm(u_+(x_1,y_1),u_-(x_1,y_1)) 
\]
and, for $c\in\kk^\times$, we set $(\mu,\varphi,A_\pm)*c:= (\mu,c*\varphi,c\sharp A_\pm)$, where
$(c\sharp A_\pm)(x_1,y_1):= A_\pm(x_1,cy_1)$. In \cite{En2} this action is shown to be free and transitive.
Notice that $\tilde A_\pm=\theta(A_\pm)$, where $\theta\in\on{Aut}(\hat{\bar{\t}}_{1,2}^\kk)$ is 
defined by $x_1\mapsto u_+(x_1,y_1)$ and $y_1\mapsto u_-(x_1,y_1)$.

\begin{proposition}\label{GRT:ell:cor}
There is an injective group morphism $\GRT_{e\ell\ell}(\KK)\to\on{GRT}_{e\ell\ell}(\kk)$.
Moreover, the bijection $\Ell(\KK) \to \on{Ell}(\kk)$ from Theorem \ref{thm:bij-ell-assoc} is equivariant along this morphism. 
\end{proposition}

\begin{proof}
For every $(G,U)\in\GRT_{e\ell\ell}(\KK)$, there are $(\lambda, g)\in\on{GRT}(\kk)$ and $u_\pm\in\hat{\bar{\t}}_{1,2}(\kk)$ such that 
\begin{itemize}
\item $G(X^{1,2})=X^{1,2}$,
\item $G(H^{1,2})=\lambda H^{1,2}$,
\item $G(a^{1,2,3})=g(t_{12},t_{23})a^{1,2,3}$,
\item $U(X^{1,2}_{e\ell\ell})=u_+(x,y)\on{Id}_{12}$,
\item $U(Y^{1,2}_{e\ell\ell})=u_-(x,y)\on{Id}_{12}$. 
\end{itemize}
In light of relations of Remark \ref{PaCD:ell:rel}, we obtain that $(\lambda, g, u_\pm)$ satisfies 
relations \eqref{def:grt:ell:1}, \eqref{def:grt:ell:2} and \eqref{def:grt:ell:3}. 
The assignment $(G,U)\mapsto (\lambda, g, u_\pm)$ defines an injective map $\GRT_{e\ell\ell}(\KK) \to \on{GRT}_{e\ell\ell}(\kk)$.

We now show that this map is a group morphism. The proof is the same as one of the analogous statement 
in Proposition \ref{Ell:GT}: for two automorphisms $(G_1,U)$ and $(G_2,V)$, we already know that the composition 
$G_2\circ G_1$ corresponds to the product in $\on{GRT}(\kk)$, and we compute: 
\[
(V\circ U)(X^{1,2}_{e\ell\ell})=V(u_+(x_1,y_1)\on{Id}_{12})=u_+\big(v_+(x_1,y_1),v_-(x_1,y_1)\big)\on{Id}_{12}\,,
\]
and, likewise, $(V\circ U)(X^{1,2}_{e\ell\ell})=u_-\big(v_+(x_1,y_1),v_-(x_1,y_1)\big)\on{Id}_{12}$. 

Finally, the equivariance of the bijection is proven in a similar way. 
\end{proof}


\section{Bitorsors}

Summarizing the results we have proven so far, we get that the bijection $\Ell(\KK) \longrightarrow \on{Ell}(\kk)$ 
from Theorem \ref{thm:bij-ell-assoc} has been promoted to a bitorsor isomorphism. Indeed, we know (by definition) that 
\[
\big(\widehat{\GT}_{e\ell\ell}(\kk),\Ell(\kk),\GRT_{e\ell\ell}(\kk)\big)
\]
is a bitorsor, and (from \cite{En2}) that 
\[
\big(\widehat{\on{GT}}_{e\ell\ell}(\kk),\on{Ell}(\kk),\on{GRT}_{e\ell\ell}(\kk)\big)
\]
is a bitorsor as well. 
\begin{theorem}\label{thm-intro1}
There is a bitorsor isomorphism
\begin{equation}\label{bitorsor:ell}
\big(\widehat{\GT}_{e\ell\ell}(\kk),\Ell(\kk),\GRT_{e\ell\ell}(\kk)\big) \tilde\longrightarrow 
\big(\widehat{\on{GT}}_{e\ell\ell}(\kk),\on{Ell}(\kk),\on{GRT}_{e\ell\ell}(\kk)\big)\,.
\end{equation}
\end{theorem}
\begin{proof}
This is a summary of most of the above results: 
\begin{itemize}
\item There is a group isomorphism between $\widehat{\GT}_{e\ell\ell}(\kk)$ and $\widehat{\on{GT}}_{e\ell\ell}(\kk)$ 
that is such that the bijection from Theorem \ref{thm:bij-ell-assoc} is a torsor isomorphism (Proposition \ref{Ell:GT}). 
\item There is an injective group morphism $\GRT_{e\ell\ell}(\kk)\to\on{GRT}_{e\ell\ell}(\kk)$ such that 
the bijection from Theorem \ref{thm:bij-ell-assoc} is equivariant (Proposition \ref{GRT:ell:cor}). 
\end{itemize}
Knowing from Example \ref{example-KZB-assoc} that $\on{Ell}(\kk)$ is non-empty, we obtain that 
$\GRT_{e\ell\ell}(\kk)\to\on{GRT}_{e\ell\ell}(\kk)$ is an isomorphism. 
\end{proof}

\chapter{The module of parenthesized ellipsitomic braids}
\label{Section4}

In this chapter, $\Gamma$ denotes the abelian group 
$\Gamma=\Z/M\Z \times \Z/N\Z$ where $M,N\geq 1$ are two integers. 
We also write $\0:=(\bar{0},\bar{0})$. 


\section{Compactified twisted configuration space of the torus} 

Let $\mathbb{T}$ be the topological torus, and consider the connected 
$\Gamma$-covering $p:\tilde{\mathbb{T}} \to \mathbb{T}$ corresponding to the 
canonical surjective group morphism $\rho:\pi_1(\mathbb{T})=\Z^2\to\Gamma$
 sending the generators of $\mathbb{Z}^2$ to their corresponding reduction in $\Gamma$. 
To any finite set $I$ with cardinality $n$ we associate the $\Gamma$-twisted configuration space
\[
\text{\gls{ConfTIG}}:=\left\{\mathbf{z}=(z_1,\dots,z_n)\in\tilde{\mathbb{T}}^I\,\big|\,p(z_i)\neq p(z_j)\textrm{ if }i\neq j\right\}\,,
\]
and let $\text{\gls{CTIG}}:=\textrm{Conf}(\mathbb{T},I,\Gamma)/\tilde{\mathbb{T}}$ be its reduced version.  

\medskip

The inclusion 
\begin{equation}\label{eqn:inclusion}
\textrm{Conf}(\mathbb{T},I,\Gamma) \hookrightarrow \textrm{Conf}(\tilde{\mathbb{T}},I\times\Gamma)
\end{equation}
sending $(z_i)_{i\in I}$ to $(\gamma\cdot z_i)_{(i,\gamma)\in I\times\Gamma}$ induces an inclusion
\[
\text{\gls{CTIG}}	\hookrightarrow \textrm{C}(\tilde{\mathbb{T}},I\times\Gamma) 
									\hookrightarrow \overline{\textrm{C}}(\tilde{\mathbb{T}},I\times\Gamma)\,,
\]
which allows us to define $\gls{FMTIG}$ as the closure of $\text{\gls{CTIG}}$ inside 
$\overline{\textrm{C}}(\tilde{\mathbb{T}},I\times\Gamma)$. 
The boundary 
$\partial\overline{\textrm{C}}(\mathbb{T},I,\Gamma)=\overline{\textrm{C}}(\mathbb{T},I,\Gamma)-\textrm{C}(\mathbb{T},I,\Gamma)$ 
is made up of the following irreducible components: for any partition $J_1\coprod\cdots\coprod J_k$ of $I$ there is a component 
\[
\partial_{J_1,\cdots,J_k}\overline{\textrm{C}}(\mathbb{T},I,\Gamma)\cong 
\prod_{i=1}^k(\overline{\rm C}(\mathbb{C},J_i))\times\overline{\rm C}(\mathbb{T},k,\Gamma)\,.
\]
The inclusion of boundary components provides $\overline{\textrm{C}}(\mathbb{T},-,\Gamma)$ with the structure of a 
module over the operad $\overline{\textrm{C}}(\mathbb{C},-)$ in topological spaces. 

\medskip

On the one hand, the left action of $\Gamma$ on $\tilde{\mathbb{T}}$ gives us an action of $\Gamma^I$, resp.~$\Gamma^I/\Gamma$, on 
$\textrm{Conf}(\tilde{\mathbb{T}},I\times\Gamma)$, resp.~$\textrm{C}(\tilde{\mathbb{T}},I\times\Gamma)$. On the other hand, 
$\Gamma^I$ also acts on $\textrm{Conf}(\tilde{\mathbb{T}},I\times\Gamma)$ and $\textrm{C}(\tilde{\mathbb{T}},I\times\Gamma)$ 
in the following way: 
\[
(\alpha\cdot\mathbf{z})_{(i,\gamma)}:=\mathbf{z}_{i,\gamma+\alpha}\,.
\]
The inclusion \eqref{eqn:inclusion} is $\Gamma^I$-equivariant, so that one gets a diagonally trivial $\Gamma$-action 
on $\overline{\textrm{C}}(\mathbb{C},-)$, in the sense of \S\ref{sec-gpaction}. 


\section{The $\mathbf{Pa}$-module of labelled parenthesized permutation}\label{sec4.2}

For every finite set $I$, there is a $\Gamma^I/\Gamma$-covering map 
\[
\phi_I\,:\,\textrm{C}(\mathbb{T},n,\Gamma) \longrightarrow  \textrm{C}(\mathbb{T},n)
\]
which extends to a continuous map 
\[
\bar\phi_I\,:\,\overline{\textrm{C}}(\mathbb{T},I,\Gamma) \longrightarrow  \overline{\textrm{C}}(\mathbb{T},I)\,,
\] 
everything being natural (with respective to bijections) in $I$. This defines a morphism $\bar\phi$ of 
$\overline{\textrm{C}}(\mathbb{C},-)$-modules from $\overline{\textrm{C}}(\mathbb{T},-,\Gamma)$ to 
$\overline{\textrm{C}}(\mathbb{T},-)$. 

\medskip

Recall from \S\ref{sec-pabell} that there are inclusions of topological operadic modules
$\mathbf{Pa}\subset\,\overline{\textrm{C}}(\mathbb{S}^1,-)\,\subset\,\overline{\textrm{C}}(\mathbb{T},-)$
over $\mathbf{Pa}\subset\,\overline{\textrm{C}}(\mathbb{R},-)\,\subset\,\overline{\textrm{C}}(\mathbb{C},-)$. 
We define the $\mathfrak{S}$-module $\mathbf{Pa^{\Gamma}}:=\bar\phi^{-1}\mathbf{Pa}$, which carries a 
$\mathbf{Pa}$-module structure. Indeed, it is a fiber product
\[
\mathbf{Pa^{\Gamma}}:=
\mathbf{Pa}\underset{\overline{\textrm{C}}(\mathbb{T},-)}{\times}\overline{\textrm{C}}(\mathbb{T},-,\Gamma)
\]
in the category of $\mathbf{Pa}$-modules in topological space. 

\medskip

The $\mathbf{Pa}$-module $\mathbf{Pa^\Gamma}$ admits the following algebraic description. First of all, it is discrete, 
in the sense that spaces of operations are discrete (i.e., they are just sets). 
Then, an element of $\mathbf{Pa^\Gamma}(n)$ is a parenthesized permutation of $1\dots n$ together with a 
label function $\{1,\dots,n\}\to\Gamma$ that is defined up to a global relabelling (i.e.~the labelling is an element of 
$\Gamma^n/\Gamma$). 
For instance, $2_\gamma1_\0=2_\01_{-\gamma}$ belongs to $\mathbf{Pa^\Gamma}(2)$ for every $\gamma\in\Gamma$. 
In geometric terms, having the label $[\gamma_1,\dots,\gamma_n]$ means that, in our configuration of points, the 
$(-\gamma_i)\cdot z_i$'s are on the same parallel of the torus. 
Here is a self-explanatory example of partial composition: 
\[
(3_{\0}2_\gamma)1_\delta\circ_2(12)3=(3_{\0}((2_\gamma3_\gamma)4_\gamma))1_\delta\,.
\]
Finally, $\mathbf{Pa^\Gamma}$ is acted on by $\Gamma$ in the following way: for $n\geq0$, $\Gamma^n$ 
only acts on the labellings, \textit{via} the group law of $\Gamma$. 
For instance, if $[\underline{\alpha}]\in\Gamma^n/\Gamma$ and $\underline{\gamma}\in\Gamma^n$, then 
$\underline{\gamma}\cdot[\underline{\alpha}]:=[\underline{\gamma}+\underline{\alpha}]$.

\medskip

In other words, according to the terminology of \S\ref{sec-gpaction} and \S\ref{sec-semifake}, 
$\mathbf{Pa^\Gamma}$ is identified with $\mathcal G(\Pa\to \overline{\Gamma})$. 


\section{The $\mathbf{PaB}$-module of parenthesized ellipsitomic braids}\label{sec-twpabell}

We define 
\[
\text{\gls{PaBeg}}:=\pi_1\left(\overline{\textrm{C}}(\mathbb{T},-,\Gamma),\mathbf{Pa^{\Gamma}}\right)\,,
\]
which is a $\PaB$-module (in groupoids), that also carries a diagonally trivial action of $\Gamma$. 
The morphism $\bar\phi$ induces a $\PaB$-module morphism 
$\mathbf{PaB}^{\mathbf{\Gamma}}_{e\ell\ell}\to \mathbf{PaB}_{e\ell\ell}$. 

\begin{example}[Notable arrows in $\mathbf{PaB}^{\mathbf{\Gamma}}_{e\ell\ell}$]
Recall the following notable arrows in $\mathbf{PaB}_{e\ell\ell}$: 
\begin{itemize}
\item $A^{1,2}$ and $B^{1,2}$ are automorphisms of $12$ in $\mathbf{PaB}_{e\ell\ell}(2)$. 
\item $R^{1,2}$ goes from $12$ to $21$ in $\mathbf{PaB}_{e\ell\ell}(2)$. 
\item $\Phi^{1,2,3}$ goes from $(12)3$ to $1(23)$ in $\mathbf{PaB}_{e\ell\ell}(2)$. 
\end{itemize}
All are represented by paths which, apart from the endpoints that are in $\mathbf{Pa}$, remain 
in the open part $\textrm{C}(\mathbb{T},n)$ of the configuration spaces ($n=2,3$). Let us set 
$\alpha:=(\bar{1},\bar{0})$ and $\beta:=(\bar{0},\bar{1})$. Since there are covering maps 
\[
\textrm{C}(\mathbb{T},n,\Gamma) \longrightarrow  \textrm{C}(\mathbb{T},n)\,,
\]
then these paths admits unique lifts, with starting point being the same parenthesized permutation 
with the trivial labelling (the one being constantly equal to $\mathbf{0}$). 
These lifts are denoted the same way:  
\begin{itemize}
\item The lift $A^{1,2}$ goes from $1_\02_\0$ to $1_{\alpha}2_{\0}=1_{\0}2_{-\alpha}$ 
in $\mathbf{PaB}^{\mathbf{\Gamma}}_{e\ell\ell}(2)$. 
\item The lift $B^{1,2}$ goes from $1_\02_\0$ to $1_{\beta}2_{\0}=1_{\0}2_{-\beta}$ 
in $\mathbf{PaB}^{\mathbf{\Gamma}}_{e\ell\ell}(2)$. 
\item etc... 
\end{itemize}
Here is a drawing of paths representing $A^{1,2}$ and $B^{1,2}$: 
\[
\includegraphics[scale=1]{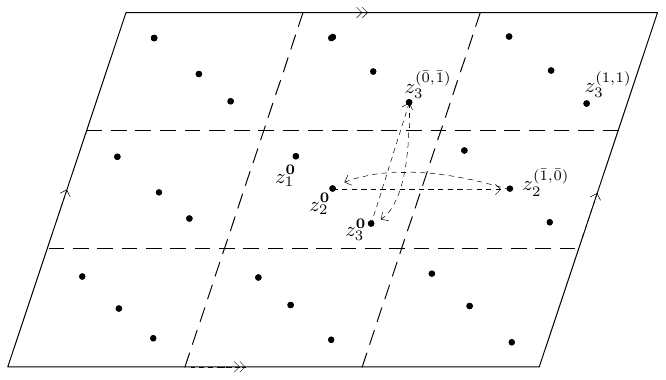}
\]
We may chose to alternatively depict them as diagrams representing elliptic pure braids (i.e.~loops in the base 
configuration space) together with appropriate base points (which uniquely determines the lift in the covering
twisted configuration space): 
\begin{center}
$A^{1,2}=$
\begin{tikzpicture}[baseline=(current bounding box.center)]
\tikzstyle point=[circle, fill=black, inner sep=0.05cm]
 \node[point, label=above:$1_\0$] at (1,1) {};
 \node[point, label=below:$1_\alpha$] at (1,0) {};
 \node[point, label=above:$2_\0$] at (2,1) {};
 \node[point, label=below:$2_\0$] at (2,0) {};
 \draw[-,thick] (1,1) .. controls (1,0.5) and (1,0.5).. (1.5,0.5); 
 \draw[->,thick] (1.5,0.5) .. controls (1,0.5) and (1,0.5).. (1,0.05); 
\node[point, white, label=left:$+$] at (1,0.5) {};
 \draw[->,thick] (2,1) .. controls (2,0.5) and (2,0.5).. (2,0.05);
\end{tikzpicture}
\quad and \quad $B^{1,2}=$
\begin{tikzpicture}[baseline=(current bounding box.center)] 
\tikzstyle point=[circle, fill=black, inner sep=0.05cm]
\tikzstyle point2=[circle, fill=black, inner sep=0.05cm]
 \node[point, label=above:$1_\0$] at (1,1) {};
 \node[point, label=below:$1_\beta$] at (1,0) {};
 \node[point, label=above:$2_\0$] at (2,1) {};
 \node[point, label=below:$2_\0$] at (2,0) {};
 \draw[-,thick] (1,1) .. controls (1,0.5) and (1,0.5).. (1.5,0.5); 
 \draw[->,thick] (1.5,0.5) .. controls (1,0.5) and (1,0.5).. (1,0.05); 
\node[point, white, label=left:$-$] at (1,0.5) {};
 \draw[->,thick] (2,1) .. controls (2,0.5) and (2,0.5).. (2,0.05);
\end{tikzpicture}
\end{center}
\begin{remark}
It is important to observe that, the action of $\Gamma$ being diagonally trivial, one can shift the global 
labelling of the indexed points, and thus $A^{1,2}$ and $B^{1,2}$ can also be represented as follows:
\begin{center}
$A^{1,2}=$
\begin{tikzpicture}[baseline=(current bounding box.center)]
\tikzstyle point=[circle, fill=black, inner sep=0.05cm]
 \node[point, label=above:$1_\0$] at (1,1) {};
 \node[point, label=below:$1_\0$] at (1,0) {};
 \node[point, label=above:$2_\0$] at (2,1) {};
 \node[point, label=below:$2_{-\alpha}$] at (2,0) {};
 \draw[-,thick] (1,1) .. controls (1,0.5) and (1,0.5).. (1.5,0.5); 
 \draw[->,thick] (1.5,0.5) .. controls (1,0.5) and (1,0.5).. (1,0.05); 
\node[point, white, label=left:$+$] at (1,0.5) {};
 \draw[->,thick] (2,1) .. controls (2,0.5) and (2,0.5).. (2,0.05);
\end{tikzpicture}
\quad and \quad $B^{1,2}=$
\begin{tikzpicture}[baseline=(current bounding box.center)] 
\tikzstyle point=[circle, fill=black, inner sep=0.05cm]
\tikzstyle point2=[circle, fill=black, inner sep=0.05cm]
 \node[point, label=above:$1_\0$] at (1,1) {};
 \node[point, label=below:$1_\0$] at (1,0) {};
 \node[point, label=above:$2_\0$] at (2,1) {};
 \node[point, label=below:$2_{-\beta}$] at (2,0) {};
 \draw[-,thick] (1,1) .. controls (1,0.5) and (1,0.5).. (1.5,0.5); 
 \draw[->,thick] (1.5,0.5) .. controls (1,0.5) and (1,0.5).. (1,0.05); 
\node[point, white, label=left:$-$] at (1,0.5) {};
 \draw[->,thick] (2,1) .. controls (2,0.5) and (2,0.5).. (2,0.05);
\end{tikzpicture}
\end{center}
\end{remark}
As for $R^{1,2}$ and $\Phi^{1,2,3}$, they are depicted in the usual way: 
\begin{center}
$R^{1,2}$=
\begin{tikzpicture}[baseline=(current bounding box.center)]
\tikzstyle point=[circle, fill=black, inner sep=0.05cm]
 \node[point, label=above:$1_\0$] at (1,1) {};
 \node[point, label=below:$2_\0$] at (1,-0.25) {};
 \node[point, label=above:$2_\0$] at (2,1) {};
 \node[point, label=below:$1_\0$] at (2,-0.25) {};
 \draw[->,thick] (2,1) .. controls (2,0.25) and (1,0.5).. (1,-0.20); 
 \node[point, ,white] at (1.5,0.4) {};
 \draw[->,thick] (1,1) .. controls (1,0.25) and (2,0.5).. (2,-0.20);
\end{tikzpicture}
\quad and \quad $\Phi^{1,2,3}$=
\begin{tikzpicture}[baseline=(current bounding box.center)]
\tikzstyle point=[circle, fill=black, inner sep=0.05cm]
 \node[point, label=above:$(1_\0$] at (1,1) {};
 \node[point, label=below:$1_\0$] at (1,-0.25) {};
 \node[point, label=above:$2_\0)$] at (1.5,1) {};
 \node[point, label=below:$(2_\0$] at (2,-0.25) {};
 \node[point, label=above:$3_\0$] at (2.5,1) {};
 \node[point, label=below:$3_\0)$] at (2.5,-0.25) {};
 \draw[->,thick] (1,1) .. controls (1,0) and (1,0).. (1,-0.20); 
 \draw[->,thick] (1.5,1) .. controls (1.5,0.25) and (2,0.5).. (2,-0.20);
 \draw[->,thick] (2.5,1) .. controls (2.5,0) and (2.5,0).. (2.5,-0.20);
\end{tikzpicture}
\end{center}
\end{example} 

\medskip

Actually, every morphism in $\PaB_{e\ell\ell}$ can be uniquely lifted to $\mathbf{PaB}^{\mathbf{\Gamma}}_{e\ell\ell}$, 
once the lift of the source object has been fixed; all other lifts are obtained by the $\overline{\Gamma}$-action. 
Moreover, all morphisms can be obtained like this. 
This shows that the $\PaB$-module $\mathbf{PaB}^{\mathbf{\Gamma}}_{e\ell\ell}$ has an alternative simple algebraic 
descrition that we explain now. 
First observe that the $\PaB$-module $\PaB_{e\ell\ell}$ comes with a morphism $\pi$ to the $*$-module 
$\overline{\Gamma}$, which is the composition of the abelianization morphism to $\overline{\mathbb{Z}^2}$ 
with the projection $\overline{\mathbb{Z}^2}\to \overline{\Gamma}$. 

In terms of the presentation from Theorem \ref{PaBell}, 
\[
\pi(A)=\alpha_1=[(\bar1,0),\0]\quad\textrm{and}\quad\pi(B)=\beta_1=[(0,\bar1),\0]\,,
\]
where we adopt the following notation: 
\begin{notation}
For $\gamma\in\Gamma$ and $1\leq i\leq n$, then we write 
\[
\gamma_i:=[\0,\dots,\0,\underset{i}{\gamma},\0,\dots,\0]\in\Gamma^n/\Gamma\,.
\]
\end{notation}

\begin{proposition}\label{proposition-gPeBell}
There is an isomorphism 
\[
\mathcal G\big(\PaB_{e\ell\ell}\to \overline{\Gamma}\big)\tilde\longrightarrow \PaB_{e\ell\ell}^{\mathbf{\Gamma}}
\]
of $\PaB$-modules with a $\overline{\Gamma}$-action, which is is the identity on objects.  
\end{proposition}
\begin{proof}
We first describe the morphism: 
\begin{itemize}
\item It is the identity on objects; 
\item Given two labelled parenthesized permutations $(\mathbf{p},\gamma)$ and $(\mathbf{q},\delta)$, it sends 
a the class in $\PaB_{e\ell\ell}$ of a path $f:\mathbf{p}\to\mathbf{q}$ such that $\gamma+\pi(f)=\delta$ to the 
class of the unique lift of $f$ that starts at the base point determined by $(\mathbf{p},\gamma)$. 
\end{itemize}
As we have already seen, to show that this morphism is in fact an isomorphism, it suffices to show that it is an 
isomorphism at the level of automorphism groups of objects arity-wise. This is indeed the case, as on both sides, 
in arity $n$, the automorphism group of an object is the kernel of the morphism 
$\overline{\on{PB}}_{1,n}\to\Gamma^n/\Gamma$ sending $X_i$ to $(\bar1,\bar0)_i$ and $Y_j$ to ($\bar0,\bar1)_j$. 
\end{proof}


\section{The universal property of $\PaB_{e\ell\ell}^\Gamma$}

We are now ready to provide an explicit presentation for the $\PaB$-module $\PaB_{e\ell\ell}^{\mathbf{\Gamma}}$. 
As before, we keep the convention that $\alpha=(\bar1,\bar0)$ and $\beta=(\bar0,\bar1)$. 
\begin{theorem}\label{PaB:ell:G}
As a $\mathbf{PaB}$-module in groupoids with a diagonally trivial $\Gamma$-action and having 
$\mathbf{Pa^{\Gamma}}$ as $\mathbf{Pa}$-module of objects, $\PaB_{e\ell\ell}^{\mathbf{\Gamma}}$ 
is freely generated by $A:1_\02_\0\to1_\alpha2_\0$ and $B:1_\02_\0\to1_\beta2_\0$, together with 
the following relations, satisfied in 
$\on{Aut}_{\PaB_{e\ell\ell}^{\mathbf{\Gamma}}(3)\rtimes(\Gamma^3/\Gamma)}\big((1_\02_\0)3_\0\big)$: 
\begin{equation}\label{def:PaB:ell:G:1}\tag{tN1}
\Phi^{1,2,3} \underline{A}^{1,23} \tilde R^{1,23} 
\Phi^{2,3,1} \underline{A}^{2,31} \tilde R^{2,31} 
\Phi^{3,1,2} \underline{A}^{3,12} \tilde R^{3,12} 
=\on{Id}_{(1_\02_\0)3_\0}
\end{equation}
\begin{equation}\label{def:PaB:ell:G:2}\tag{tN2}
\Phi^{1,2,3} \underline{B}^{1,23} \tilde R^{1,23} 
\Phi^{2,3,1} \underline{B}^{2,31} \tilde R^{2,31} 
\Phi^{3,1,2} \underline{B}^{3,12} \tilde R^{3,12} 
=\on{Id}_{(1_\02_\0)3_\0}
\end{equation}
\begin{equation}
R^{1,2} R^{2, 1}=\big(\Phi^{1,2,3} \underline{A}^{1,23}(\Phi^{1,2,3})^{- 1}\,,\,
\tilde R^{1,2}\Phi^{2,1,3}\underline{B}^{2,13}(\Phi^{2,1,3})^{- 1} 
\tilde R^{2,1}\big)\label{def:PaB:ell:G:3}\tag{tE}
\end{equation}
where $\underline{A}:=A\alpha_1$ and $\underline{B}:=B\beta_1$. 
\end{theorem}
\begin{remark}
The above relations are clearer when stated within the semidirect product, even though they can be written within 
$\PaB_{e\ell\ell}^{\mathbf{\Gamma}}$ itself. For instance, \eqref{def:PaB:ell:G:1} can be written as 
\[
\Phi^{1,2,3} A^{1,23} \alpha_1\cdot\Big(\tilde R^{1,23} 
\Phi^{2,3,1} A^{2,31} \alpha_2\cdot\big(\tilde R^{2,31} 
\Phi^{3,1,2} A^{3,12}\big)\Big) \tilde R^{3,12} 
=\on{Id}_{(1_\02_\0)3_\0}\,.
\]
The expression for \eqref{def:PaB:ell:G:3} becomes unpleasant to write. 
\end{remark}
\begin{proof}[Proof of the Theorem]
Let $\mathcal Q^{\Gamma}_{e\ell\ell}$ be the $\PaB$-module with the above presentation, and let 
$\mathcal Q_{e\ell\ell}$ be the $\PaB$-module with the presentation in Theorem \ref{PaBell}. 
Our goal is to prove that there is an isomorphism 
\[
\mathcal G(\mathcal Q_{e\ell\ell}\to\overline{\Gamma})\tilde\longrightarrow\mathcal Q^{\Gamma}_{e\ell\ell}
\]
of $\PaB$-modules with a $\overline{\Gamma}$-action, which is is the identity on objects. 
The result will then follow from Proposition \ref{proposition-gPeBell}. 

By definition there is a morphism 
$\mathcal Q_{e\ell\ell}\longrightarrow\mathcal Q^{\Gamma}_{e\ell\ell}\rtimes\overline{\Gamma}$, 
which sends $A$ to $\underline{A}$, and $B$ to $\underline{B}$. Moreover, when we compose this morphism with 
the projection $\mathcal Q^{\Gamma}_{e\ell\ell}\rtimes\overline{\Gamma}\to\overline{\Gamma}$, we get back the 
morphism $\pi:\mathcal Q_{e\ell\ell}\to \overline{\Gamma}$ from the previous chapter, that sends $A$ to 
$\alpha_1$ and $B$ to $\beta_1$. 

By the adjunction from \S\ref{sec-semifake}, we therefore get a morphism 
\[
\mathcal G(\mathcal Q_{e\ell\ell}\to\overline{\Gamma})\longrightarrow\mathcal Q^{\Gamma}_{e\ell\ell}\,.
\]
of $\PaB$-modules with a $\overline{\Gamma}$-action. It is surjective on morphisms, because both generators of 
$\mathcal Q^{\Gamma}_{e\ell\ell}$ have preimages. 
Finally, as we have already seen, to show that this is in fact an isomorphism, it suffices to show that it is an isomorphism 
at the level of automorphism groups of objects arity-wise, and it is sufficient to do it for a single object in every arity. 

Let $n\geq 1$ and let $\tilde p$ be the object $(\cdots((1_\02_\0)3_\0)\cdots )n_\0$ 
of $\mathcal Q^{\Gamma}_{e\ell\ell}(n)$ and $\mathcal G\big(\mathcal Q_{e\ell\ell}(n)\to\Gamma^n/\Gamma\big)$, 
which lifts $p=(\cdots((12)3)\cdots )n$ in $\mathcal Q_{e\ell\ell}(n)$. 
There is a commuting diagram 
\[
\xymatrix{
1\ar[r] & \ar[r] \on{Aut}_{\mathcal G\big(\mathcal Q_{e\ell\ell}(n)\to\Gamma^n/\Gamma\big)}(\tilde p) \ar[r]\ar[d] & 
\on{Aut}_{\mathcal Q_{e\ell\ell}(n)}(p)\ar[r] & \Gamma^n/\Gamma\ar[r] & 1 \\
& \on{Aut}_{\mathcal Q_{e\ell\ell}^\Gamma(n)}(\tilde p) \ar[ur]& & & 
}
\]
where the horizontal sequence is exact. Therefore the vertical morphism is injective, and we are done. 
\end{proof}


\section{Ellipsitomic Grothendieck--Teichm\"uller groups}

\begin{definition}
The ($\kk$-pro-unipotent version of the) 
\textit{ellipsitomic Grothendieck--Teichm\"uller group} is defined as  
\[
\text{\gls{kGTeg}}:=
\on{Aut}_{\tmop{OpR}\mathbf{Grpd}_\kk}^+\big(\widehat\PaB(\KK),\widehat\PaB_{e\ell\ell}^{\mathbf{\Gamma}}(\KK)\big)^{\mathbf{\Gamma}}\,,
\]
where, as usual, the superscript $\Gamma$ means that we 
are considering the subgroup of $\Gamma$-equivariant automorphisms. 
\end{definition}

Let $M',N'\geq 1$, and assume we are given a surjective group morphism 
\[
\rho:\Gamma\twoheadrightarrow\Gamma':=\mathbb{Z}/M'\mathbb{Z}\times \mathbb{Z}/N'\mathbb{Z}\,.
\]
This gives a (surjective) map between the corresponding covering 
spaces of the torus, which can be used to construct 
a morphism of $\overline{C}(\mathbb{C},-)$-modules 
\[
\overline{C}(\mathbb{T},-,\Gamma)\longrightarrow \overline{C}(\mathbb{T},-,\Gamma')\,.
\]
Following the construction of sections \ref{sec4.2} and \ref{sec-twpabell}, we get a morphism of $\PaB$-modules 
\[
\PaB^{\mathbf{\rho}}_{e\ell\ell}\,:\,\PaB^{\mathbf{\Gamma}}_{e\ell\ell}\longrightarrow \PaB^{\mathbf{\Gamma'}}_{e\ell\ell}\,.
\]
The morphism $\PaB^{\mathbf{\rho}}_{e\ell\ell}$ is $\Gamma$-equivariant, and has a straightforward algebraic description: 
\begin{itemize}
\item On objects, it consists in applying $\rho$ to the labelling, keeping the underlying parentheiszed permutation unchanged; 
\item It sends the generating morphisms $A^{1_{\0}2_{\0}}$ and $B^{1_{\0}2_{\0}}$ in $\PaB_{e\ell\ell}^\Gamma$ to their 
counterparts (which are denoted the same way) in $\PaB_{e\ell\ell}^{\mathbf{\Gamma'}}$. 
\end{itemize}
As a consequence, the $\PaB$-module $\PaB^{\mathbf{\Gamma'}}_{e\ell\ell}$ can be obtained as the quotient of 
$\PaB^{\mathbf{\Gamma}}_{e\ell\ell}$ by $\ker\rho$. 
We therefore obtain a group morphism 
$\widehat{\GT}_{e\ell\ell}^{\mathbf{\Gamma}}(\kk)\longrightarrow \widehat{\GT}_{e\ell\ell}^{\mathbf{\Gamma'}}(\kk)$. 

\chapter{Ellipsitomic chord diagrams and ellipsitomic associators}
\label{Section5}

\section{Infinitesimal ellipsitomic braids}\label{sec:deft1n}

In this paragraph and the next one, $(\Gamma,\mathbf{0},+)$ can be any finite abelian group. 

For any $n\geq 0$ we define \gls{t1nG} to be the bigraded $\kk$-Lie algebra with generators 
$x_i$ ($1\leq i \leq n$) in degree $(1,0)$, $y_i$ ($1\leq i \leq n$) in degree $(0,1)$, and $t^{\gamma}_{ij}$ 
($\gamma\in\Gamma$, $i\neq j$) in degree $(1,1)$, and relations 
\begin{flalign}
& t^{\gamma}_{ij} = t^{-\gamma}_{ji}\,, 									\tag{tS$_{e\ell\ell}1$}\label{eqn:etS} \\
& [x_i,y_j] = [x_j,y_i] = \sum_{\gamma\in\Gamma} t^{\gamma}_{ij}\,,	 		\tag{tS$_{e\ell\ell}2$}\label{eqn:etSbis} \\
& [x_i,x_j] = [y_i,y_j] = 0\,, 												\tag{tN$_{e\ell\ell}$}\label{eqn:etN} \\
& [x_i,y_i] = - \sum_{j:j\neq i} \sum_{\gamma\in\Gamma} t^{\gamma}_{ij},	\tag{tT$_{e\ell\ell}$}\label{eqn:etT} \\
& [t^{\gamma}_{ij},t^{\delta}_{kl}]=0\,,									\tag{tL$_{e\ell\ell}1$}\label{eqn:etL1} \\
& [x_i,t^{\gamma}_{jk}] = [y_i,t^{\gamma}_{jk}] = 0,  						\tag{tL$_{e\ell\ell}2$}\label{eqn:etL2}   \\
& [t^{\gamma}_{ij},t^{\gamma+\delta}_{ik}+t^{\delta}_{jk}] =0\,, 			\tag{t4T$_{e\ell\ell}1$}\label{eqn:et4T1} \\
& [x_i + x_j,t^{\gamma}_{ij}] = [y_i+y_j,t^{\gamma}_{ij}] = 0\,, &			\tag{t4T$_{e\ell\ell}2$}\label{eqn:et4T2} 
\end{flalign}
where $1\leq i,j,k,l\leq n$ are pairwise distinct and $\gamma, \delta \in\Gamma$. 
We will call $\t_{1,n}^\Gamma(\kk)$ the $\kk$-Lie algebra of \textit{infinitesimal ellipsitomic braids}.
Observe that $\sum_i x_i$ and $\sum_i y_i$ are central in $\t_{1,n}^\Gamma$. Then we denote by 
$\bar\t_{1,n}^\Gamma(\kk)$ the quotient of $\t_{1,n}^\Gamma(\kk)$ by $\sum_i x_i$ and $\sum_i y_i$, and the natural 
morphism $\t_{1,n}^\Gamma(\kk)\to\bar\t_{1,n}^\Gamma(\kk)\,;\,u\mapsto\bar u$.

\medskip

There is an alternative presentation of $\t_{1,n}^\Gamma(\kk)$ and $\bar\t_{1,n}^\Gamma(\kk)$: 
\begin{lemma}\label{lem:pres1}
The Lie $\kk$-algebra $\t_{1,n}^\Gamma(\kk)$ (resp.~$\bar\t_{1,n}^\Gamma(\kk)$) can equivalently be 
presented with the same generators, and the following relations: 
\eqref{eqn:etS}, \eqref{eqn:etSbis}, \eqref{eqn:etN}, \eqref{eqn:etL1}, \eqref{eqn:etL2}, \eqref{eqn:et4T1}, 
and, for every $i\in I$,
\begin{equation}\tag{tC$_{e\ell\ell}$}\label{eqn:etC} 
[\sum_jx_j,y_i]=[\sum_jy_j,x_i]=0
\end{equation}
(resp.~$\sum_jx_j=\sum_jy_j=0$). 
\end{lemma}
\begin{proof}
If $x_i,y_i$ and $t^{\alpha}_{ij}$ satisfy the initial relations, then 
\[
[\underset{j}{\sum} x_j, y_i]=[ x_i, y_i]+[\underset{j\neq i}{\sum} x_j, y_i]
=- \sum_{j:j\neq i} \sum_{\gamma\in\Gamma} t^{\gamma}_{ij}+ \sum_{j:j\neq i} \sum_{\gamma\in\Gamma} t^{\gamma}_{ij}=0\,.
\] 
Now, if $x_i,y_i$ and $t^{\alpha}_{ij}$ satisfy the above relations, then relations 
$[\underset{j}{\sum} x_j, y_i]=0$ and $[x_j,y_i] = \sum_{\gamma\in\Gamma} t^{\gamma}_{ij}$, 
for $i\neq j$, imply that $[ x_i, y_i]=- \sum_{j:j\neq i} \sum_{\gamma\in\Gamma} t^{\gamma}_{ij}$. 
Now, relations $[\underset{k}{\sum} x_k, y_j]=0$ and $[\underset{k}{\sum} x_k, x_i]=0$ 
imply that $[\underset{k}{\sum} x_k, \sum_{\gamma\in\Gamma} t^{\gamma}_{ij}]=0$. 
Thus, as $[x_i,t^{\gamma}_{jk}] = 0 $ if $ \on{card}\{i,j,k\}=3, $ we obtain relation 
$[x_i + x_j,t^{\gamma}_{ij}] = 0 $, for $i\neq j$. 
In the same way we obtain $[y_i+y_j,t^{\gamma}_{ij}] = 0$, for $i\neq j$.
\end{proof}

Both  $\t^\Gamma_{1,n}(\KK)$ and $\bar\t^\Gamma_{1,n}(\KK)$ are acted on 
by the symmetric group $\mathfrak{S}_n$, we get that 
\[
\t_{e\ell\ell}^{\Gamma}(\kk):=\{\t^\Gamma_{1,n}(\KK) \}_{n\geq 0}\quad\mathrm{and}\quad
\bar\t^\Gamma_{e\ell\ell}({\KK}):=\{\bar\t^\Gamma_{1,n}(\KK) \}_{n\geq 0}
\]
define $\mathfrak{S}$-modules in $grLie_\kk$. They are actually $\t(\kk)$-module in 
$grLie_\kk$, where partial compositions are defined as follows\footnote{We give the 
formul\ae~for $\t_{e\ell\ell}^{\Gamma}(\kk)$. Formul\ae~for $\bar\t^\Gamma_{e\ell\ell}({\KK})$ 
are the exact same. }: for $I$ a finite set and $k\in I$, 
\[
\begin{array}{cccl}
\circ_k : & \t^{\Gamma}_{1,I}(\KK) \oplus \t_J(\KK)  & \longrightarrow & \t^{\Gamma}_{1,J\sqcup I-\{k\}}(\KK) \\
    & (0,t_{uv}) & \longmapsto & \quad t_{uv}^{\0} \\
 & (t^{\gamma}_{ij},0) & \longmapsto & 
 \begin{cases}
  \begin{tabular}{lll}
  $t^{\gamma}_{ij}$ & if & $ k\notin\{i,j\} $ \\
  $\sum\limits_{p\in J} t^{\gamma}_{pj}$ & if & $k=i$ \\
  $\sum\limits_{p\in J} t^{\gamma}_{ip}$ & if & $j=k$ 
  \end{tabular}
  \end{cases}\\
   & (x_i,0) & \longmapsto & 
 \begin{cases}
  \begin{tabular}{lll}
  $x_i$ & if & $ k \neq i $ \\
  $\sum\limits_{p\in J} x_{p}$ & if & $k=i$ 
  \end{tabular}
  \end{cases}\\
     & (y_i,0) & \longmapsto & 
 \begin{cases}
  \begin{tabular}{lll}
  $y_i$ & if & $ k \neq i $ \\
  $\sum\limits_{p\in J} y_{p}$ & if & $k=i$ 
  \end{tabular}
  \end{cases}
\end{array}
\]
We call $\t^\Gamma_{e\ell\ell}({\KK})$, resp.~$\bar{\t}^\Gamma_{e\ell\ell}({\KK})$, the module of \textit{infinitesimal 
ellipsitomic braids}, resp.~of \textit{infinitesimal reduced ellipsitomic braids}. 
When $\kk=\C$ we write $\t_{1,n}^\Gamma:=\t_{1,n}^\Gamma(\C)$ and $\bar\t_{1,n}^\Gamma:=\bar\t_{1,n}^\Gamma(\C)$.

\medskip

Both $\t(\KK)$-modules are acted on by $\Gamma$: 
any element $\underline{\gamma}=(\gamma_i)_{i\in I}\in\Gamma^I$ acts as 
\begin{eqnarray*}
& \underline{\gamma} \cdot x_i=x_i 											& (i\in I)\,, 			\\
& \underline{\gamma} \cdot y_i=y_i 											& (i\in I)\,, 			\\
& \underline{\gamma}\cdot t^\delta_{ij}=t^{\delta+\gamma_i-\gamma_j}_{ij}	& (\delta\in\Gamma\textrm{ and }i\neq i)\,.
\end{eqnarray*}

We also have the following functoriality in $\Gamma$, with respect to surjections: 
\begin{proposition}\label{prop:comparison}
Let $\rho:\Gamma_1\twoheadrightarrow\Gamma_2$ be a surjective group morphism, and  let $a,b,c,d\in\kk$ 
such that $ad-bc=|\ker\rho|$. There are $\Gamma_1^I$-equivariant surjective \textit{comparison morphisms} 
$\t_{1,I}^{\Gamma_1}(\kk)\to\t_{1,I}^{\Gamma_2}(\kk)$ and 
$\bar\t_{1,I}^{\Gamma_1}(\kk)\to\bar\t_{1,I}^{\Gamma_2}(\kk)$, defined by 
\[
x_i\mapsto ax_i+by_i\,,\quad y_i\mapsto cx_i+dy_i\,,\quad t_{ij}^\gamma\mapsto t_{ij}^{\rho(\gamma)}\,.
\]
These are morphisms of $\t(\kk)$-modules in $grLie_\kk$. 
\end{proposition}
\begin{proof}
This follows from direct computations. 
\end{proof}
Actually, these morphisms exhibit $\t_{1,I}^{\Gamma_2}(\kk)$, resp.~$\bar\t_{1,I}^{\Gamma_2}(\kk)$, 
as the quotient of $\t_{1,I}^{\Gamma_1}(\kk)$, resp.~$\bar\t_{1,I}^{\Gamma_1}(\kk)$, by $(\ker\rho)^I$. 
\begin{remark}\label{rem:comparison}
Whenever $\Gamma_i=\mathbb{Z}/M_i\mathbb{Z}\times \mathbb{Z}/N_i\mathbb{Z}$, there is a natural choice for the scalars $a,b,c,d$. 
Indeed, if $\rho:\Gamma_1\to\Gamma_2$ is surjective, then there exists elements $(a,b)$ and $(c,d)$ in the lattice 
$M_2\mathbb{Z}\times N_2\mathbb{Z}$ that generate the sublattice $M_1\mathbb{Z}\times N_1\mathbb{Z}$. Hence, in particular, 
the determinant $ad-bc$ equals $\frac{M_1N_1}{M_2N_2}=|\ker\rho|$. 
\end{remark}


\section{Horizontal ellipsitomic chord diagrams}

In this paragraph we define the $\CD(\kk)$-module $\CD^{\mathbf{\Gamma}}_{e\ell\ell}(\kk)$ of 
\textit{ellipsitomic chord diagrams}. 

We first consider the $\CD(\kk)$-module $\mathcal{\hat U}(\bar\t^{\Gamma}_{e\ell\ell}(\kk))$. 
Morphisms in $\mathcal{\hat U}(\bar\t^{\Gamma}_{e\ell\ell}(\kk))$ can be given a pictorial description, 
which mixes the features of the horizontal $N$-chord diagrams from \cite{Adrien} (see also \cite{CG-cyclo}) 
together with the elliptic chord diagrams from \S\ref{sec-cdell}. Diagrams corresponding to 
$x_i$ and $y_j$ are, respectively, 
\begin{center}
\tik{
\tell{0}{1}{+}
\node[point, label=above:$i$] at (1,0) {};
\node[point, label=below:$i$] at (1,-1) {};
}
=
\tik{
\tell{0}{1}{+}
\node[point, label=above:$i$] at (1,0) {};
\node[point, label=below:$i$] at (1,-1) {};
\node[diam, label=right:$\gamma$] at (1,-0.35) {};
\node[diam, label=right:$-\gamma$] at (1,-0.65) {}; 
}
\text{ 	and }
\tik{
\tell{0}{1}{-}
\node[point, label=above:$j$] at (1,0) {};
\node[point, label=below:$j$] at (1,-1) {};
}
=
\tik{
\tell{0}{1}{-} \node[point, label=above:$j$] at (1,0) {};
\node[point, label=below:$j$] at (1,-1) {};
\node[diam, label=right:$\gamma$] at (1,-0.35) {};
\node[diam, label=right:$-\gamma$] at (1,-0.65) {};
}
\end{center}
and the one corresponding to $t_{ij}^\gamma=t_{ji}^{-\gamma}$ is 
\begin{center}
\tik{
\hori[->]{0}{0}{1}	\node[point, label=above:$i$] at (0,0) {};
\node[point, label=above:$j$] at (1,0) {};
\node[point, label=below:$i$] at (0,-1) {};
\node[point, label=below:$j$] at (1,-1) {};
\node[diam, label=left:$\gamma$] at (0,-0.35) {};
\node[diam, label=left:$-\gamma$] at (0,-0.65) {};}
=
\tik{
\hori[->]{0}{0}{1} \node[point, label=above:$i$] at (0,0) {};
\node[point, label=above:$j$] at (1,0) {};
\node[point, label=below:$i$] at (0,-1) {};
\node[point, label=below:$j$] at (1,-1) {};
\node[diam, label=right:$\gamma$] at (1,-0.65) {};
\node[diam, label=right:$-\gamma$] at (1,-0.35) {};}
\end{center}
Relations can be depicted as follows:
\begin{equation}\label{eq:labEll1t}\tag{tS$_{e\ell\ell}2$}
\tik{
\tell{0}{1}{\mp} 
\tell[->]{1}{2}{\pm}
\straight[->]{1}{1}
\straight{2}{0}
\node[point, label=above:$i$] at (1,0) {};
\node[point, label=below:$i$] at (1,-2) {}; 
\node[point, label=above:$j$] at (2,0) {};
\node[point, label=below:$j$] at (2,-2) {};
}
-
\tik{
\tell{0}{2}{\pm}
\tell[->]{1}{1}{\mp}
\straight[->]{2}{1}
\straight{1}{0}
\node[point, label=above:$i$] at (1,0) {};
\node[point, label=below:$i$] at (1,-2) {};
\node[point, label=above:$j$] at (2,0) {};
\node[point, label=below:$j$] at (2,-2) {};
}
=\sum_{\gamma\in\Gamma}
\tik{
\hori[->]{0}{1}{1}
\node[point, label=above:$i$] at (0,-1) {};
\node[point, label=below:$i$] at (0,-2) {};
\node[point, label=above:$j$] at (1,-1) {};
\node[point, label=below:$j$] at (1,-2) {};	
\node[diam, label=left:$\gamma$] at (0,-1.35) {};
\node[diam, label=left:$-\gamma$] at (0,-1.65) {};
}
\end{equation}
\begin{equation}\label{eq:labEll2t}\tag{tN$_{e\ell\ell}$}
\tik{
\tell{0}{1}{\pm}
\tell[->]{1}{2}{\pm}
\straight[->]{1}{1}
\straight{2}{0}
\node[point, label=above:$i$] at (1,0) {};
\node[point, label=below:$i$] at (1,-2) {};
\node[point, label=above:$j$] at (2,0) {};
\node[point, label=below:$j$] at (2,-2) {};
}
=
\tik{
\tell{0}{2}{\pm}
\tell[->]{1}{1}{\pm}
\straight[->]{2}{1}
\straight{1}{0}
\node[point, label=above:$i$] at (1,0) {};
\node[point, label=below:$i$] at (1,-2) {};
\node[point, label=above:$j$] at (2,0) {};
\node[point, label=below:$j$] at (2,-2) {};
}
\end{equation}

\begin{equation}\label{eq:labNEll1G}\tag{tT$_{e\ell\ell}$}
\tik{
\tell{0}{1}{+}
\tell[->]{1}{1}{-}
\straight[->]{1}{1}
\node[point, label=above:$i$] at (1,0) {};
\node[point, label=below:$i$] at (1,-2) {};
}
-
\tik{
\tell{0}{1}{-}
\tell[->]{1}{1}{+}
\node[point, label=above:$i$] at (1,0) {};
\node[point, label=above:$i$] at (1,0) {};
\node[point, label=below:$i$] at (1,-2) {}; 
}
=
\ - \sum_{j;j\neq i}\sum_{\gamma\in\Gamma} 
\tik{
\hori[->]{0}{1}{1}
\node[point, label=above:$i$] at (0,-1) {};
\node[point, label=below:$i$] at (0,-2) {};
\node[point, label=above:$j$] at (1,-1) {};
\node[point, label=below:$j$] at (1,-2) {};	
\node[diam, label=left:$\gamma$] at (0,-1.35) {};
\node[diam, label=left:$-\gamma$] at (0,-1.65) {};
}
\end{equation}

\begin{equation}\label{eq:labNEll1G1}\tag{tL$_{e\ell\ell}1$}
\tik{ 
\hori{0}{0}{1}\straight[->]{0}{1};
 \hori[->]{2}{1}{1}\straight[->]{1}{1}
\straight{2}{0}\straight{3}{0}
	\node[point, label=above:$j$] at (1,0) {}; 
	\node[point, label=above:$k$] at (2,0) {};
	\node[point, label=above:$i$] at (0,0) {}; 
	\node[point, label=above:$l$] at (3,0) {};
	\node[point, label=below:$i$] at (0,-2) {}; 
	\node[point, label=below:$l$] at (3,-2) {};
	\node[point, label=below:$j$] at (1,-2) {}; 
	\node[point, label=below:$k$] at (2,-2) {};
	\node[diam, label=left:$\gamma$] at (0,-0.35) {};
	\node[diam, label=left:$-\gamma$] at (0,-0.65) {};
	\node[diam, label=left:$\delta$] at (2,-1.35) {};
	\node[diam, label=left:$-\delta$] at (2,-1.65) {};
}
= \tik{
\hori[->]{0}{1}{1}\straight{0}{0}; 
\hori{2}{0}{1}\straight{1}{0}\straight[->]{2}{1}\straight[->]{3}{1}
\node[point, label=above:$j$] at (1,0) {}; 
\node[point, label=above:$k$] at (2,0) {};
\node[point, label=above:$i$] at (0,0) {}; 
\node[point, label=above:$l$] at (3,0) {};
\node[point, label=below:$i$] at (0,-2) {}; 
\node[point, label=below:$l$] at (3,-2) {};
\node[point, label=below:$j$] at (1,-2) {};
\node[point, label=below:$k$] at (2,-2) {};
	\node[diam, label=left:$\gamma$] at (0,-1.35) {};
	\node[diam, label=left:$-\gamma$] at (0,-1.65) {};
	\node[diam, label=left:$\delta$] at (2,-0.35) {};
	\node[diam, label=left:$-\delta$] at (2,-0.65) {};
}
\end{equation}

\begin{equation}\label{eq:labNEll2G2}\tag{tL$_{e\ell\ell}2$}
\tik{
\tell{0}{1}{\pm}
\draw[zell] (0,-1)--(0,-2);
\hori[->]{2}{1}{1}
\straight[->]{1}{1}
\straight{2}{0}
\straight{3}{0}
\node[point, label=above:$i$] at (1,0) {};
\node[point, label=below:$i$] at (1,-2) {};
\node[point, label=above:$j$] at (2,0) {};
\node[point, label=below:$j$] at (2,-2) {};
\node[point, label=above:$k$] at (3,0) {};
\node[point, label=below:$k$] at (3,-2) {};	
\node[diam, label=left:$\gamma$] at (2,-1.35) {};
\node[diam, label=left:$-\gamma$] at (2,-1.65) {};
}
=
\tik{
\tell[->]{1}{1}{\pm}
\draw[zell] (0,0)--(0,-1);
\hori{2}{0}{1}
\straight{1}{0}
\straight[->]{2}{1}
\straight[->]{3}{1}
\node[point, label=above:$i$] at (1,0) {};
\node[point, label=below:$i$] at (1,-2) {};
\node[point, label=above:$j$] at (2,0) {};
\node[point, label=below:$j$] at (2,-2) {};
\node[diam, label=left:$\gamma$] at (2,-0.35) {};
\node[diam, label=left:$-\gamma$] at (2,-0.65) {};
\node[point, label=above:$k$] at (3,0) {};
\node[point, label=below:$k$] at (3,-2) {};
}
\end{equation}

\begin{equation}\label{eq:labNEll2G}\tag{t4T$_{e\ell\ell}1$}
		\tik{
		\hori{0}{0}{1}\straight{2}{0}\hori[->]{0}{1}{2}\straight[->]{1}{1}
			\node[point, label=above:$i$] at (0,0) {};
			\node[point, label=above:$j$] at (1,0) {};
			\node[point, label=above:$k$] at (2,0) {};
			\node[point, label=below:$i$] at (0,-2) {};
			\node[point, label=below:$j$] at (1,-2) {};
			\node[point, label=below:$k$] at (2,-2) {};
			\node[diam, label=left:$\gamma$] at (0,-0.35) {};
			\node[diam, label=left:$\delta$] at (0,-1) {};
			\node[diam, label=left:$-\gamma-\delta$] at (0,-1.65) {};
		}
	+	\tik{
	\hori{0}{0}{1} \straight{2}{0}\hori[->]{1}{1}{1}\straight[->]{0}{1}
		\node[point, label=above:$i$] at (0,0) {};
			\node[point, label=above:$j$] at (1,0) {};
			\node[point, label=above:$k$] at (2,0) {};
			\node[point, label=below:$i$] at (0,-2) {};
			\node[point, label=below:$j$] at (1,-2) {};
			\node[point, label=below:$k$] at (2,-2) {};
			\node[diam, label=left:$\gamma$] at (0,-0.35) {};
			\node[diam, label=left:$-\gamma$] at (0,-1) {};
			\node[diam, label=left:$\delta$] at (1,-1) {};
			\node[diam, label=left:$-\delta$] at (1,-1.65) {};
	}
	=	\tik{
	\hori[->]{0}{1}{1}\straight[->]{2}{1}\hori{0}{0}{2}\straight{1}{0}
	\node[point, label=above:$i$] at (0,0) {};
			\node[point, label=above:$j$] at (1,0) {};
			\node[point, label=above:$k$] at (2,0) {};
			\node[point, label=below:$i$] at (0,-2) {};
			\node[point, label=below:$j$] at (1,-2) {};
			\node[point, label=below:$k$] at (2,-2) {};	
			\node[diam, label=left:$\gamma+\delta$] at (0,-0.35) {};
			\node[diam, label=left:$-\delta$] at (0,-1) {};
			\node[diam, label=left:$-\gamma$] at (0,-1.65) {};
	}
	+	\tik{
	\hori[->]{0}{1}{1} \straight[->]{2}{1}\hori{1}{0}{1}\straight{0}{0}
	\node[point, label=above:$i$] at (0,0) {};
			\node[point, label=above:$j$] at (1,0) {};
			\node[point, label=above:$k$] at (2,0) {};
			\node[point, label=below:$i$] at (0,-2) {};
			\node[point, label=below:$j$] at (1,-2) {};
			\node[point, label=below:$k$] at (2,-2) {};	
				\node[diam, label=left:$\delta$] at (1,-0.35) {};
			\node[diam, label=left:$-\delta$] at (1,-1) {};
			\node[diam, label=left:$\gamma$] at (0,-1) {};
			\node[diam, label=left:$-\gamma$] at (0,-1.65) {};
	}
\end{equation}

\begin{equation}\label{eq:labNEll3G}\tag{t4T$_{e\ell\ell}2$}
\tik{
\tell{0}{1}{\pm}
\draw[zell] (0,-1)--(0,-2);
\hori[->]{1}{1}{1}
\straight[->]{1}{1}
\straight{2}{0}\node[point, label=above:$i$] at (1,0) {};
\node[point, label=below:$i$] at (1,-2) {};
\node[point, label=above:$j$] at (2,0) {};
\node[point, label=below:$j$] at (2,-2) {};
\node[diam, label=left:$\gamma$] at (1,-1.35) {};
\node[diam, label=left:$-\gamma$] at (1,-1.65) {};
}
+
\tik{
\tell{0}{2}{\pm}
\draw[zell] (0,-1)--(0,-2);
\hori[->]{1}{1}{1}
\straight[->]{1}{1}
\straight{2}{0}
\straight{1}{0}
\node[point, label=above:$i$] at (1,0) {};
\node[point, label=below:$i$] at (1,-2) {};
\node[point, label=above:$j$] at (2,0) {};
\node[point, label=below:$j$] at (2,-2) {};
\node[diam, label=left:$\gamma$] at (1,-1.35) {};
\node[diam, label=left:$-\gamma$] at (1,-1.65) {};
}
=
\tik{
\tell[->]{1}{1}{\pm}
\draw[zell] (0,0)--(0,-1);
\hori{1}{0}{1}
\straight[->]{2}{1}
\node[point, label=above:$i$] at (1,0) {};
\node[point, label=below:$i$] at (1,-2) {};
\node[point, label=above:$j$] at (2,0) {};
\node[point, label=below:$j$] at (2,-2) {};
\node[diam, label=left:$\gamma$] at (1,-0.35) {};
	\node[diam, label=left:$-\gamma$] at (1,-0.65) {};
}
+
\tik{
\tell[->]{1}{2}{\pm}
\draw[zell] (0,0)--(0,-1);
\hori{1}{0}{1}
\straight[->]{2}{1}
\straight[->]{1}{1}
\node[point, label=above:$i$] at (1,0) {};
\node[point, label=below:$i$] at (1,-2) {};
\node[point, label=above:$j$] at (2,0) {};
\node[point, label=below:$j$] at (2,-2) {};
\node[diam, label=left:$\gamma$] at (1,-0.35) {};
\node[diam, label=left:$-\gamma$] at (1,-0.65) {};
}
\end{equation}

\begin{equation}\label{eq:labCt}\tag{tC$_{e\ell\ell}$}
\sum_i
\tik{
\tell{0}{1}{\pm}
\node[point, label=above:$i$] at (1,0) {};
\node[point, label=below:$i$] at (1,-1) {};
}
=
0
\end{equation}

One can notice that labels sum to $\mathbf{0}$ on each strand of all the above diagrams. 

\medskip

We are now ready to define the $\CD(\KK)$-module $\CD_{e\ell\ell}^{\mathbf{\Gamma}}(\KK)$. 
\begin{itemize}
\item In arity $n$, objects of $\CD_{e\ell\ell}^{\mathbf{\Gamma}}(\KK)$ are just labellings $\{1,\dots,n\}\to\Gamma$ 
up to a global shift: $\on{Ob}(\CD_{e\ell\ell}^{\mathbf{\Gamma}}(\KK))(n)=\Gamma^n/\Gamma$; 
\item The $*$-module structure is given as follows on objects: for every $i$, 
$\circ_i:\Gamma^n\to\Gamma^{n+m-1}$ is the partial diagonal 
\[
(\alpha_1,\dots,\alpha_n)\longmapsto (\alpha_1,\dots,\alpha_{i-1},
\underbrace{\alpha_i,\dots,\alpha_i}_{m~\textrm{times}},\alpha_{i+1},\dots,\alpha_n)\,;
\]
\item Given two objects $[\underline{\alpha}]=[\alpha_1,\dots,\alpha_n]$ and 
$[\underline{\beta}]=[\beta_1,\dots,\beta_n]$ in arity $n$, the $\kk$-vector space of morphisms 
from $[\underline{\alpha}]$ to $[\underline{\beta}]$ in $\CD_{e\ell\ell}^{\mathbf{\Gamma}}(\KK)$ is 
the vector space of horizontal $\Gamma$-chord diagrams such that, for every $i$, 
the sum of labels on the $i$-th strand is $\beta_i-\alpha_i$; 
\item The $\CD(\kk)$-module structure on morphisms is the exact same as the one for 
$\mathcal{\hat U}(\bar\t^{\Gamma}_{e\ell\ell}(\kk))$. 
\end{itemize}
Moreover, $\CD_{e\ell\ell}^{\mathbf{\Gamma}}(\kk)$ carries an action of $\overline{\Gamma}$, by translation 
on the labelling of objects.

\medskip

For every surjective group morphism $\rho:\Gamma\twoheadrightarrow\Gamma'$, 
the morphism of Proposition \ref{prop:comparison} 
gives rise to a $\Gamma$-equivariant surjective morphism 
\[
\CD_{e\ell\ell}^{\mathbf{\Gamma}}(\kk)\to\mathbf{CD^{\Gamma'}_{e\ell\ell}}(\kk)
\]
which exhibits $\mathbf{CD^{\Gamma'}_{e\ell\ell}}(\kk)$ as the quotient 
$\overline{\ker(\rho)}\backslash\CD_{e\ell\ell}^{\mathbf{\Gamma}}(\kk)$. 

\begin{example}[Notable arrows in $\CD_{e\ell\ell}^{\mathbf{\Gamma}}(\kk)(2)$]
Assume that $\Gamma=\Z/M\Z \times \Z/N\Z$. 
In addition to the arrows of $\mathcal{\hat U}(\bar\t_{1,2}^{\Gamma}(\kk))$, we also have, 
in $\CD_{e\ell\ell}^{\mathbf{\Gamma}}(\kk)(2)$, 
\begin{center}
$I^{1,2}_{e\ell\ell}=$
\begin{tikzpicture}[baseline=(current bounding box.center)]
\tikzstyle point=[circle, fill=black, inner sep=0.05cm]
 \node[point, label=above:$1_\0$] at (1,1) {};
 \node[point, label=below:$1_\alpha$] at (1,-0.25) {};
 \node[point, label=above:$2_\0$] at (2.5,1) {};
 \node[point, label=below:$2_\0$] at (2.5,-0.25) {};
 \node[diam, label=right:$\alpha$] at (1,0.5) {};
 \draw[->,thick] (1,1) .. controls (1,0) and (1,0).. (1,-0.20); 
 \draw[->,thick] (2.5,1) .. controls (2.5,0.25) and (2.5,0.5).. (2.5,-0.20);
\end{tikzpicture}
\qquad and \qquad
$J^{1,2}_{e\ell\ell}=$
\begin{tikzpicture}[baseline=(current bounding box.center)]
\tikzstyle point=[circle, fill=black, inner sep=0.05cm]
 \node[point, label=above:$1_\0$] at (1,1) {};
 \node[point, label=below:$1_{\beta}$] at (1,-0.25) {};
 \node[point, label=above:$2_\0$] at (2.5,1) {};
 \node[point, label=below:$2_\0$] at (2.5,-0.25) {};
  \node[diam, label=right:$\beta$] at (1,0.5) {};
 \draw[->,thick] (1,1) .. controls (1,0) and (1,0).. (1,-0.20); 
 \draw[->,thick] (2.5,1) .. controls (2.5,0.25) and (2.5,0.5).. (2.5,-0.20);
\end{tikzpicture}
\end{center}
recalling that $\alpha=(\bar1,\bar0)$ and $\beta=(\bar0,\bar1)$. 
\end{example}
Let us introduce $\underline{I}^{1,2}_{e\ell\ell}:=I^{1,2}_{e\ell\ell}\alpha_1$ and 
$\underline{J}^{1,2}_{e\ell\ell}:=J^{1,2}_{e\ell\ell}\beta_1$, that are automorphisms of $1_\02_\0$ in the semi-direct product groupoid 
$\mathbf{CD}_{e\ell\ell}^{{\mathbf{\Gamma}}}(\kk)(2)\rtimes(\Gamma^2/\Gamma)$. Then, by definition, for every $\gamma=(\bar{p},\bar{q})\in\Gamma$,
\[
\on{Ad}\big((\underline{I}^{1,2}_{e\ell\ell})^p(\underline{J}^{1,2}_{e\ell\ell})^q\big)(t_{12}^\0)=t_{12}^\gamma\,.
\]
\begin{notation}
For later purposes, we also introduce the notation 
\begin{center}
$X^{1,2}_{e\ell\ell}=x_1 \cdot\mathrm{Id}_{1_\02_\0}=$
\begin{tikzpicture}[baseline=(current bounding box.center)]
\tikzstyle point=[circle, fill=black, inner sep=0.05cm]
 \node[point, label=above:$1_\0$] at (1,0) {};
 \node[point, label=below:$1_\0$] at (1,-1) {};
 \node[point, label=above:$2_\0$] at (2,0) {};
 \node[point, label=below:$2_\0$] at (2,-1) {};
 \tell{0}{1}{+}
 \draw[->,thick] (1,0) .. controls (1,-0.5) and (1,-0.5).. (1,-0.95); 
 \draw[->,thick] (2,0) .. controls (2,-0.5) and (2,-0.5).. (2,-0.95);
\end{tikzpicture}
\qquad and\qquad 
$Y^{1,2}_{e\ell\ell}=y_1 \cdot\mathrm{Id}_{1_\02_\0}=$
\begin{tikzpicture}[baseline=(current bounding box.center)]
\tikzstyle point=[circle, fill=black, inner sep=0.05cm]
 \node[point, label=above:$1_\0$] at (1,0) {};
 \node[point, label=below:$1_\0$] at (1,-1) {};
 \node[point, label=above:$2_\0$] at (2,0) {};
 \node[point, label=below:$2_\0$] at (2,-1) {};
 \tell{0}{1}{-}
 \draw[->,thick] (1,0) .. controls (1,-0.5) and (1,-0.5).. (1,-0.95); 
 \draw[->,thick] (2,0) .. controls (2,-0.5) and (2,-0.5).. (2,-0.95);
\end{tikzpicture}
\end{center}
\end{notation}


\section{Parenthesized ellipsitomic chord diagrams}

There is a $\Gamma$-equivariant morphism of modules $\omega_3:\mathbf{Pa^\Gamma}\to\on{Ob}(\CD_{e\ell\ell}^{\mathbf{\Gamma}}(\kk))$, 
which forgets the parenthesized permutation (and only remembers the labelling), over the terminal operad morphism 
$\omega_1:\Pa\to*=\on{Ob}(\CD(\kk))$ from \S\ref{sec-pacd}. Hence we can consider the fake pull-back $\PaCD(\KK)$-module
\[
\text{\gls{PaCDeg}}:=\omega_3^\star \CD_{e\ell\ell}^{\mathbf{\Gamma}}(\kk)
\]
of \textit{parenthesized ellipsitomic chord diagrams}, which is still acted on by $\overline{\Gamma}$. 

\begin{remark}\label{PaCD:cyc:rel}
As explained in section \ref{sec-pointings}, there is a map of $\mathfrak{S}$-modules 
$\PaCD(\KK) \longrightarrow \PaCD_{e\ell\ell}^{\mathbf{\Gamma}}(\KK)$ and we keep the same symbol 
for the image in $\PaCD_{e\ell\ell}^{\mathbf{\Gamma}}(\KK)$ an arrows in $\PaCD(\KK)$.  
\begin{center}
$X^{1,2}=1\cdot$ 
\begin{tikzpicture}[baseline=(current bounding box.center)]
\tikzstyle point=[circle, fill=black, inner sep=0.05cm]
 \node[point, label=above:$1_\0$] at (1,1) {};
 \node[point, label=below:$2_\0$] at (1,-0.25) {};
 \node[point, label=above:$2_\0$] at (2,1) {};
 \node[point, label=below:$1_\0$] at (2,-0.25) {};
 \draw[->,thick] (2,1) .. controls (2,1) and (1,-0.20).. (1,-0.20); 
 \draw[->,thick] (1,1) .. controls (1,1) and (2,-0.20).. (2,-0.20);
\end{tikzpicture} \qquad
$H^{1,2}=t_{12}^{\0}\cdot$
\begin{tikzpicture}[baseline=(current bounding box.center)]
\tikzstyle point=[circle, fill=black, inner sep=0.05cm]
 \node[point, label=above:$1_\0$] at (1,1) {};
 \node[point, label=below:$1_\0$] at (1,-0.25) {};
 \node[point, label=above:$2_\0$] at (2,1) {};
 \node[point, label=below:$2_\0$] at (2,-0.25) {};
 \draw[->,thick] (2,1) .. controls (2,0.25) and (2,0.5).. (2,-0.20); 
 \draw[->,thick] (1,1) .. controls (1,0.25) and (1,0.5).. (1,-0.20);
\end{tikzpicture}
\qquad $a^{1,2,3}=1\cdot$
\begin{tikzpicture}[baseline=(current bounding box.center)]
\tikzstyle point=[circle, fill=black, inner sep=0.05cm]
 \node[point, label=above:$(1_\0$] at (1,1) {};
 \node[point, label=below:$1_\0$] at (1,-0.25) {};
 \node[point, label=above:$2_\0)$] at (1.5,1) {};
 \node[point, label=below:$(2_\0$] at (3.5,-0.25) {};
 \node[point, label=above:$3_\0$] at (4,1) {};
 \node[point, label=below:$3_\0)$] at (4,-0.25) {};
 \draw[->,thick] (1,1) .. controls (1,0) and (1,0).. (1,-0.20); 
 \draw[->,thick] (1.5,1) .. controls (1.5,1) and (3.5,-0.20).. (3.5,-0.20);
 \draw[->,thick] (4,1) .. controls (4,0) and (4,0).. (4,-0.20);
\end{tikzpicture}
\end{center}
Assuming again that $\Gamma=\mathbb{Z}/M\mathbb{Z}\times\mathbb{Z}/N\mathbb{Z}$, the following relations hold 
in $\on{End}_{\PaCD^{\mathbf{\Gamma}}_{e\ell\ell}(\kk)(2)\rtimes(\Gamma^2/\Gamma)}(1_\02_\0)$: 
\begin{itemize}
\item $(\underline{I}_{e\ell\ell}^{1,2})^M=\on{Id}_{1_\02_\0}$, 
\item $(\underline{J}_{e\ell\ell}^{1,2})^N=\on{Id}_{1_\02_\0}$, 
\item $\big(\underline{I}_{e\ell\ell}^{1,2}\,,\,\underline{J}_{e\ell\ell}^{1,2}\big)=\on{Id}_{1_\02_\0}$,  
\end{itemize}
These relations allow to unambiguously define, for every $\gamma=(\bar{p},\bar{q})\in\Gamma$, a morphism 
$K_\gamma^{1,2}:1_\02_\0\to 1_\gamma2_\0$ by 
\[
\underline{K}_\gamma^{1,2}:=K_\gamma\gamma_1=(\underline{I}_{e\ell\ell}^{1,2})^p(\underline{J}_{e\ell\ell}^{1,2})^q\,,
\]
so that the assignement $\gamma\mapsto\underline{K}_\gamma$ is multiplicative. 

We also have the following relations in 
$\on{End}_{\PaCD^{\mathbf{\Gamma}}_{e\ell\ell}(\kk)(3)\rtimes(\Gamma^3/\Gamma)}\big((1_\02_\0)3_\0\big)$:
\begin{eqnarray*}
0 & = & X_{e\ell\ell}^{12,3}+\on{Ad}\left(a^{1,2,3}X^{1,23}\right)(X_{e\ell\ell}^{23,1})
		+\on{Ad}\left(X^{12,3}(a^{3,1,2})^{-1}\right)(X_{e\ell\ell}^{31,2})\,, \\
0 & = & Y_{e\ell\ell}^{12,3}+\on{Ad}\left(a^{1,2,3}X^{1,23}\right)(Y_{e\ell\ell}^{23,1})
		+\on{Ad}\left(X^{12,3}(a^{3,1,2})^{-1}\right)(Y_{e\ell\ell}^{31,2})\,, \\
0 & = & \underline{I}_{e\ell\ell}^{12,3}+\on{Ad}\left(a^{1,2,3}X^{1,23}\right)(\underline{I}_{e\ell\ell}^{23,1})
		+\on{Ad}\left(X^{12,3}(a^{3,1,2})^{-1}\right)(\underline{I}_{e\ell\ell}^{31,2})\,, \\
0 & = & \underline{J}_{e\ell\ell}^{12,3}+\on{Ad}\left(a^{1,2,3}X^{1,23}\right)(\underline{J}_{e\ell\ell}^{23,1})
		+\on{Ad}\left(X^{12,3}(a^{3,1,2})^{-1}\right)(\underline{J}_{e\ell\ell}^{31,2})\,, \\
\sum_{\gamma\in\Gamma} \on{Ad}(\underline{K}^{1,2}_{\gamma})(H^{1,2}) & = &
\big[\on{Ad}\left(a^{1,2,3}\right)(X_{e\ell\ell}^{1,23}),\on{Ad}\left(X^{1,2}a^{2,1,3}\right)(Y_{e\ell\ell}^{2,13})\big]\,.
\end{eqnarray*}
\end{remark}

\begin{definition}
The \textit{graded ellipsitomic Grothendieck-Teichm\"uller group} is defined as
\[
\text{\gls{GRTeg}}:=\on{Aut}^+_{\tmop{OpR}\mathbf{Cat}(\mathbf{CoAlg}_\kk)}
\big(\PaCD(\KK),\PaCD^{\mathbf{\Gamma}}_{e\ell\ell}(\KK))\big)^{\mathbf{\Gamma}}
\]
\end{definition}
Recall that there is an isomorphism 
\[
\on{Aut}^+_{\tmop{OpR}\mathbf{Cat}(\mathbf{CoAlg}_\kk)}
\big(\PaCD(\KK),\PaCD_{e\ell\ell}^{\mathbf{\Gamma}}(\KK))\big)^{\mathbf{\Gamma}}
\simeq
\on{Aut}^+_{\tmop{OpR}\mathbf{Grpd}_\kk}\big(G\PaCD(\KK),
G\PaCD_{e\ell\ell}^{\mathbf{\Gamma}}(\KK))\big)^{\mathbf{\Gamma}}\,.
\]
For every group surjective morphism $\rho:\Gamma\to\Gamma'$, and every $a,b,c,d\in\kk$ such that 
$ad-bc=|\ker(\rho)|$, using the fact that 
$\overline{\ker(\rho)}\backslash\PaCD_{e\ell\ell}^{\mathbf{\Gamma}}(\KK)
\simeq \PaCD_{e\ell\ell}^{\mathbf{\Gamma'}}(\KK)$, 
we obtain a group morphism 
\[
\GRT_{e\ell\ell}^{\mathbf{\Gamma}}(\kk)\longrightarrow \GRT_{e\ell\ell}^{\mathbf{\Gamma'}}(\kk)\,.
\]


\section{Ellipsitomic associators}

We now fix $\Gamma:=\Z/M\Z \times \Z/N\Z$.
\begin{definition}
The set of \textit{ellipsitomic $\kk$-associators} is  
\[
\text{\gls{Asseg}}:=\on{Iso}^+_{\tmop{OpR}\mathbf{Grpd}_\kk}\Big(
\big(\widehat{\PaB}(\KK),\widehat{\PaB}_{e\ell\ell}^{\mathbf{\Gamma}}(\KK)\big)\,,\,
\big(G\PaCD(\KK),G\PaCD_{e\ell\ell}^{\mathbf{\Gamma}}(\KK)\big)\Big)^{\mathbf{\Gamma}}\,.
\]
\end{definition}

\begin{theorem} \label{th:ass:tell}
There is a one-to-one correspondence between the set 
$\Ell^{\mathbf{\Gamma}}(\KK)$ and the set \gls{bAsseg} consisting 
of quadruples $(\mu,\varphi,A,B)\in\on{Ass}(\KK)\times\on{exp}\left(\hat{\bar{\t}}^{\Gamma}_{1,2}(\KK)\right)$ such that, 
for $\underline{A}:=A\alpha_1$ and $\underline{B}:=B\beta_1$, 
the following relations hold in $\on{exp}(\hat{\bar{\t}}^{\Gamma}_{1,3}(\KK))\rtimes(\Gamma^3/\Gamma)$: 
\begin{equation}\label{def:tell:ass:G:1}
1=\varphi^{1, 2, 3}\underline{A}^{1, 23} e^{-\mu(\bar{t}^\0_{12} + \bar{t}^\0_{13})/2}
\varphi^{2, 3, 1}\underline{A}^{2, 31}e^{-\mu(\bar{t}^\0_{23} + \bar{t}^\0_{12})/2}
\varphi^{3, 1, 2}\underline{A}^{3, 12}e^{-\mu (\bar{t}^\0_{31} + \bar{t}^\0_{32})/2} \,,
\end{equation}
\begin{equation}\label{def:tell:ass:G:2}
1=\varphi^{1, 2, 3}\underline{B}^{1, 23} e^{-\mu(\bar{t}^\0_{12} + \bar{t}^\0_{13})/2}
\varphi^{2, 3, 1}\underline{B}^{2, 31}e^{-\mu(\bar{t}^\0_{23} + \bar{t}^\0_{12})/2}
\varphi^{3, 1, 2}\underline{B}^{3, 12}e^{-\mu (\bar{t}^\0_{31} + \bar{t}^\0_{32})/2} \,,
\end{equation}
\begin{equation}\label{def:tell:ass:G:3}
e^{\mu \bar{t}^\0_{12}}= 
\big(\varphi^{1, 2, 3}\underline{A}^{1, 23}\varphi^{3,2,1}\,,\,
e^{- \mu \bar{t}^\0_{12} / 2}\varphi^{2,1,3}\underline{B}^{2,13}\varphi^{3,2,1}e^{-\mu \bar{t}^\0_{12} / 2}\big)\,.
\end{equation}
\end{theorem}
\begin{proof}
Let $(F,G)$ be an ellipsitomic associator. 
We have already seen that the choice of the operad isomorphism $F$ corresponds bijectively to the choice of an 
element $(\mu,\varphi)\in \on{Ass}(\KK)$. 
From the presentation of $\PaB_{e\ell\ell}^{\mathbf{\Gamma}}$, we know that $G$ is uniquely determined by the images 
of $A^{1,2}\in \on{Hom}_{\PaB_{e\ell\ell}^{\mathbf{\Gamma}}(\KK)(2)}(1_\02_\0,1_\alpha2_\0)$ 
and $B^{1,2}\in \on{Hom}_{\PaB_{e\ell\ell}^{\mathbf{\Gamma}}(\KK)(2)}(1_\02_\0,1_\beta2_\0)$.
There are elements $A,B\in \exp(\hat{\bar{\t}}_{1,2}^\Gamma(\kk))$ such that 
\begin{itemize}
\item $G(A^{1,2})=A \cdot I_{e\ell\ell}^{1,2}$;
\item $G(B^{11,2})=B \cdot J_{e\ell\ell}^{1,2}$. 
\end{itemize}
These elements must satisfy relations \eqref{def:tell:ass:G:1}, \eqref{def:tell:ass:G:2} and \eqref{def:tell:ass:G:3}, 
that are images of \eqref{def:PaB:ell:G:1}, \eqref{def:PaB:ell:G:2} and \eqref{def:PaB:ell:G:3}. Conversely, if 
\eqref{def:tell:ass:G:1}, \eqref{def:tell:ass:G:2} and \eqref{def:tell:ass:G:3} are satisfied, then $G$ is well-defined. 
\end{proof} 
\begin{remark}\label{rem:ass:tell}
It follows from the alternative presentation of $\PaB_{e\ell\ell}^\Gamma$ (see Theorem \ref{PaB:ell:G:bis}) that 
$\Ell^\Gamma(\KK)$ is also in bijection with the set of 
$(\mu,\varphi,A,B)\in \on{Ass}(\KK)\times\Big(\on{exp}\big(\hat{\bar{\t}}^{\Gamma}_{1,2}(\KK)\big)\Big)^{\times 2}$ 
satisfying 
\begin{equation}\label{def:tell:ass:G:1bis}
\underline{A}^{12, 3} =  
  \varphi^{1, 2, 3}  \underline{A}^{1, 23}\varphi^{3,2,1} e^{-\mu \bar{t}_{12}^{\0}/2}
  \varphi^{2, 1, 3}  \underline{A}^{2, 13}\varphi^{3,1,2}e^{-\mu \bar{t}_{12}^{\0}/2}
\end{equation}
\begin{equation}\label{def:tell:ass:G:2bis}
\underline{B}^{12, 3} =  
  \varphi^{1, 2, 3}  \underline{B}^{1, 23} \varphi^{3,2,1} e^{-\mu \bar{t}_{12}^{\0}/2}
  \varphi^{2, 1, 3}  \underline{B}^{2, 13} \varphi^{3,1,2}e^{-\mu \bar{t}_{12}^{\0}/2}
\end{equation}
\begin{equation}\label{def:tell:ass:G:3bis}
\varphi^{1,2,3} e^{\mu \bar{t}_{23}^\0} \varphi^{3,2,1} = 
\big(\underline{A}^{12, 3}\varphi^{1,2,3}(\underline{A}^{1, 23})^{- 1}\varphi^{3,2,1}\,,\,(\underline{B}^{12, 3})^{-1}\big)
\end{equation}
\end{remark}

\medskip

As before, if we are given a surjective group morphism 
\[
\rho\,:\,\Gamma\twoheadrightarrow\Gamma':=\mathbb{Z}/M'\mathbb{Z}\times\mathbb{Z}/N'\mathbb{Z}\,,
\]
then, given $a,b,c,d\in\kk$ as in Remark \ref{rem:comparison}, there is a bitorsor morphism 
\[
\big(\widehat{\GT}^{\mathbf{\Gamma}}(\kk),\Ell^{\mathbf{\Gamma}}(\KK),\GRT^{\mathbf{\Gamma}}(\kk)\big)
\longrightarrow
\big(\widehat{\GT}^{\mathbf{\Gamma'}}(\kk),\Ell^{\mathbf{\Gamma'}}(\KK),\GRT^{\mathbf{\Gamma'}}(\kk)\big)\,.
\]

\medskip

In chapter \ref{Section6} we prove that ellipsitomic associators (with complex coefficients) do exist. 
\begin{remark}
Drinfeld's argument in \cite{DrGal} (see also \cite{BN}) that shows how to deduce the existence of an associator over 
$\mathbb{Q}$ from the existence of an associator over $\mathbb{C}$ can be repeated \textit{verbatim} for 
ellipsitomic associators. We leave the details for future work. 
\end{remark}

\chapter{The KZB ellipsitomic associator}
\label{Section6}


In this chapter, $\Gamma$ still denotes the abelian group 
$\Gamma=\Z/M\Z \times \Z/N\Z$ where $M,N\geq 1$ are two integers. 

\medskip

Recall from Theorem \ref{th:ass:tell} that the set of ellipsitomic associators can 
be regarded, either as the set of $\overline{\Gamma}$-equivariant 
$\widehat\PaB(\kk)$-module isomorphisms 
$\widehat\PaB_{e\ell\ell}^{\mathbf{\Gamma}} (\kk) \longrightarrow 
G\PaCD_{e\ell\ell}^{\mathbf{\Gamma}}(\kk)$ which are the identity on objects, 
or as quadruples $(\lambda,\Phi,A_+,A_-)$, where $(\lambda,\Phi)\in \on{Ass}(\KK)$ 
and $A_{\pm}\in \on{exp}(\hat{\bar{\t}}^{\Gamma}_{1,2}(\KK))$, satisfying 
relations \eqref{def:tell:ass:G:1}, \eqref{def:tell:ass:G:2}, \eqref{def:tell:ass:G:3}. 
The following result tells us that the set $\on{Ell}^\Gamma(\mathbb{C})$ is not empty. 
We write $\on{Ell}^\Gamma_{\on{KZB}}:=\on{Ell}^\Gamma(\mathbb{C})
\times_{\on{Ass}(\mathbb{C})} \{ 2\pi \i, \Phi_{\on{KZ}}\}$.

\begin{theorem}\label{theorem:twistedKZBass}
There is an analytic map 
\begin{eqnarray*}
  \mathfrak{H} & \longrightarrow & \on{Ell}^\Gamma_{\on{KZB}} \\
  \tau & \longmapsto & e^{\Gamma} (\tau)=(A^{\Gamma}(\tau),B^{\Gamma}(\tau)).
\end{eqnarray*}
In particular, for each $\tau \in \mathfrak{H}$, where $\mathfrak{H}$ is the upper half-plane, 
the element $(2\pi \i,\Phi_{\on{KZ}},A^{\Gamma}(\tau),B^{\Gamma}(\tau))$ is 
an ellipsitomic $\C$-associator (i.e. it belongs to $\on{Ell}^\Gamma(\mathbb{C})$).
\end{theorem}
The rest of this chapter is devoted to the proof of the above theorem. 


\section{The pair $e^{\Gamma} (\tau)$}\label{section:pair}

We adopt the convention for monodromy actions of \cite[Appendix A]{CaGo}.
First of all, recall that $\bar{\mathfrak{t}}_{1,2}^{\Gamma}$ is the Lie $\C$-algebra generated by 
$x:=\bar x_1$, $y:=\bar y_2$ and $t^\alpha:=\bar t_{12}^\alpha$, for $\alpha\in\Gamma$, such that 
$[x,y]=\sum_{\alpha\in\Gamma} t^\alpha$. 
We define the \textit{KZB ellipsitomic associator} as the couple 
\[
e^\Gamma(\tau):=\big(A^{\Gamma}(\tau),B^{\Gamma}(\tau)\big)\in 
\on{exp}(\hat{\bar{\mathfrak{t}}}_{1,2}^{\Gamma})\times \on{exp}(\hat{\bar{\mathfrak{t}}}_{1,2}^{\Gamma})
\]
consisting in the renormalized holonomies along the straight paths from $0$ to $1/M$ and from $0$ to $\tau/N$, 
respectively, of the differential equation
\begin{equation}\label{equationKZBonevariable}
 J'(z) =  \left(y + \sum\limits_{\alpha \in \Gamma}\left(e^{-\frac{2\pi\i v_\alpha}{N}\on{ad}(x)}{\frac{\theta(z-\tilde\alpha+\on{ad}(x)|\tau)}{\theta(z-\tilde\alpha|\tau)\theta(\on{ad}(x)|\tau)}}-\frac{1}{\on{ad}(x)}\right)(t^{\alpha})\right)\cdot J(z)\,, 
\end{equation}
with values in the group $\on{exp}(\hat{\bar\t}_{1,2}^{\Gamma})\rtimes \Gamma$ 
(here, $v_\alpha$ is any integer such that $\alpha=(\bar{u_\alpha},\bar{v_\alpha})\in\Gamma$). 

More precisely, 
equation \eqref{equationKZBonevariable} has a unique solution $J(z)$ defined over 
$\{\frac{s_1}{M}+\frac{s_2}{N}\tau\,|\,s_1, s_2\in\mathbb{R},\,s_1\textrm{ or }s_2\in(0,1)\}$ 
and such that 
\[
J(z)\simeq (-z)^{t^{\0}}
\]
at $z\to 0$. 
\begin{remark}\label{branchremark}
We always consider a branch of $\log$ that is defined outside the half line $\mathbb{R}_+\tau$, and we always make sure that the domains of definition 
never contain this half-line. Above, we indeed have that every $z$ in the domain of definition satisfies $-z\notin\mathbb{R}_+\tau$.  
\end{remark}
We define
\[
\underline{A}^{\Gamma}(\tau):= J(z+\frac{1}{M})^{-1}(\bar1,\bar0)J(z)
\in\on{exp}(\hat{\bar\t}_{1,2}^{\Gamma}) \rtimes \Gamma\,.
\]
Then the $A$-ellipsitomic KZB associator \gls{Ag} is the $\on{exp}(\hat{\bar\t}_{1,2}^{\Gamma})$-component 
of $\underline{A}^{\Gamma}(\tau)$: 
\[
A^{\Gamma}(\tau):=\underline{A}^{\Gamma}(\tau)(-\bar1,\bar0)
								=J(z+\frac{1}{M})^{-1}(\bar1,\bar0)\cdot J(z)\in \on{exp}(\hat{\bar\t}_{1,2}^{\Gamma})\,.
\]
In the same way, we define
\[
\underline{B}^{\Gamma}(\tau):= J(z+\frac{\tau}{N})^{-1}e^{-\frac{2\pi\on{i}x}{N}}(\bar0,\bar1)J(z)\,,
\]
and the $B$-ellipsitomic KZB associator \gls{Bg} is then its $\on{exp}(\hat{\bar\t}_{1,2}^{\Gamma})$-component: 
\[
B^{\Gamma}(\tau):=\underline{B}^{\Gamma}(\tau)(\bar0,-\bar1)
								=J(z+\frac{\tau}{N})^{-1}e^{-\frac{2\pi\on{i}x}{N}}(\bar0,\bar1)\cdot J(z) \in \on{exp}(\hat{\bar\t}_{1,2}^{\Gamma})\,.
\]


\section{The ellipsitomic KZB system}\label{sec:twistedsol}

Recall from \cite{CaGo} the ellipsitomic KZB system, that is a several variables version of the differential equation from the previous subsection: 
\begin{equation}\label{equation-system}
\begin{cases}
\partial_{z_i} F (\mathbf{z} | \tau) = K_i (\mathbf{z} | \tau) F (\mathbf{z} | \tau) & (i=1,\dots,n) \\
\partial_{\tau} F (\mathbf{z} | \tau) =  \Delta (\mathbf{z} | \tau) F (\mathbf{z} | \tau) & 
\end{cases}
\end{equation}
Here $F(\mathbf{z} | \tau)$ is a holomorphic function 
$(\mathbb{C}^n \times \mathfrak{H})- \on{Diag}_{n, \Gamma} \supset U \longrightarrow \mathbf{G}^{\Gamma}_n$, 
\[
\on{Diag}_{n, \Gamma}=
\left\{(\mathbf{z}|\tau)\in\mathbb{C}^n\times\mathfrak{H}\,|\,z_i-z_j\in \frac1M\mathbb{Z}+\frac{\tau}{N}\mathbb{Z}\textrm{ for }i\neq j\right\}\,,
\] 
and the $\mathbb{C}$-group $\mathbf{G}^{\Gamma}_n$, 
$K_i (\mathbf{z} | \tau)\in\exp(\hat{\t}_{1,n}^\Gamma)\subset\mathbf{G}^{\Gamma}_n$, and $ \Delta (\mathbf{z} | \tau)\in\mathbf{G}^{\Gamma}_n$, 
are defined in \cite{CaGo}. 

If one denotes $z_{ij}=z_i-z_j$, then
\begin{eqnarray*}
  K_i (\mathbf{z} | \tau) & = & - y_i + \underset{j ;
  j \neq i}{\sum}  \underset{\alpha \in \Gamma}{\sum} \left( e^{- 2 \pi\i a
  \tmop{ad} (x_i)} \frac{\theta (z_{ij} - \tilde{\alpha} + \tmop{ad}
  (x_i) | \tau)}{\theta (z_{ij} - \tilde{\alpha} | \tau) \theta
  (\tmop{ad} (x_i) | \tau)} - \frac{1}{\tmop{ad} (x_i)} \right)
  (t^{\alpha}_{i j})\\
  & = & \underset{j ; j \neq i}{\sum}  \underset{\alpha \in \Gamma}{\sum}
  \left( \frac{1}{\tmop{ad} (x_i)} + \frac{t^{\alpha}_{i j}}{z_{ij} - \tilde{\alpha}} - \frac{1}{\tmop{ad} (x_i)} \right)
  (t^{\alpha}_{i j}) + O (1)\\
  & = & \underset{j ; j \neq i}{\sum}  \underset{\alpha \in \Gamma}{\sum}
  \frac{t^{\alpha}_{i j}}{z_{ij} - \tilde{\alpha}} + O (1) =
  \underset{j ; j \neq i}{\sum}  \underset{\alpha \in \Gamma}{\sum}
  \frac{t^{\alpha}_{i j}}{z_{ij} - \frac{a_0}{M}} + O (1)\,,
\end{eqnarray*}
where $O(1)$ stands for a holomorphic function on $\mathbb{C}^n \times \mathfrak{H}$. 
Then it follows directly from the definition of $\Delta (\mathbf{z} | \tau)$ in \cite[\S3.3]{CaGo} that, for $|z_{ij}|\ll 1$, 
\begin{eqnarray*}
  \Delta (\mathbf{z} | \tau) & = & -\frac{1}{2\pi\on{i}}
  \left( \Delta_0 + \frac{1}{2} \underset{s \geqslant 0}{\sum}
  \underset{\gamma\in \Gamma}{\sum} A_{s, \gamma} (\tau) \left( \delta_{s,
  \gamma} - 2\underset{i < j}{\sum} \tmop{ad} (x_i)^s
  (t^{-\gamma}_{i j})  \right)  \right) + o (1)\,,
\end{eqnarray*}
where $o(1)$ denotes a function of the form $\sum_{ij}z_{ij}f_{ij}(\mathbf{z}|\tau)$, 
with $f_{ij}$'s being holomorphic on $\mathbb{C}^n \times \mathfrak{H}$. 

\begin{remark}
In chapter \ref{Section7} we study the modularity properties of the coefficients $A_{s, \gamma} (\tau)$. 
\end{remark}

We now determine a particular solution $F_{\tau_0,n,\Gamma}$ of the ellipsitomic KZB system \eqref{equation-system}, associated with every $\tau_0\in\mathfrak{H}$. 

Let $D_{n,\Gamma} \subset (\mathbb{C}^n \times \mathfrak{H}) - \on{Diag}_{n,\Gamma}$ be defined as 
\[
\left\{(\mathbf{z}, \tau) \in \mathbb{C}^n \times
\mathfrak{H} |z_i = a_i + b_i \tau, a_i, b_i \in \mathbb{R}, a_1< a_2< \cdots  <
a_n < a_1 + \frac{1}{M}, b_n< \cdots  <
b_1 < b_n + \frac{1}{N} \right\}\,,
\]
which is simply connected. 
A solution of the ellipsitomic KZB system on this domain is then unique, up to right multiplication by a constant element in $\mathbf{G}^{\Gamma}_n$. 
Then, by applying \cite[Appendix A, Proposition 85]{CEE}
with $u_{n-1} = z_{n1}$, $u_{n-2} = z_{(n-1)1}/z_{n1}$,..., $u_{1}=z_{21}/z_{31}$, 
we obtain a unique solution $F_{\tau_0,n,\Gamma}$ with the expansion 
\[
F_{\tau_0,n,\Gamma}(\zz|\tau_0) \simeq z_{21}^{t^\mathbf{0}_{12}}z_{31}^{t^\mathbf{0}_{13} + t^\mathbf{0}_{23}}
\cdots  z_{n1}^{t^\mathbf{0}_{1n} + \cdots  + t^\mathbf{0}_{n-1,n}}
\]
in the region $|z_{21}| \ll |z_{31}| \ll \cdots \ll |z_{n1}| \ll 1$, $(\zz,\tau_0)\in D_{n,\Gamma}$. 
The sign $\simeq$ means here that any of the ratios of both sides is of the form 
\[
1+\sum_{k>0} \sum_{i,a_{1},\dots,a_{n-1}}r_{k}^{i,a_{1},\dots,a_{n-1}}(u_{1},\dots,u_{n-1}|\tau_0)\,,
\]
where the second sum is finite with $a_{i}\geq 0$, $i\in \{1,\dots,n-1\}$, $r_{k}^{i,a_{1},\dots,a_{n-1}}(u_{1},\dots,u_{n-1}|\tau_0)$ has degree $k$, 
and is $O(u_{i}(\on{log}u_{1})^{a_{1}}\dots(\on{log}u_{n-1})^{a_{n-1}})$.

\medskip

In the remainder of this chapter, we keep $\tau$ fixed and consider $F_{\tau,n}(\zz):=F_{\tau,n,\Gamma}(\zz|\tau)$, which is a solution of the first line of 
the ellipsitomic KZB system \eqref{equation-system}, defined on 
\[
D_{\tau,n,\Gamma}:=\{\zz\in\mathbb{C}^n|(\zz,\tau)\in D_{n,\Gamma}\}\,,
\]
and taking its values in $\exp(\hat{\t}_{1,n}^\Gamma)\subset\mathbf{G}^{\Gamma}_n$. 



\section{Generators for the group $\on{B}_{1,n}^\Gamma$}

Let us define, for $(\mathbf{z}_0,\tau)\in D_{n,\Gamma} $, the group 
\gls{B1nG}$:=\pi_1\left(\textrm{Conf}(E_{\tau,\Gamma},n,\Gamma)/\mathfrak{S}_n,[\zz_0]\right)$, 
and recall that $\on{B}_{1,n}=\pi_1\left(\textrm{Conf}(E_{\tau,\Gamma},n)/\mathfrak{S}_n,[\zz_0]\right)$. Now, since the canonical surjective map 
\[
\textrm{Conf}(E_{\tau,\Gamma},n,\Gamma)/\mathfrak{S}_n\twoheadrightarrow \textrm{Conf}(E_{\tau,\Gamma},n)/\mathfrak{S}_n
\]
defines a $\Gamma$-covering, then  $\on{B}_{1,n}^\Gamma=\ker(\rho)$, where $\rho: \on{B}_{1,n}\to\Gamma$ sends $\sigma_i$ to 
$\mathbf{0}=(\bar{0},\bar{0})$, $X_i$ to $(\bar{1},\bar{0})$ and $Y_i$ to $(\bar{0},\bar{1})$. We let $A_i$ (resp.~$B_i$) be the class of the path 
given by $[0,1]\ni t\mapsto\zz_0-\frac{t}{M}\sum_{j =i}^n \delta_j$ (resp.~$[0,1]\ni t\mapsto\zz_0-\frac{t}{N}\tau\sum_{j =i}^n \delta_j$), so that 
$X_i=A_i^{-1}A_{i+1}$ (resp.~$Y_i=B_i^{-1}B_{i+1}$). It then follows from the geometric description of $\on{B}_{1,n}^\Gamma$ that 
$A_i^M$, $B_i^N$ ($i=1,\dots,n$) and 
\[
\sigma_{i}^{(\bar p,\bar q)}:=X_{i}^{p}Y_i^{q}\sigma_{i}Y_{i+1}^{-q}X_{i+1}^{-p}\qquad \big(1\leq p\leq M,~1\leq q\leq N\big)
\]
are generators of $\on{B}_{1,n}^\Gamma$. Similarly, $A_i^M$, $B_i^N$ ($i=1,\dots,n$) and  
\[
P_{ij}^{(\bar p,\bar q)}:=X_i^{p}Y_i^{q}P_{ij}Y_i^{-q}X_i^{-p}\qquad \big(i<j,~1\leq p\leq M,~1\leq q\leq N\big)
\]
generate $\on{PB}_{1,n}^\Gamma$. 

We denote with the same symbols $A_i^M$, $B_i^N$, $\sigma_{i}^\alpha$ and $P_{ij}^\alpha$ ($\alpha\in\Gamma, i=1,...,n$) for the projections 
of these elements to $\overline{\on{B}}_{1,n}^\Gamma:=\pi_1\left(\textrm{C}(E_{\tau,\Gamma},n,\Gamma)/\mathfrak{S}_n,[\zz_0]\right)$.


\section{The monodromy morphism $\mu_n: \on{B}_{1,n} \to \exp(\hat{\t}_{1,n}^\Gamma)\rtimes(\Gamma^n \rtimes \mathfrak{S}_n)$}\label{sec:twistedsol2}

Recall from \cite[\S3.1]{CaGo} the moduli space 
\[
\mathcal{M}_{1,n}^\Gamma:=(\mathbb{Z}^n)^2\rtimes \on{SL}_2^\Gamma \backslash \big((\mathbb{C}^n\times\mathfrak{H})-\on{Diag}_{n,\Gamma}\big)
\]
of $\Gamma$-structured elliptic curves with $n$ ordered marked points, where 
\[
\on{SL}_2^\Gamma:=
\left\{\begin{pmatrix}a&b\\c&d\end{pmatrix}\in\on{SL}_2(\Z)\,\big|\,
a \equiv 1~\mathrm{mod}~M,d \equiv 1~\mathrm{mod}~N, b\equiv0~\mathrm{mod}~N~\mathrm{and}~c\equiv0~\mathrm{mod}~M\right\}\,. 
\]

The ellipsitomic KZB system \eqref{equation-system} can be used to define a flat ${\mathbf G}^\Gamma_n$-bundle 
$(\mathcal P_{n,\Gamma},\nabla_{n,\Gamma})$ on $\mathcal{M}_{1,n}^\Gamma$ (see \cite[Theorem 3.9 \& Theorem 3.12]{CaGo}), that descends to 
a flat ${\mathbf G}^\Gamma_n\rtimes(\Gamma^n\rtimes\mathfrak{S}_n)$-bundle
$(\mathcal P_{(\Gamma),[n]},\nabla_{(\Gamma),[n]})$ on $(\Gamma^n\rtimes\mathfrak{S}_n)\backslash\mathcal{M}_{1,n}^\Gamma$ (see \cite[\S3.5]{CaGo}). 
For every $\tau\in\mathfrak{H}$, $(\mathcal P_{n,\Gamma},\nabla_{n,\Gamma})$ restricts to a flat $\exp(\hat{\t}_{1,n}^\Gamma)$-bundle 
$(\mathcal P_{\tau,n,\Gamma},\nabla_{\tau,n,\Gamma})$ on $\textrm{Conf}(E_{\tau,\Gamma},n,\Gamma)$ (see \cite[Theorem 1.11]{CaGo}), that descends 
to a flat $\exp(\hat{\t}_{1,n}^\Gamma)\rtimes(\Gamma^n\rtimes\mathfrak{S}_n)$-bundle 
$(\mathcal P_{(\tau,\Gamma),[n]},\nabla_{(\tau,\Gamma),[n]})$ 
on $\Gamma^n\rtimes\mathfrak{S}_n\backslash\textrm{Conf}(E_{\tau,\Gamma},n,\Gamma)=\mathfrak{S}_n\backslash\textrm{Conf}(E_{\tau,\Gamma},n)$. 

\medskip

This flat bundle determines a monodromy morphism 
\[
\mu_n^{\zz_0}=\mu^{\zz_0}_{(\tau,\Gamma),[n]}\,:\,\on{B}_{1,n}\longrightarrow\exp(\hat{\t}_{1,n}^\Gamma)\rtimes(\Gamma^n\rtimes\mathfrak{S}_n)\,.
\]
whose restriction to $\on{PB}_{1, n}^{\Gamma}$ takes values in $\exp(\hat{\t}_{1,n}^\Gamma)$. In other words, there is a morphism of short exact sequences
\[
\xymatrix{
	1 \ar[r] &
		\on{PB}_{1, n}^{\Gamma} \ar[r] \ar[d] &
		\on{B}_{1,n} \ar[r] \ar[d]^{\mu_n^{\zz_0}} &
		\Gamma^n \rtimes \mathfrak{S}_n \ar[r] \ar@{=}[d] &
		1 \\
		1 \ar[r] &
		\exp(\hat{\t}_{1,n}^\Gamma) \ar[r] &
		\exp(\hat{\t}_{1,n}^\Gamma)\rtimes(\Gamma^n \rtimes \mathfrak{S}_n) \ar[r] &
		\Gamma^n \rtimes \mathfrak{S}_n \ar[r] &
		1
}
\]
where the first vertical morphism is the monodromy morphism of $\nabla_{\tau,n,\Gamma}$. 
It is important to keep in mind that the monodromy morphism depends on the base point $[\zz_0]$, e.g. for $\zz_0\in D_{\tau,n,\Gamma}$. 

\medskip

Note that every local solution $F$ of (the first line of) the ellipsitomic KZB system \eqref{equation-system} around $\zz_0\in D_{\tau,n,\Gamma}$ 
determines a local $\nabla_{(\tau,\Gamma),[n]}$-flat 
section of $\mathcal P_{(\tau,\Gamma),[n]}$, and thus can be used to compute the monodromy in a way that we explain now (we refer to \cite[Appendix A]{CaGo} 
for more details on our conventions). 

For every loop $\gamma$ based at $[\zz_0]$ in $\textrm{Conf}(E_{\tau,\Gamma},n)/\mathfrak{S}_n$, we consider its unique lift $\tilde\gamma$ 
starting at $\zz_0\in D_{\tau,n,\Gamma}$, and choose a simply connected open neighbourhood $U$ of $\tilde\gamma$ that contains $D_{\tau,n,\Gamma}$. 
Then the solution $F$ extends uniquely to $U$, and we define 
\[
\mu_n^{\zz_0}([\gamma]):=F\big(\zz_0\big)F\big(h_\gamma\cdot\zz_0\big)^{-1}c_{h_\gamma}\,,
\]
where $h_\gamma\in \Gamma^n \rtimes \mathfrak{S}_n$ is such that $\tilde\gamma(1)=h_\gamma\tilde\gamma(0)$, and $c$ is the 
non-abelian $1$-cocycle from \cite{CaGo} defining the underlying principal bundle of the flat connection. 

\medskip

Recall that for any other solution $G$ defined on $U$, there exists $g\in \exp(\hat{\t}_{1,n}^\Gamma)\rtimes(\Gamma^n \rtimes \mathfrak{S}_n)$ such that 
$G(\zz)=F(\zz)g$ for every $\zz\in U$. Hence $\mu_n^{\zz_0}([\gamma]):=G\big(\zz_0\big)G\big(h_\gamma\cdot\zz_0\big)^{-1}c_{h_\gamma}$, and the 
monodromy does not depend on the choice of local solution. 

\begin{example}
Let us consider the domains 
\[
H_{\tau,n,\Gamma} \assign\left\{\mathbf{z}\in \mathbb{C}^n\, |\, 
z_i = a_i + b_i \tau, \, a_i, b_i \in \mathbb{R}, \, b_n< \cdots <b_1 < b_n + \frac{1}{N} \right\} 
\] 
and 
\[
V_{\tau,n,\Gamma} \assign \left\{\mathbf{z} \in \mathbb{C}^n \,|\,
z_i = a_i + b_i \tau, \, a_i, b_i \in\mathbb{R}, \, a_1< a_2< \cdots <a_n < a_1 + \frac{1}{M}\right\}.
\]
Both of these domains are simply connected, and contain $D_{\tau,n,\Gamma}$. We denote $F_H (\mathbf{z})$, resp.~$F_V (\mathbf{z})$, 
the prolongation to $H_{\tau,n,\Gamma}$, resp.~to $V_{\tau,n,\Gamma}$, of a given local solution $F (\mathbf{z})$ defined on $D_{\tau,n,\Gamma}$. 
We then consider 
\[
\underline{A}_i^{\zz_0}:=\mu_n^{\zz_0}(A_i)
=F_H(\zz_0)F_H\left(\zz_0 - \sum_{j=i}^n \frac{\delta_j}{M}\right)^{-1}(-\bar{1},\bar{0})_{i,\ldots,n}
\in \exp(\hat{\t}_{1,n}^\Gamma)\rtimes\Gamma^n
\]
and 
\[
\underline{B}_i^{\zz_0}:=\mu_n^{\zz_0}(B_i)
=F_V(\zz_0)F_V\left(\zz_0 - \tau\,\sum_{j=i}^n \frac{\delta_j}{N}\right)^{-1}e^{\frac{2\pi\i}{N}( x_i + ... + x_n)}(\bar{0},-\bar{1})_{i,\ldots,n}
\in \exp(\hat{\t}_{1,n}^\Gamma)\rtimes\Gamma^n\,.
\]
We also consider the projections of these elements on the first factor $\exp(\hat{\t}_{1,n}^\Gamma)$:
\[
A_i^{\zz_0}:=\underline{A}_i^{\zz_0}(\bar{1},\bar{0})_{i,\ldots,n}
=F_H(\zz_0)F_H\left(\zz_0 - \sum_{j=i}^n \frac{\delta_j}{M}\right)^{-1}\in \exp(\hat{\t}_{1,n}^\Gamma)
\]
and 
\[
B_i^{\zz_0}:=\underline{B}_i^{\zz_0}(\bar{0},\bar{1})_{i,\ldots,n}=F_V(\zz_0)F_V\left(\zz_0 - \tau\,\sum_{j=i}^n \frac{\delta_j}{N}\right)^{-1}
e^{\frac{2\pi\i}{N}( x_i + \cdots + x_n)}\in \exp(\hat{\t}_{1,n}^\Gamma)\,.
\]

We finally introduce the simply connected domain $S_{\tau,n,\Gamma}$ consisting of $\zz\in\mathbb{C}^n$, with $z_i=a_i+b_i\tau$ ($a_i,b_i\in\mathbb{R}$) satisfying 
the following conditions: 
\begin{itemize}
\item for every $i<j$, $|a_i-a_j|<\frac1M$ and $|b_i-b_j|<\frac1N$;
\item for every $i<j$, $z_{ji}\notin\mathbb{R}_+\tau$. 
\end{itemize}
Note that $\mathfrak{S}_n(D_{\tau,n,\Gamma})\subset S_{\tau,n,\Gamma}$. We denote $F_S(\zz)$ the prolongation to $S_{\tau,n,\Gamma}$ 
of a given local solution $F(\zz)$ defined on $D_{\tau,n,\Gamma}$, and then consider for every $\sigma\in\mathfrak{S}_n$, 
\[
\underline{\sigma}^{\zz_0}:=F_S(\zz_0)F_S(\sigma\cdot\zz_0)^{-1}\sigma\in \exp(\hat{\t}_{1,n}^\Gamma)\rtimes\mathfrak{S}_n\,.
\]
Observe that the (unique) homotopy class of a path going from $\zz_0$ to $\sigma\cdot\zz_0$ represents the unique braid with underlying permutation 
$\sigma$ such that for every $i<j$, the $i$-th strand passes under the $j$-th strand whenever they cross (this is just a translation, in terms of braids, of the condition 
that $z_{ji}\notin\mathbb{R}_+\tau$). In other words, denoting this braid $\tilde{\sigma}$, $\underline{\sigma}^F=\mu_n^F(\tilde\sigma)$. 
As before, we also consider the projection of $\underline{\sigma}^{\zz_0}\in \exp(\hat{\t}_{1,n}^\Gamma)\rtimes\mathfrak{S}_n$ on the first factor: 
\[
\sigma^{\zz_0}:=\underline{\sigma}^{\zz_0}\sigma\in\exp(\hat{\t}_{1,n}^\Gamma)\,.
\]
\end{example}

Even though $\mu_n^{\zz_0}$ does not depend on the choice of local solution $F$, it is conjugated to a morphism that does depend on $F$. Indeed, one can define 
\[
\mu_n^F([\gamma]):=F(\zz_0)^{-1}\mu_n^{\zz_0}([\gamma])F(\zz_0)=F(h_\gamma\cdot\zz_0)^{-1}c_{h_\gamma} F(\zz_0)\,. 
\]
The resulting monodromy morphism $\mu_n^F$ does not depend on $\zz_0$ (because it is a ratio of two solutions of the ellipsitomic KZB system), but does depend on $F$. 
Whenever $F(\zz_0)=1$, we obvisouly have $\mu_n^F=\mu_n^{\zz_0}$. 

In what follows, we considering the monodromy morphism $\mu_n:=\mu_n^F$ associated with the particular solution $F=F_{\tau,n}$ from Section \ref{sec:twistedsol}. 


\section{Formul\ae~for $\mu_n: \on{B}_{1,n} \to \exp(\hat{\t}_{1,n}^\Gamma)\rtimes(\Gamma^n \rtimes \mathfrak{S}_n)$}\label{sec:twistedsol3}

\begin{lemma} \label{lemma:tholonomy}
If $\sigma=(12)\in\mathfrak{S}_n$ then $\tilde\sigma=\sigma_1$, and $\mu_n(\tilde\sigma)=e^{\pi\i t_{12}^\mathbf{0}}\sigma$. 
\end{lemma}
\begin{proof}
Only the last claim requires a proof. 
Let us consider $\zz$ such that $a_1<a_2$ (e.g.~$\zz\in D_{\tau,n,\Gamma}$), guaranteeing that $\sigma\zz=(z_2,z_1,z_3,\dots,z_n)\in S_{\tau,n,\Gamma}$. 
Recall the expansion 
\[
F(\zz)\simeq z_{21}^{t^\mathbf{0}_{12}}z_{31}^{t^\mathbf{0}_{13} + t^\mathbf{0}_{23}}
\cdots  z_{n1}^{t^\mathbf{0}_{1n} + \cdots  + t^\mathbf{0}_{n-1,n}}
\]
in the region $|z_{21}| \ll |z_{31}| \ll \cdots \ll |z_{n1}| \ll 1$. Hence $\sigma\cdot F(\zz)$ has a similar expansion, and 
\[
F(\sigma\zz)\simeq z_{12}^{t_{12}^\mathbf{0}}z_{32}^{t^\mathbf{0}_{13} + t^\mathbf{0}_{23}}
\cdots  z_{n2}^{t^\mathbf{0}_{1n} + \cdots  + t^\mathbf{0}_{n-1,n}}
\]
in the same region. 
With our choice of branch of log (see Remark \ref{branchremark}, and the definition of the domain $S_{\tau,n,\Gamma}$), one gets that 
$\log(z_{12})=\log(z_{21})-\pi\i$. Therefore 
\[
\mu_n(\tilde\sigma)=F(\sigma\zz)^{-1}\sigma\cdot F(\zz)\sigma \simeq e^{\pi\i t_{12}^\mathbf{0}}\sigma\,.
\]
The last equivalence is an equality, as $\mu_n(\tilde\sigma)$ is constant. 
\end{proof}


Let $\Phi=\Phi_{\0}^{1,2,3}$ be the image in $\on{exp}(\hat\t_{1,3}^\Gamma)$ of the 
KZ associator $\Phi_{\on{KZ}}^{1,2,3}$ from Example \ref{exampleKZassoc} along the map $\on{exp}(\hat\t_{3})\to\on{exp}(\hat\t_{1,3}^\Gamma)$ 
given by $t_{ij}\mapsto  t_{ij}^\mathbf{0}$. Define
$$
\Phi_i:= \Phi_{\0}^{1\dots i-1,i,i+1\dots n}\cdots\Phi_{\0}^{1\dots n-2,n-1,n}\in \on{exp}(\hat\t_{1,n}^\Gamma).
$$

\begin{proposition}\label{proposition:abholonomy}
For every $n\geq3$, and every $i=2,\dots,n$, 
\[
\mu_n(A_i) = \Phi_i\mu_2(A_2)^{1\dots i-1,i\dots n}\Phi_i^{-1}
\quad\textrm{and}\quad
\mu_n(B_i)= \Phi_i\mu_2(B_2)^{1\dots i-1,i\dots n}\Phi_i^{-1}\,. 
\]
\end{proposition}

\begin{proof}
We first compute the monodromy $\mu_{i,n}:=\mu_n^G$ associated with another solution $G$ of the (first line of the) ellipsitomic KZB system: the one (for a fixed $\tau$) 
having the expansion
\[
G(\zz)\simeq z_{21}^{ t^\mathbf{0}_{12}}\cdots z_{i-1,1}^{ t^\mathbf{0}_{12}+\cdots+ t^\mathbf{0}_{1,i-1}}
z_{n,i}^{ t^\mathbf{0}_{i,n}+\cdots+ t^\mathbf{0}_{n-1,n}}\cdots z_{n,n-1}^{ t^\mathbf{0}_{n-1,n}}
\]
in the region where $|z_{21}|\ll \cdots \ll |z_{i-1,1}|\ll 1$ and $|z_{n,n-1}|\ll \cdots \ll |z_{n,i}|\ll 1$. 

We claim that 
\[
\mu_{i,n}(A_i) = \mu_2(A_2)^{1\dots i-1,i\dots n}
\quad\textrm{and}\quad
\mu_{i,n}(B_i)= \mu_2(B_2)^{1\dots i-1,i\dots n}\,,
\]
and only give the proof for $A_i$, as the proof for $B_i$ is the exact same. 
The element $A_i$ can be represented by a path inside the region $|z_{21}|\ll \cdots \ll |z_{i-1,1}|\ll 1$ and $|z_{n,n-1}|\ll \cdots \ll |z_{n,i}|\ll 1$, 
that keeps the coordinates $z_1,\dots,z_{i-1}$, as well as the differences between the remaining coordinates, fixed. 
Hence, computing $\mu_n^G(A_i)$ amounts to compute the monodromy along the same path 
for the differential equation
\[
\partial_{z_n}G(\zz)=\sum_{j=i}^nK_j(\zz|\tau)G(\zz)\,, 
\]
where $\zz=(z_1,\dots,z_{i-1},z_n+s_i,z_n+s_{i+1},\dots,z_n+s_{n-1},z_n)$. 
Now observe that  the difference $\sum_{j=i}^nK_i(\zz|\tau)-K_2(z_1,z_n|\tau)^{1\dots i-1,i\dots n}$ 
(where $K_2$ is from the $n=2$ points system) tends to $0$ whenever $z_j\to z_n$ for $j\geq i$ and $z_\ell\to z_1$ for $\ell<i$.
\footnote{
More precisely, 
\[
K_j(\zz|\tau)=-y_j+\sum_{\ell:\ell\neq j}\sum_{\alpha\in \Gamma}k_\alpha\big(\on{ad}(x_j),z_{j\ell}|\tau\big)(t_{j\ell}^\alpha)\,,
\]
where $k_\alpha(u,v|\tau)$ are formal power series in $u$ with coefficient being meromorphic functions in $v$, and satisfying the identity 
\[
k_\alpha(-u,-v|\tau)+k_{-\alpha}(u,v|\tau)=0\,.
\]
We refer to \cite{CaGo} for more details (see also the next chapter). In particular 
\[
k_\alpha\big(\on{ad}(x_\ell),z_{\ell j}|\tau\big)(t_{\ell j}^\alpha)+k_{-\alpha}\big(\on{ad}(x_j),z_{j\ell}|\tau\big)(t_{j\ell}^\alpha)=0\,,
\]
and thus 
\[
\sum_{j=i}^nK_j(\zz|\tau)=-\sum_{j=i}^ny_j+\sum_{j=i}^n\sum_{\ell=1}^{i-1}\sum_{\alpha\in\Gamma}k_\alpha\big(\on{ad}(x_j),z_{j\ell}|\tau\big)(t_{j\ell}^\alpha)\,.
\]
On the other hand, 
\[
K_2(z_1,z_n|\tau)=-\sum_{j=i}^ny_j+\sum_{\ell=1}^{i-1}\sum_{j=i}^n\sum_{\alpha\in\Gamma}k_\alpha\big(\on{ad}(x_j),z_{n1}|\tau\big)(t_{j\ell}^\alpha)\,.
\]
Therefore their difference indeed tends to $0$ whenever $z_j\to z_n$ for $j\geq i$ and $z_\ell\to z_1$ for $\ell<i$.
} 
Hence $\mu_{i,n}(A_i)=\mu_2(A_2)^{1\dots i-1,i\dots n}$. 

Finally, the two monodromy representations $\mu_n=\mu_n^F$ and $\mu_{i,n}=\mu_n^G$ are conjugated. Indeed, 
\[
\mu_n^F([\gamma])=\Phi_{F,G}(h_\gamma\cdot\zz)\mu_n^G([\gamma])\Phi_{F,G}(\zz)^{-1}\,,
\]
with $\Phi_{F,G}(\zz):=F(\zz)^{-1}G(\zz)$ being constant as it is a ratio of two solutions of the ellipsitomic KZB system. 
To conclude, we prove that $\Phi_{F,G}=\Phi_i$. For this we consider the rational universal KZ system from \cite[(2.2)]{DrGal} (with $\hbar=1$, and $t^{ij}=t_{ij}^\0$), 
and denote by $\tilde{F}$ (resp.~$\tilde{G}$) the solution of this KZ system that have the same expansion as $F$ (resp.~$G$). 
In the whole region where $|z_{ij}| \ll 1$ for every $i\neq j$, the ellipsitomic KZB system and the rational KZ system only differ by a holomorphic part, 
therefore $F\simeq\tilde{F}$ and $G\simeq\tilde{G}$ in this region. Therefore, as they are constant, $\Phi_{F,G}=\Phi_{\tilde{F},\tilde{G}}$. 
Finally, it is a standard fact that $\Phi_{\tilde{F},\tilde{G}}=\Phi_i$ (see again \cite{DrGal}). 
\end{proof}

Using similar techniques, one can actually prove that the restriction of $\mu_n$ on $\on{B}_n\subset \overline{\on{B}}_{1,n}$ coincides with the monodromy morphism for 
the rational KZ system from \cite[(2.2)]{DrGal} associated with the solution $\tilde{F}$ having the same expansion as $F$. 
In particular, $\mu_3(\sigma_2)=\Phi e^{\pi\i\bar{\t}_{23}^{\mathbf{0}}}(23)\Phi^{-1}$. 


\section{Algebraic relations for the ellipsitomic KZB associator}

We now finish the proof of Theorem \ref{theorem:twistedKZBass}. 
\begin{remark}
The results of Sections \ref{sec:twistedsol}, \ref{sec:twistedsol2} and \ref{sec:twistedsol3} remain true in the reduced case, and we will make use of 
the same notation as in the previous sections.
\end{remark}

Let us set $\underline{A}:=\mu_2(A_2)$ and  $\underline{B}:=\mu_2(B_2)$, both viewed in $\exp(\hat{\bar{\t}}_{1,2}^\Gamma)\rtimes\Gamma^2/\Gamma$. 
In other words, 
\[
\underline{A}=F_{\tau,2}(z,-\frac1M)^{-1}(-\overline{1},\overline{0})_2F_{\tau,2}(z,0)=F_{\tau,2}(z+\frac1M,0)^{-1}(\overline{1},\overline{0})_1F_{\tau,2}(z,0)
\]
and 
\[
\underline{B}=F_{\tau,2}(z,-\frac{\tau}{N})^{-1}(\overline{0},-\overline{1})_2F_{\tau,2}(z,0)=F_{\tau,2}(z+\frac{\tau}{N},0)^{-1}(\overline{0},\overline{1})_1F_{\tau,2}(z,0)\,.
\]
One can easily check that the pair $(\underline{A},\underline{B})$ coincides with the pair $\big(\underline{A}^\Gamma(\tau),\underline{B}^\Gamma(\tau)\big)$ from 
Section \ref{section:pair}. Indeed, if $F_{\tau,2}(z_1,z_2)$ is the solution of the ellipsitomic KZB system defined on $D_{\tau,2,\Gamma}$ with expansion 
$F_{\tau,2}(z_1,z_2)\simeq z_{21}^{t_{12}^{\mathbf{0}}}$ whenever $|z_{21}| \ll 1$, then $J(z)=F_{\tau,2}(z,0)$ is the solution 
of the differential equation \eqref{equationKZBonevariable} with expanson $J(z)\simeq (-z)^{t^{\mathbf{0}}}$ whenever $z\to 0$ from Section \ref{section:pair}. 

\medskip

The identity $A_{3}^{-1}A_{2} = \sigma_{1}A_{2}^{-1}\sigma_{1}$ obviously holding in $\overline{\on{B}}_{1,3}$, is equivalent to the identity 
$A_3=A_2\sigma_1^{-1}A_2\sigma_1^{-1}$. Applying the monodromy morphism $\mu_3$ therefore yields  
\begin{equation}\label{eqn:560bis}
\underline{A}^{12, 3} =  
  \Phi^{1, 2, 3}  \underline{A}^{1, 23} (\Phi^{1, 2, 3})^{- 1} e^{-\pi\i \bar{t}_{12}^{\0}}
  \Phi^{2, 1, 3}  \underline{A}^{2, 13}   (\Phi^{2, 1, 3})^{- 1} e^{-\pi\i \bar{t}_{12}^{\0}},
\end{equation}
that is \eqref{def:tell:ass:G:1bis}. 
Similarly, the identity $B_{3}=B_{2}\sigma_{1}^{-1}B_{2}\sigma_{1}^{-1}$ yields  
\begin{equation}\label{eqn:561}
\underline{B}^{12, 3} =  
  \Phi^{1, 2, 3}  \underline{B}^{1, 23} (\Phi^{1, 2, 3})^{- 1} e^{-\pi\i \bar{t}_{12}^{\0}}
  \Phi^{2, 1, 3}  \underline{B}^{2, 13}   (\Phi^{2, 1, 3})^{- 1} e^{-\pi\i \bar{t}_{12}^{\0}},
\end{equation}
that is \eqref{def:tell:ass:G:2bis}. 

In $\overline{\on{B}}_{1,3}$, on also has $(X_2,Y_3)=P_{23}$. Recalling that $X_2=A_3A_2^{-1}$ and $Y_3=B_3^{-1}$, one gets $P_{23}=(A_3A_2^{-1},B_3^{-1})$ which, 
after applying the monodromy morphism $\mu_3$, yields 
\begin{eqnarray}\label{eqn:562bis}
\Phi e^{2 \pi \i \bar{t}_{23}^{\tmmathbf{0}}} \Phi^{- 1} & = & \underline{{A}}^{12, 3}
  \Phi (\underline{{A}}^{1, 23})^{- 1}\Phi^{- 1} (\underline{{B}}^{12, 3})^{-
  1} \Phi \underline{{A}}^{1, 23}\Phi^{- 1} (\underline{{A}}^{12,
  3})^{- 1}  \underline{{B}}^{12, 3},
\end{eqnarray}
which is  \eqref{def:tell:ass:G:3bis}


\medskip

This proves that the pair $\big(\underline{A}^\Gamma(\tau),\underline{B}^\Gamma(\tau)\big)=(\underline{A},\underline{B})$ satisfies \eqref{def:tell:ass:G:1bis}, \eqref{def:tell:ass:G:2bis} 
and \eqref{def:tell:ass:G:3bis}. Hence, according to Remark \ref{rem:ass:tell} it satisfies \eqref{def:tell:ass:G:1}, \eqref{def:tell:ass:G:2} and \eqref{def:tell:ass:G:3}, and thus 
$e^\Gamma(\tau)=(A^\Gamma(\tau),B^\Gamma(\tau))$ defines an element in $\on{Ell}^\Gamma_{\on{KZB}}$.

This concludes the proof of Theorem \ref{theorem:twistedKZBass}.  \hfill$\Box$

\begin{remark}
If $\Gamma$ is trivial, we retrieve relations (22), (23), (25) and (26) from \cite{CEE}, up to some changes of convention 
(for the monodromy action, and for the open subset of ``base configurations'' of marked points). 
\end{remark}

\chapter{Number theoretic aspects: Eisenstein series}
\label{Section7}

In the previous chapter we studied (the first line of) the ellipsitomic KZB system \eqref{equation-system} of differential equations and deduced from it an element in the 
set of ellipsitomic associators over $\C$. One of the main ingredients defining this differential system is given by
\begin{equation}\label{k-gamma}
k_\gamma(x,z|\tau):=e^{-\frac{2\pi\i v}{N}x}{{\theta(z-\tilde\gamma+x|\tau)}
\over{\theta(z-\tilde\gamma|\tau)\theta(x|\tau)}}-{1\over x}\,,
\end{equation}
where $\tau \in \h$, $\gamma=(\bar{u},\bar{v}) \in \Gamma:=\mathbb{Z}/M\mathbb{Z}\times \mathbb{Z}/N\mathbb{Z}$ and 
$\tilde{\gamma} = \frac{u}{M}+\frac{v}{N}\tau  \in \Lambda_{\tau, \Gamma}:=\frac1M\mathbb{Z}+\frac{\tau}N\mathbb{Z}$ is any lift of $\gamma$.  
Here we implicitely used the canonical identification 
$\Gamma\simeq \Lambda_{\tau,\Gamma}/\Lambda_{\tau}$, where $\Lambda_{\tau}:=\mathbb{Z}+\tau\mathbb{Z}$. 

\medskip

Denote by $g_{\gamma} (x,z| \tau) \assign \partial_x k_{\gamma} (x,z| \tau)$ its partial derivative with respect to $x$.
In this chapter we take a closer look at the functions $A_{s,\gamma}(\tau)$, defined in \cite[Subsection 3.3]{CaGo} as the Taylor 
coefficients of $g_{- \gamma} (x,0| \tau)$: 
\[
 g_{- \gamma} (x,0| \tau) =  \sum_{s \geq 0} A_{s, \gamma} (\tau) x^s.
\]
After a brief account on Eisenstein series for congruence subgroups, we express $A_{s,\gamma}(\tau)$ 
in terms of these Eisenstein series, giving evidence that they should be quasi-modular forms for the group $\on{SL}_2^\Gamma$. 

\medskip

We end the chapter with some perspectives about ellipsitomic Grothendieck--Teichm\"uller theory and twisted elliptic multiple zeta values. 


\section{Eisenstein series for $\on{SL}^{\Gamma}_2$}

We refer to \cite{Diamond-Shurman,Zagier} for generalities about modular forms and Eisenstein series. 

Recall the \textit{Eisenstein series} $G_{s}$, defined for all integers $s\geq 2$ by
\[
G_{s}(\tau) := \sum_{n=-\infty}^{+\infty} \left( \sum_{\underset{m\neq0 \text{ if } n=0}{m=-\infty}}^{+\infty} \frac{1}{(m+n\tau)^s}\right)\,.
\]
The Eisenstein series $G_s$ are modular forms for $\on{SL}_2(\Z)$ of weight $s$ for $s\geq3$, and $G_2$ is quasimodular (in the sense of \cite{KaZa}). 
One easily sees that $G_{s}(\tau)=0$ whenever $s$ is odd, and that the value of $G_s$ at the cusp $\i\infty$ for an even $s=2n$ is 
$2\zeta(s)=(-1)^{n+1}\frac{(2\pi)^{2n}B_{2n}}{(2n)!}$, where $B_s$ are Bernoulli numbers and $\zeta$ is the Riemann zeta function.  
Hence, for every integer $s \geq 2$, one defines the \textit{normalized Eisenstein series} $E_s (\tau) \assign \frac{G_s (\tau)}{2 \zeta (s)}$. 

Let us now introduce the functions $G_{s}(z|\tau)$ defined for $(z|\tau)\in\C\times\mathfrak{h}$ such that $z\notin\Lambda_\tau$, and for every integer $s\geq 2$ as 
\[ 
G_{s}(z|\tau) := \sum_{n=-\infty}^{+\infty}\sum_{m=-\infty}^{+\infty}\frac{1}{(m + n \tau + z)^s}\,.
\]
For $s \geq 3$, the series  $G_{s}(z|\tau)$ is absolutely and locally uniformly convergent, and defines a holomorphic function on $\mathfrak{h}$ 
for every $z\in\C-\Lambda_\tau$. 
For $s=2$, it is still locally uniformly convergent, and thus still holomorphic, but is no longer absolutely convergent (so that we are not allowed to 
re-order terms in the series). 

For a fixed $\tau\in\mathfrak{h}$, one can see that $G_s(-|\tau)$ is $\Lambda_\tau$-periodic, so that $G_{s}(z|\tau)$ only depends on the class 
$[z]\in E_\tau^\times=(\C-\Lambda_\tau)/\Lambda_\tau$. 
Hence, for $\gamma=(\bar{u},\bar{v})\in\Gamma\simeq \Lambda_{\tau,\Gamma}/\Lambda_\tau\subset E_\tau$, we can define 
\[
\text{\gls{Gsg}}\assign
\begin{cases}
G_{s}(z|\tau)~\textrm{ with }~z=\frac{u}{M}+\frac{v}{N}\tau & \textrm{ if }\gamma\neq\mathbf{0}\,, \\
G_s(\tau) & \textrm{ else}\,.
\end{cases}
\]

\begin{proposition}\label{modularity}
Let $s\geq 3$ and $\gamma\in\Gamma$. 
The function $G_{s,\gamma}$ is a modular form of weight 
$s$ with respect to $\on{SL}_2^\Gamma$.  
\end{proposition}

\begin{proof}
This is a classical fact for $\gamma=\0$ (see \cite{Diamond-Shurman,Zagier}). The proof is probably standard and known to experts even in the case 
$\gamma\neq\mathbf{0}$, but we provide it here as we could not find a reference for it. We will first prove weak modularity, and then holomorphy at the cusps. 

\begin{lemma}[Weak modularity]\label{lemma:weakmodularity}
For every $s\geq 3$, $\tau\in\mathfrak{h}$, $z\in\C-\Lambda_\tau$, and $\alpha=\begin{pmatrix}a&b\\c&d\end{pmatrix}\in\on{SL}_2(\Z)$, 
\[
G_{s}\big(\alpha \cdot(z|\tau)\big) =(c\tau+d)^s G_{s}(z|\tau)\,.
\]
In particular, if $s\geq 3$ and $\gamma\in\Gamma-\{\mathbf{0}\}$, $G_{s,\gamma}$ is weakly modular of weight $s$ with respect to $\on{SL}_2^\Gamma$. 
\end{lemma}
Recall that $\alpha \cdot(z|\tau):=\left(\frac{z}{c\tau+d}|\frac{a\tau+b}{c\tau+d}\right)$. 
\begin{proof}
For $z\notin\Lambda_\tau$, we compute:  
\begin{eqnarray*}
G_{s}\big(\alpha \cdot(z|\tau)\big)
& = & \underset{(m, n) \in \mathbb{Z}^2}{\sum} \frac{1}{\big(m + n(\frac{a\tau+b}{c\tau+d})+\frac{z}{c\tau+d}\big)^{s}} \\
& = & (c\tau+d)^s \underset{(m, n) \in \mathbb{Z}^2}{\sum} \frac{1}{\big(m(c\tau+d)+n(a\tau+b)+z\big)^{s}} \\
& = & (c\tau+d)^s \underset{(m, n) \in \mathbb{Z}^2}{\sum} \frac{1}{(m+n\tau+z)^{s}} = G_s(z|\tau)\,.
\end{eqnarray*}
Observe that if $(z'|\tau')=\alpha\cdot(z|\tau)$ with $z'=x+\tau' y$, then $z=(c\tau+d)x+(a\tau+b)y$. 
Hence the inverse of the action of $\alpha$ gives an isomorphism $E_{\tau'}\to  E_{\tau}$ that is precisely given by 
$x+\tau' y\mapsto (c\tau+d)x+(a\tau+b)y$. 

Now assume that $\alpha\in\on{SL}_2^\Gamma$, which means that $a \equiv 1$ mod $M$, $d \equiv 1$ mod $N$, $b\equiv0$ mod $N$ 
and $c\equiv0$ mod $M$. In particular, $z'\in\Lambda_{\tau',\Gamma}$ if and only if $z\in\Lambda_{\tau,\Gamma}$, and moreover the induced 
composed isomorphism 
\[
\Gamma\simeq\Lambda_{\tau',\Gamma}/\Lambda_{\tau'} \, \tilde\longrightarrow \, \Lambda_{\tau,\Gamma}/\Lambda_\tau\simeq\Gamma
\]
is the identity. Indeed, $z'=\frac{u}{M}+\frac{v}{N}\tau'\in\Lambda_{\tau',\Gamma}$ is sent to 
$\frac{u}{M}d+\frac{v}{N}b+\big(\frac{v}{N}a+\frac{u}{M}c\big)\tau\in \frac{u}{M}+\frac{v}{N}\tau+\Lambda_\tau$. 
Therefore, if $\gamma\in\Gamma-\{\mathbf{0}\}$, 
\[
G_{s,\gamma}(\alpha\cdot\tau)	= G_s\Big(\alpha\cdot\big(\frac{u}{M}+\frac{v}{N}\tau|\tau\big)\Big)
											= (c\tau+d)^sG_s\big(\frac{u}{M}+\frac{v}{N}\tau|\tau\big)
											= G_{s,\gamma}(\tau)\,.
\]
This ends the proof of the Lemma. 
\end{proof}
\begin{remark}\label{remark:G2periodic}
The proof does not work in the case $s=2$ because we need to reorder the terms of the series to prove the required identity. 
Nevertheless, for elements of the form $\alpha=\begin{pmatrix}1 & H \\ 0 & 1\end{pmatrix}$, we can keep $n$ fixed and apply a shift by $nH$ in 
the internal series (the one running over $m$). Hence, for these $\alpha$'s, the required identity is true even in the case $s=2$. 
\end{remark}
As the function $G_{s,\gamma}$ is holomorphic on $\h$, it remains to show that it is also holomorphic at all cusps for $\on{SL}_2^\Gamma$. 
Recall that these cusps are orbits of the action of $\on{SL}_2^\Gamma$ on $\mathbb{P}^1(\mathbb{Q})=\mathbb{Q}\cup\{\infty\}$. 
\begin{lemma}\label{lemma:holomorphyatcusps}
For every $s\geq 3$ and $\gamma\in\Gamma-\{\mathbf{0}\}$, the function $G_{s,\gamma}$ is holomorphic at all cusps for $\on{SL}_2^\Gamma$.
\end{lemma}
\begin{proof}
Recall that for every $\alpha=\begin{pmatrix}a&b\\c&d\end{pmatrix}\in\on{SL}_2(\Z)$, 
the width of the cusp $[\alpha(\infty)]$ is the smallest positive integer $H$ such that $\alpha\begin{pmatrix}1&H\\0&1\end{pmatrix}\alpha^{-1}\in \on{SL}_2^\Gamma$. 
This condition is equivalent to the following requirements: $M\big|acH$, $M\big|c^2H$, $N\big|a^2H$, and $N\big|acH$. Since $a$ and $c$ are relatively prime, 
this in turn boils down to the condition that $M\big|cH$ and $N\big|aH$. 

Now observe that, in order to prove that $G_{s,\gamma}$ is holomorphic at this cusp, one is reduced to prove that the function 
\[
G_{s,\gamma,\alpha}:\tau\longmapsto (c\tau+d)^{-s}G_{s,\gamma}(\alpha\cdot\tau)
\]
is holomorphic at $\infty$. From the modularity property of $G_{s,\gamma}$, we know that $G_{s,\gamma,\alpha}$ is $H$-periodic, 
and thus descends to a holomorphic function $\hat{G}_{s,\gamma,\alpha}(q)=G_{s,\gamma,\alpha}(\tau)$ defined on the punctured unit disk, 
with $q=e^{\frac{2\pi\i\tau}{H}}$. 
Hence it remains to show that $G_{s,\gamma,\alpha}$ has a $q$-expansion with non-negative Fourier coefficients. Note furthermore that, according to 
Lemma \ref{lemma:weakmodularity}, $G_{s,\gamma,\alpha}(\tau)=\tilde{G}_{s}(z|\tau)$ with 
\[
z=(c\tau+d)\frac{u}{M}+(a\tau+b)\frac{v}{N}=\frac{u}{M}d+\frac{v}{N}b+(\frac{v}{N}a+\frac{u}{M}c)\tau=x+\frac{K}{H}\tau\,,
\] 
$K=u\frac{aH}{N}+v\frac{cH}{M}\in\mathbb{Z}$ (and $x=\frac{u}{M}d+\frac{v}{N}b\in\mathbb{Q}$). 
We let $K=QH+R$ be the euclidean division of $K$ by $H$, define $w:=x+\frac{R}{H}\tau$, and compute: 
\begin{eqnarray*}
G_{s,\gamma,\alpha}(\tau)	& = & \sum_{(m,n)\in\Z^2}\frac{1}{(m+n\tau+z)^s} = \sum_{(m,n)\in\Z^2}\frac{1}{(m+n\tau+w)^s} \\
										& = & \sum_{m\in\Z}\frac{1}{(m+w)^s}+\sum_{n\geq 1}\sum_{m\in\Z}\frac{1}{(m+n\tau+w)^s}
																								+\sum_{n\geq1}\sum_{m\in\Z}\frac{1}{(m-n\tau+w)^s}
\end{eqnarray*}
Let us show that the three series in the last expression have a $q$-expansion with non-negative coefficients, and start with $\sum_{m\in\Z}(m+w)^{-s}$. 
If $w\in\mathbb{R}$ then it is constant in $\tau$, and we are done. If $w\notin\mathbb{R}$, a standard calculation shows that 
\[
\sum_{m\in\Z}\frac{1}{(m+w)^s}	=  \frac{(-2\pi\i)^s}{(s-1)!}\sum_{r=1}^{+\infty}r^{s-1}e^{2\pi\i rw} 
												= \frac{(-2\pi\i)^s}{(s-1)!}\sum_{r=1}^{+\infty}r^{s-1}e^{2\pi i rx}q^{Rr}\,.
\]
The proofs for both double series are identical, hence we restrict ourselves to the first one, and compute: 
\[
\sum_{n\geq 1}\sum_{m\in\Z}\frac{1}{(m+n\tau+w)^s}=\sum_{n\geq1}\frac{(-2\pi\i)^s}{(s-1)!}\sum_{r=1}^{+\infty}r^{s-1}e^{2\pi i rx}q^{(nH+R)r}
=\sum_{k\geq1}\left(\sum_{\underset{\frac{k}{r}\equiv R\textrm{ mod }H}{r|k}}r^{s-1}e^{2\pi i rx}\right)q^k\,.
\]
This ends the proof of the Lemma. 
\end{proof}
The proof of Proposition \ref{modularity} is now completed.
\end{proof}
\begin{remark}
It follows from the proof of Lemma \ref{lemma:holomorphyatcusps} that in many cases (i.e.~whenever $w\notin\mathbb{R}$), 
$G_{s,\gamma}$ actually vanishes at the corresponding cusp. In the case of the cusp $[\infty]$, $G_{s,\gamma}$ does not vanish at the cusp only if 
$\gamma=(\bar{u},\bar{0})$, $u\in\{1,\dots, M-1\}$. In this case, the value at the cusp is $\zeta(s,u/M)+(-1)^s\zeta(s,-u/M)-\big(M/u\big)^s$, 
where
\[
\zeta(s,z)  \assign \underset{m\geq0}{\sum} \frac{1}{(m + z)^{s}}
\]
is the Hurwitz zeta function. 
\end{remark}

\begin{remark}\label{conjectural-remark}
It is likely that, using a variation on Hecke's trick (see e.g.~\cite[Proposition 6]{Zagier}), one could prove that 
$G_{2,\gamma}$ is quasi-modular with respect to $\on{SL}^{\Gamma}_2$. 
\end{remark}


\section{The coefficients $A_{s,\gamma}(\tau)$}

Let us recall some standard properties of the Weierstrass function $\wp : \C\times\mathfrak{h}\longrightarrow \C$ given by
\[
\wp (z|\tau) = \frac{1}{z^2}+\sum_{(m,n)\in\Z^2-\{(0,0)\}}\left(\frac{1}{(z+m+n\tau)^2}-\frac{1}{(m+n\tau)^2}\right)=G_2(z|\tau)-G_2(\tau).
\]
In the variable $z$, it is even, periodic with respect to $\Lambda_\tau$, and meromorphic with poles of order two in $\Lambda_\tau$. 
There exists a constant $c\in \C$ such that 
\[
\wp (z | \tau) = - \partial_z^2 \log \big(\theta (z| \tau)\big) + c\,.
\]
In a suitable punctured neighbourhood of $0$ (e.g.~the maximal punctured open disk centred at $0$ which does not contain 
any lattice point), there is a Laurent expansion
\[
\wp (z | \tau) = \frac{1}{z^2} + \underset{k = 0}{\overset{\infty}{\sum}}
   \frac{f^{(2 k)} (0)}{(2 k) !} z^{2 k} = \frac{1}{z^2} + \underset{k = 1}{\overset{\infty}{\sum}}
   (2 k + 1) G_{2 k + 2} (\tau) z^{2 k}, \]
where  $f (z) = \wp (z|\tau) - \frac{1}{z^2}$. 

\begin{proposition}
For every $s\geq0$, $A_{s,\0}(\tau)=-(s+1)G_{s+2}(\tau)$. 
\end{proposition}
Therefore $A_{s,\0}(\tau)$ is a modular form of weight $s+2$ for $\on{SL}_2^\Gamma$ whenever $s>0$, while $A_{0,\0}(\tau)$ is only quasi-modular (of weight $^2$). 
\begin{proof}
This is proven in \cite{CEE}. Roughly, one first sees that $g_{\0}(x,0|\tau)=\partial_x^2\log \big(\theta (x| \tau)\big)+1/x^2$, 
which proves the required equality for $s\geq1$. Then a specific analysis of the constant term (see e.g.~\cite[\S 4.1]{CEE}) 
tells us that $g_{\0}(0,0|\tau)=4\pi\i\partial_\tau \log\eta(\tau)=-G_2(\tau)$, where $\eta$ is the Dedekind eta function. 
\end{proof}

Let now $\gamma=(\bar{u},\bar{v})\in\Gamma-\{\0\}$, and $\tilde\gamma=\frac{u}{M}+\frac{v}{N}\tau\in\Lambda_{\tau,\Gamma}-\Lambda_\tau$ 
be a lift of $\gamma$. Recall that we are interested in the Taylor coefficients $A_{s,\gamma}(\tau)$ of 
\[
g_{-\gamma} (x,0| \tau)=\partial_x\left(e^{\frac{2\pi\i v}{N}x}\frac{\theta(\tilde\gamma+x|\tau)}{\theta(\tilde\gamma|\tau)\theta(x|\tau)}-\frac{1}{x}\right)\,.
\]
We define 
$F_\gamma(x|\tau):=e^{\frac{2\pi\i v}{N}x}\frac{\theta(\tilde\gamma+x|\tau)}{\theta(\tilde\gamma|\tau)\theta(x|\tau)}$, 
so that $g_{-\gamma} (x,0| \tau)=\partial_xF_\gamma(x|\tau)+\frac{1}{x^2}$. 
\begin{lemma}
If we define
\[
a_{1,\gamma}(\tau):=\frac{2\pi\i v}{N}+\frac{\partial_x\theta(\tilde\gamma|\tau)}{\theta(\tilde\gamma|\tau)}\quad\textrm{and}\quad
a_{s,\gamma}(\tau):=(-1)^s\frac{G_{s}(\tau)-G_{s,\gamma}(\tau)}{s}\quad (s\geq2)\,,
\]
then $F_\gamma(x|\tau)=\frac{1}{x}\exp\left(\underset{s\geq 1}{\sum}a_{s,\gamma}(\tau)x^s\right)$. 
\end{lemma}
\begin{proof}
We first compute: 
\[
\partial_x^2\log\big(F_\gamma(x|\tau)\big)
=\partial_x^2\log\big(\theta(\tilde\gamma+x|\tau)\big)-\partial_x^2\log\big(\theta(x|\tau)\big)
=G_2(x|\tau)-G_2(\tilde\gamma+x|\tau)
\]
A simple calculation shows that
\[
G_2(x|\tau)=\frac{1}{x^2}+\sum_{s\geq0}(-1)^s(s+1)G_{s+2}(\tau)x^s
\]
and that 
\[
G_2(\tilde\gamma+x|\tau)=\sum_{s\geq0}(-1)^s(s+1)G_{s+2,\gamma}(\tau)x^s\,.
\]
Hence 
\[
\partial_x^2\log\big(F_\gamma(x|\tau)\big)-\frac{1}{x^2}=\sum_{s\geq0}(-1)^s(s+1)\big(G_{s+2}(\tau)-G_{s+2,\gamma}(\tau)\big)x^s\,.
\]
Knowing that $\left(\frac{1}{x}-\frac{\partial_x\theta(x|\tau)}{\theta(x|\tau)}\right)_{x=0}=0$, we obtain 
\[
\left(\partial_x\log\big(F_\gamma(x|\tau)\big)+\frac{1}{x}\right)_{x=0}
=\frac{2\pi\i v}{N}+\frac{\partial_x\theta(\tilde\gamma|\tau)}{\theta(\tilde\gamma|\tau)}=a_{1,\gamma}(\tau)\,,
\]
and thus 
\[
\partial_x\log\big(F_\gamma(x|\tau)\big)+\frac{1}{x}=a_{1,\gamma}(\tau)+\sum_{s\geq0}(-1)^s\big(G_{s+2}(\tau)-G_{s+2,\gamma}(\tau)\big)x^{s+1}\,.
\]
Finally, since $\Big(\log(x)-\log\big(\theta(x|\tau)\big)\Big)_{x=0}=0$, we get that 
$\left(\log\big(F_\gamma(x|\tau)\big)+\log(x)\right)_{x=0}=0$, 
and thus 
\[
\log\big(F_\gamma(x|\tau)\big)+\log(x)=a_{1,\gamma}(\tau) x+\sum_{s\geq0}(-1)^s\frac{G_{s+2}(\tau)-G_{s+2,\gamma}(\tau)}{s+2}x^{s+2}
=\sum_{s\geq 1}a_{s,\gamma}(\tau)x^s\,.
\]
This shows that $F_\gamma(x|\tau)=\frac{1}{x}\exp\left(\underset{s\geq 1}{\sum}a_{s,\gamma}(\tau)x^s\right)$. 
\end{proof}
As a consequence, we get 
\[
F_\gamma(x|\tau)-\frac{1}{x}=\frac{\exp\left(\underset{s\geq 1}{\sum}a_{s,\gamma}(\tau)x^s\right)-1}{x}\,, 
\]
so that 
\[
\sum_{s \geq 0} A_{s, \gamma}(\tau):=g_{-\gamma} (x,0| \tau)=\partial_x\left(\frac{\exp\left(\underset{s\geq 1}{\sum}a_{s,\gamma}(\tau)x^s\right)-1}{x}\right)\,.
\]
Hence $A_{s,\gamma}$ can be expressed as an explicit linear combination of products of the form 
$a_{s_1,\gamma}\cdots a_{s_k,\gamma}$, with $s_1+\cdots+s_k=s+2$, 
and we expect $A_{s,\gamma}$ to be quasi-modular of weight $s+2$ with respect to $\on{SL}_2^\Gamma$ 
(as hinted from Remark \ref{conjectural-remark} and Proposition \ref{modularity}).

\section{Concluding remarks and outlook}

Observe that $\bar{\mathfrak{t}}_{1,2}^{\Gamma}$ is the free Lie algebra generated by $x:=\bar x_1$, $y:=\bar y_2$ and $t^\alpha:=\bar t_{12}^\alpha$, 
for $\alpha\in\Gamma-\{\0\}$. Moreover, the element defining the differential equation \eqref{equationKZBonevariable} lies in the (degree completion of the) 
ideal generated by $y$ and $t^\alpha$, $\alpha\in\Gamma-\{\0\}$. 
By Lazard's elimination theorem (see \cite[Theorem 1]{Laz}), as a Lie algebra this ideal is isomorphic to the free Lie algebra 
generated by $y_n$ and $t^\alpha_n$, $n\geq0$ and $\alpha\in\Gamma-\{\0\}$; the isomorphism sends $y_n$ to $\on{ad}(x)^n(y)$ and 
$t^\alpha$ to $\on{ad}(x)^n(t^\alpha)$. 

As a consequence, $A^\Gamma(\tau)$ and $B^\Gamma(\tau)$ can be seen as elements of the formal power series algebra 
$\C\langle \langle y_n, t_n^\alpha\,|\, n\geq0,\alpha\in\Gamma-{\textbf 0} \rangle \rangle$. 
The coefficients of these series can be computed as iterated integrals, and are $\Gamma$-twisted versions of Enriquez's elliptic analogs of 
multiple zeta values \cite{En3}. 

This approach to elliptic multiple zeta values at torsion points seems different to that of the work of 
Broedel--Matthes--Richter--Schlotterer \cite{M2}. The relation between the twisted elliptic multiple zeta values obtained in this paper and 
that in \cite{M2} deserves further investigations. A comparison with the values at torsion points of Goncharov's multiple elliptic polylogarithms 
\cite[Section 8]{Go} would also be interesting. 

\bigskip

Finally, in addition to the agebraic properties of $e^\Gamma(\tau)$, that are essentially given by Theorem \ref{theorem:twistedKZBass}, it would be 
interesting to study its analytic and modularity properties. 
In the elliptic case, when $\Gamma$ is the trivial group, this was done in \cite[\S 5.4 \& \S 5.5]{En2}, and we expect something similar in the more general 
ellipsitomic case. 

For the analytic properties of the ellipsitomic associator, it amounts to understanding how $e^\Gamma(\tau)$ depends on small variations of the modulus $\tau$. 
For that, one can use the second line of the ellipsitomic KZB system \eqref{equation-system}, and compute $\partial_\tau e^\Gamma(\tau)$. 
Indeed, recall from \cite[Subsection 2.3]{CaGo} that $\delta_{s, \gamma}$ acts as a derivation on 
$\bar{\t}_{1,2}^\Gamma$. We can modify it in the following way, by introducing a new derivation 
\[
\varepsilon_{s, \gamma} \assign \delta_{s, \gamma} - 2[(\tmop{ad} x)^s t^{- \gamma}, -]\,.
\]
Then the second line of the ellipsitomic KZB system \eqref{equation-system} for $n=2$ reads as 
\[
2\pi\on{i}\partial_\tau F(z|\tau)
=-\left(\Delta_0+\frac12\sum_{s\geq0}\sum_{\gamma\in\Gamma}A_{s,\gamma}(\tau)\varepsilon_{s, \gamma}+O(z)\right)\cdot F(z|\tau)\,,
\]
where $z=z_{12}$ and $O(z)$ denotes a term of the form $zf(z|\tau)$, with $f$ being holomorphic on $\C\times\mathfrak{H}$. 
Hence, going along the lines of \cite[\S 5.4]{En2} on can prove that 
\[
2 \pi \i\partial_\tau A^\Gamma(\tau)=
-\left(\Delta_0 + \frac12\underset{\gamma \in \Gamma}{\sum} \underset{s \geqslant 0}{\sum}A_{s, \gamma}(\tau) \varepsilon_{s, \gamma} \right)
\cdot A^\Gamma(\tau)
\]
and 
\[
2 \pi \i\partial_\tau B^\Gamma(\tau)=
-\left(\Delta_0 + \frac12\underset{\gamma \in \Gamma}{\sum} \underset{s \geqslant 0}{\sum}A_{s, \gamma}(\tau) \varepsilon_{s, \gamma} \right)
\cdot B^\Gamma(\tau)\,.
\]
The derivation $\varepsilon_{s, \gamma}$ shall be relevant for the study of the ellipsitomic Grothendieck-Teichm\"uller 
group, as well as of a yet to be defined analog of Tsunogai's special derivation algebra from \cite{Ts} in the ellipsitomic case.

\appendix
\chapter{An alternative presentation for $\PaB_{e\ell\ell}^\Gamma$}

In this appendix, we provide an alternative presentation for $\PaB_{e\ell\ell}^\Gamma$, 
that we use in chapter \ref{Section6} when proving that the monodromy of the ellipsitomic KZB 
connection indeed gives rise to an ellipsitomic associator. 

\section{An alternative presentation for $\PaB_{e\ell\ell}$}

We first state and prove the result when the group $\Gamma$ is trivial. 

\medskip

\paragraph{\underline{The relations \eqref{eqn:N1bis} and \eqref{eqn:N2bis}}}
We introduce three new relations, which are satisfied in the automorphism group 
of the object $(12)3$ in $\PaB_{e\ell\ell}$ (this can be seen topologically): 
\begin{flalign}
& A^{12,3} = \Phi^{1,2,3}A^{1,23}(\Phi^{1,2,3})^{-1}\tilde R^{1,2}\Phi^{2,1,3}A^{2,13} (\Phi^{2,1,3})^{-1}\tilde R^{2,1}\,, 
\tag{N1bis}\label{eqn:N1bis} \\
& B^{12,3} = \Phi^{1,2,3}B^{1,23}(\Phi^{1,2,3})^{-1}\tilde R^{1,2}\Phi^{2,1,3}B^{2,13} (\Phi^{2,1,3})^{-1}\tilde R^{2,1} \,,
\tag{N2bis}\label{eqn:N2bis} \\ 
& \Phi^{1,2,3}R^{2,3}R^{3,2}(\Phi^{1,2,3})^{-1} =\big(A^{12,3}\Phi^{1,2,3}(A^{1,23})^{-1}(\Phi^{1,2,3})^{-1}\,,\,(B^{12,3})^{-1}\big)\,.
\tag{Ebis}\label{eqn:Ebis}
\end{flalign}
For instance, equations \eqref{eqn:N1bis} and \eqref{eqn:N2bis} can be depicted as
\begin{center}
\begin{align}\tag{\ref{eqn:N1bis},\ref{eqn:N2bis}}
\begin{tikzpicture}[baseline=(current bounding box.center)]
\tikzstyle point=[circle, fill=black, inner sep=0.05cm] 
\node[point, label=above:$(1$] at (1,1) {};
 \node[point, label=below:$(1$] at (1,0) {};
 \node[point, label=above:$2)$] at (1.5,1) {};
 \node[point, label=below:$2)$] at (1.5,0) {};
 \node[point, label=above:$3$] at (3,1) {};
 \node[point, label=below:$3$] at (3,0) {};
 \draw[-,thick] (1,1) .. controls (1,0.5) and (1,0.5).. (1.5,0.5); 
 \draw[->,thick] (1.5,0.5) .. controls (1,0.5) and (1,0.5).. (1,0.05); 
 \draw[-,thick] (1.5,1) .. controls (1.5,0.5) and (1.5,0.5).. (2,0.5); 
 \draw[->,thick] (2,0.5) .. controls (1.5,0.5) and (1.5,0.5).. (1.5,0.05); 
\node[point, white, label=left:$\pm$] at (1,0.5) {};
 \draw[->,thick] (3,1) .. controls (3,0.5) and (3,0.5).. (3,0.05);
\end{tikzpicture} 
\qquad = \qquad
\begin{tikzpicture}[baseline=(current bounding box.center)] 
\tikzstyle point=[circle, fill=black, inner sep=0.05cm]
 \node[point, label=above:$(1$] at (1,4) {};
 \node[point, label=below:$(1$] at (1,0) {};
 \node[point, label=above:$2)$] at (1.5,4) {};
 \node[point, label=below:$2)$] at (1.5,0) {};
 \node[point, label=above:$3$] at (3,4) {};
 \node[point, label=below:$3$] at (3,0) {};
 \draw[-,thick] (1,4) .. controls (1,3.75) and (1,3.75).. (1,3.5);
 \draw[-,thick] (1.5,4) .. controls (1.5,3.75) and (2.5,3.75).. (2.5,3.5);
 \draw[-,thick] (1,3.5) .. controls (1,3.25) and (1,3.25).. (1.5,3.25); 
 \draw[-,thick] (1.5,3.25) .. controls (1,3.25) and (1,3.25).. (1,3); 
\node[point, white, label=left:$\pm$] at (1,3.25) {};
 \draw[-,thick] (2.5,3.5) .. controls (2.5,3.25) and (2.5,3.25).. (2.5,3);
 \draw[-,thick] (1,3) .. controls (1,2.75) and (1,2.75).. (1,2.5);
 \draw[-,thick] (2.5,3) .. controls (2.5,2.75) and (1.5,2.75).. (1.5,2.5);
\draw[-,thick] (1.5,2.5) .. controls (1.5,2.25) and (1,2.25).. (1,2); 
 \node[point, ,white] at (1.25,2.25) {};
 \draw[-,thick] (1,2.5) .. controls (1,2.25) and (1.5,2.25).. (1.5,2);
 \draw[-,thick] (1,2) .. controls (1,1.75) and (1,1.75).. (1,1.5);
 \draw[-,thick] (1.5,2) .. controls (1.5,1.75) and (2.5,1.75).. (2.5,1.5);
 \draw[-,thick] (1,1.5) .. controls (1,1.25) and (1,1.25).. (1.5,1.25); 
 \draw[-,thick] (1.5,1.25) .. controls (1,1.25) and (1,1.25).. (1,1); 
\node[point, white, label=left:$\pm$] at (1,1.25) {};
 \draw[-,thick] (2.5,1.5) .. controls (2.5,1.25) and (2.5,1.25).. (2.5,1);
 \draw[-,thick] (1,1) .. controls (1,0.85) and (1,0.85).. (1,0.5);
 \draw[-,thick] (2.5,1) .. controls (2.5,0.85) and (1.5,0.85).. (1.5,0.5);
 \draw[->,thick] (1.5,0.5) .. controls (1.5,0.25) and (1,0.25).. (1,0.05); 
 \node[point, ,white] at (1.25,0.25) {};
 \draw[->,thick] (1,0.5) .. controls (1,0.25) and (1.5,0.25).. (1.5,0.05);
 \draw[->,thick] (3,4) .. controls (3,0.75) and (3,0.75).. (3,0.05);
\end{tikzpicture}
\end{align}
\end{center}

\medskip

\paragraph{\underline{The statement}}

\begin{theorem}\label{alternativePaBell}
As a $\mathbf{PaB}$-module in groupoids having $\mathbf{Pa}$ as $\mathbf{Pa}$-module of objects, $\mathbf{PaB}_{e\ell\ell}$ is freely 
generated by $A:=A^{1,2}$ and $B:=B^{1,2}$, together with the relations \eqref{eqn:N1bis}, \eqref{eqn:N2bis}, and \eqref{eqn:Ebis}. 
\end{theorem}

The above theorem is a direct consequence of Theorem \ref{PaBell} together with the following 
\begin{proposition}\label{greatprop}
Let us consider a $\PaB$-module in groupoids $\mathbf{PaM}$, having $\mathbf{Pa}$ as $\mathbf{Pa}$-module of objects, 
and let $A,B$ be a automorphisms of $12$. Then 
\begin{itemize}
\item[(i)] Equations \eqref{eqn:N1} and \eqref{eqn:N1bis} are equivalent; 
\item[(ii)] Equations \eqref{eqn:N2} and \eqref{eqn:N2bis} are equivalent; 
\item[(iii)] If \eqref{eqn:N1} and \eqref{eqn:N2} are satisfied, then equations \eqref{eqn:E} and \eqref{eqn:Ebis} are equivalent. 
\end{itemize}
\end{proposition}

\medskip

\paragraph{\underline{A useful observation}}
Both \eqref{eqn:N1} and \eqref{eqn:N1bis} imply 
\[
A^{1,2}\tilde R^{1,2}A^{2,1}\tilde R^{2,1}=\mathrm{Id}_{12}\,.
\]
For both, this is obtained by applying $(-)^{1,2,\emptyset}$. 
Similarly, both \eqref{eqn:N1} and \eqref{eqn:N1bis} imply 
\[
B^{1,2}\tilde R^{1,2}B^{2,1}\tilde R^{2,1}=\mathrm{Id}_{12}\,.
\]

\medskip

\paragraph{\underline{Proof of (i) and (ii) in Proposition \ref{greatprop}}}

The following calculation takes place in $\mathbf{PaM}(3)$. 
For ease of comprehension, we put a brace under a sequence of symbols where we use a relation in order to pass to the next step. 
\begin{align*} 
& \Phi^{1,2,3}A^{1,23}\underbrace{\tilde R^{1,23}\Phi^{2,3,1}}A^{2,31}\tilde R^{2,31}\Phi^{3,1,2}A^{3,12}\tilde R^{3,12} \\
= & \Phi^{1,2,3}A^{1,23}(\Phi^{1,2,3})^{-1}\tilde R^{1,2}\Phi^{2,1,3}\underbrace{\tilde R^{1,3}A^{2,31}}
\tilde R^{2,31}\Phi^{3,1,2}A^{3,12}\tilde R^{3,12} \\
= & \Phi^{1,2,3}A^{1,23}(\Phi^{1,2,3})^{-1}\tilde R^{1,2}\Phi^{2,1,3}A^{2,13}
\underbrace{\tilde R^{1,3}\tilde R^{2,31}\Phi^{3,1,2}}A^{3,12}\tilde R^{3,12} \\
= & \Phi^{1,2,3}A^{1,23}(\Phi^{1,2,3})^{-1}\tilde R^{1,2}\Phi^{2,1,3}A^{2,13}(\Phi^{2,1,3})^{-1}\tilde R^{2,1}
\underbrace{\tilde R^{12,3}A^{3,12}\tilde R^{3,12}}\\
= & \Phi^{1,2,3}A^{1,23}(\Phi^{1,2,3})^{-1}\tilde R^{1,2}\Phi^{2,1,3}A^{2,13} (\Phi^{2,1,3})^{-1}\tilde R^{2,1}(A^{12,3})^{-1}
\end{align*}
We repeatedly used (various forms of) the hexagon equation, and only at the last step we used the useful 
observation from the previous paragraph. This gives that \eqref{eqn:N1} and \eqref{eqn:N1bis} are both a 
consequence of each other. The proof that \eqref{eqn:N2} and \eqref{eqn:N2bis} are equivalent is the same. \hfill$\Box$

\medskip

\paragraph{\underline{Another useful fact}}
One can also show that \eqref{eqn:N1} and \eqref{eqn:N1bis} are equivalent to 
\begin{equation}
A^{12,3}\Phi^{1,2,3}\tilde R^{1,23}A^{23,1}\Phi^{2,3,1}\tilde R^{2,31}A^{31,2}\Phi^{3,1,2}\tilde R^{3,12}=\mathrm{Id}_{(12)3}\,.
\end{equation}

\medskip

\paragraph{\underline{Proof of (iii) in Proposition \ref{greatprop}}}

Relation \eqref{eqn:N1bis} is equivalent to
\[ \Phi^{1, 2, 3}  (A^{1, 23})^{- 1}
   (\Phi^{1, 2, 3})^{- 1} A^{12, 3} = \tilde{R}^{1,2}\Phi^{2, 1, 3}  A^{2, 13} 
   (\Phi^{2, 1, 3})^{- 1} \tilde{R}^{2,1}\,. \]
Thus, \eqref{eqn:Ebis} is equivalent to 
\begin{eqnarray*}
\Phi R^{2,3}R^{3,2} \Phi^{- 1} & = & \left(\tilde{R}^{1,2}\Phi^{2, 1, 3}  A^{2, 13} 
   (\Phi^{2, 1, 3})^{- 1} \tilde{R}^{2,1} ,(B^{12,3})^{-1} \right).
\end{eqnarray*}
Using $\tilde{R}^{2,1}B^{12,3}=B^{21, 3} \tilde{R}^{2,1}$, 
we deduce that \eqref{eqn:Ebis} is equivalent to
\[
\Phi R^{2,3}R^{3,2} \Phi^{- 1}  =  \tilde{R}^{1,2}\Phi^{2, 1, 3}  A^{2, 13} 
   (\Phi^{2, 1, 3})^{- 1} (B^{21,3})^{-1}\Phi^{2, 1, 3}  (A^{2, 13})^{-1} 
   (\Phi^{2, 1, 3})^{- 1}B^{21,3}(\tilde{R}^{1,2})^{-1}\,,
\]
which is equivalent to
\[
 (\Phi^{2, 1, 3})^{- 1}(\tilde{R}^{1,2})^{-1}\Phi R^{2,3}R^{3,2} \Phi^{- 1}\tilde{R}^{1,2}(B^{21,3})^{-1}
 \Phi^{2, 1, 3} = A^{2, 13} (\Phi^{2, 1, 3})^{- 1} B^{21,3}\Phi^{2, 1, 3}   (A^{2, 13} )^{-1}\,.
\]
Now, by using
\begin{itemize}
\item $(\Phi^{2, 1, 3})^{- 1}(\tilde{R}^{1,2})^{-1}\Phi R^{2,3}R^{3,2} 
\Phi^{- 1}\tilde{R}^{1,2}=\tilde{R}^{1,3}(\Phi^{2,3,1})^{-1}R^{2,3} \Phi^{3,2,1}R^{3,21}$
\item $ (B^{21,3})^{-1}=(R^{3,21})^{-1}B^{3,21}(R^{21,3})^{-1}$
\item $ \Phi^{2, 1, 3}  =  R^{21,3}(\Phi^{3,2,1})^{-1}
\tilde R^{3,2} \Phi^{2, 3, 1} \tilde R^{1,3} $,
\item $(\Phi^{2, 1, 3})^{- 1}=\tilde R^{1,3}(\Phi^{2, 3, 1})^{-1} \tilde R^{2,3}\Phi^{3,2,1}R^{3,21}$
\end{itemize}
we obtain
\[
\tilde R^{1,3}(\Phi^{2, 3, 1})^{-1}R^{2,3}\Phi^{3,2,1}B^{3,21}(\Phi^{3,2,1})^{-1}
\tilde R^{3,2} \Phi^{2, 3, 1}\tilde R^{1,3}=A^{2,13}\tilde R^{1,3}(\Phi^{2, 3, 1})^{-1}
\]
\[
\tilde R^{2,3}\Phi^{3,2,1}
B^{3,21}(\Phi^{3,2,1})^{-1}
\tilde R^{3,2} \Phi^{2, 3, 1} \tilde R^{1,3} (A^{2,13})^{-1}\,.
\]
After performing $A^{2,13}\tilde R^{1,3}=\tilde R^{1,3}A^{2,31}$ in the r.h.s. of the above equation, 
one can cancel out the $\tilde R^{1,3}$ terms in both sides of the equation.
We obtain, by performing the permutation $(123) \mapsto (312)$ that
\[
(\Phi^{1,2,3})^{-1}R^{1,2}\Phi^{2,1,3}B^{2,13}(\Phi^{2,1,3})^{-1}
\tilde R^{2,1} \Phi^{1,2,3} 
\]
\[
=A^{1,23}(\Phi^{1,2,3})^{-1}\tilde R^{1,2}\Phi^{2,1,3}
B^{2,13}(\Phi^{2,1,3})^{-1}
\tilde{R}^{2,1} \Phi^{1,2,3}(A^{1,23})^{-1}\,.
\]
This is equivalent to 
\[
\Phi^{1,2,3} A^{1,23}(\Phi^{1,2,3})^{-1}\tilde R^{1,2}\Phi^{2,1,3}
B^{2,13}(\Phi^{2,1,3})^{-1}
\tilde{R}^{2,1} \Phi^{1,2,3}(A^{1,23})^{-1}(\Phi^{1,2,3})^{-1}
\]
\[
=R^{1,2}\Phi^{2,1,3}B^{2,13}(\Phi^{2,1,3})^{-1}\tilde R^{2,1}\,,
\]
which is equivalent to 
\[
\Phi^{1,2,3} A^{1,23}(\Phi^{1,2,3})^{-1}\tilde R^{1,2}\Phi^{2,1,3}
B^{2,13}(\Phi^{2,1,3})^{-1}
\tilde{R}^{2,1} \Phi^{1,2,3}(A^{1,23})^{-1}(\Phi^{1,2,3})^{-1}
\]
\[
(\tilde R^{2,1})^{-1}\Phi^{2,1,3}(B^{2,13})^{-1}(\Phi^{2,1,3})^{-1} (R^{1,2})^{-1}
=\on{Id}_{(12)3}\,.
\]
As $(R^{1,2})^{-1}R^{1,2}R^{2,1}=R^{2,1}=(\tilde R^{1,2})^{-1}$, we obtain 
\[
R^{1,2} R^{2,1}=\left(\Phi^{1,2,3}A^{1,23}(\Phi^{1,2,3})^{-1}, \tilde R^{1,2}\Phi^{2,1,3}B^{2,13}(\Phi^{2,1,3})^{-1}\tilde R^{2,1}\right)\,.
\]

\section{An alternative presentation for $\PaB_{e\ell\ell}^\Gamma$}
Below, we borrow the notation from Theorem \ref{PaB:ell:G}. 
\begin{theorem}\label{PaB:ell:G:bis}
As a $\mathbf{PaB}$-module in groupoids with a diagonally trivial $\Gamma$-action and having 
$\mathbf{Pa}^{\Gamma}$ as $\mathbf{Pa}$-module of objects, $\mathbf{PaB}_{e\ell\ell}^{\Gamma}$ 
is freely generated by $A$ and $B$ together with the relations 
\begin{flalign}
& \underline{A}^{12,3} = \Phi^{1,2,3}\underline{A}^{1,23}(\Phi^{1,2,3})^{-1}\tilde R^{1,2}\Phi^{2,1,3}
\underline{A}^{2,13} (\Phi^{2,1,3})^{-1}\tilde R^{2,1}\,, 
\label{def:PaB:ell:G:1bis}\tag{tN1bis} \\
& \underline{B}^{12,3} = \Phi^{1,2,3}\underline{B}^{1,23}(\Phi^{1,2,3})^{-1}\tilde R^{1,2}\Phi^{2,1,3}
\underline{B}^{2,13} (\Phi^{2,1,3})^{-1}\tilde R^{2,1} \,,
\label{def:PaB:ell:G:2bis}\tag{tN2bis} \\ 
& \Phi^{1,2,3}R^{2,3}R^{3,2}(\Phi^{1,2,3})^{-1} =
\big(\underline{A}^{12,3}\Phi^{1,2,3}(\underline{A}^{1,23})^{-1}(\Phi^{1,2,3})^{-1}\,,\,(\underline{B}^{12,3})^{-1}\big)\,.
\label{def:PaB:ell:G:Ebis}\tag{tEbis}
\end{flalign}
\end{theorem}
In order to prove Theorem \ref{PaB:ell:G:bis}, one can
\begin{itemize}
\item[(i)] Either deduce it from Theorem \ref{PaB:ell:G} in a similar manner as we deduced 
Theorem \ref{alternativePaBell} from Theorem \ref{PaBell}; 
\item[(ii)] Or deduce it from Theorem \ref{alternativePaBell} in a similar manner as we deduced 
Theorem \ref{PaB:ell:G} from Theorem \ref{PaBell}. 
\end{itemize}
Both strategies are straightforward to implement; this is left to the reader.  

\backmatter

\bibliographystyle{amsalpha}



\chapter*{Glossary}
\printnoidxglossaries

\end{document}